\documentclass{amsart}
\usepackage{math,mathabx,amsrefs,color,wrapfig,mathrsfs,enumitem,tikz,tikz-3dplot}
\usetikzlibrary{calc,patterns,arrows.meta,automata,decorations.markings}
\usepackage[outline]{contour}
\contourlength{1pt}
\usepackage{hyperref}
\newtheorem*{thm*}{Theorem}

\newcommand\wrwr{\operatorname{\wr\wr}}
\newcommand\wrf{\wr^{\text{\upshape f.v.}}}
\renewcommand\sym{\operatorname{Sym}}
\def\pair<#1>{{\langle\!\langle}#1{\rangle\!\rangle}}
\newcommand\CG{\widehat\Gamma}
\newcommand\SL{{\operatorname{SL}}}

\newenvironment{fsa}[1][auto]{
    \begin{tikzpicture}[#1,->,>=Stealth,shorten >=1pt,auto,
      node distance=2.8cm,semithick,double equal sign distance]
      \tikzstyle{every state}=[fill=gray,draw=none,text=white]}{%
    \end{tikzpicture}}
\def\9{{\mathbf 0}}
\def\8{{\mathbf 1}}

\begin{document}
\title{Growth of groups and wreath products}

\author{Laurent Bartholdi}
\address{\'Ecole Normale Sup\'erieure, Paris \emph{and} Mathematisches Institut, Georg-August Universit\"at zu G\"ottingen}
\thanks{Partially supported by ANR grant ANR-14-ACHN-0018-01 and DFG grant BA4197/6-1}
\dedicatory{To Wolfgang Woess, in fond remembrance of many a visit to Graz}
\date{22 December 2015}
\maketitle

\setcounter{tocdepth}{1}
\tableofcontents

\section*{Introduction}
These notes are an expanded version of a mini-course given at ``Le
Louverain'', June 24-27, 2014. Its main objective was to gather
together useful facts about wreath products, and especially their
geometry, in its application to problems and questions about growth of
groups. The wreath product is a fundamental construction in group
theory, and I hope to help make the reader more familiar with it.

It has proven very useful, in the recent years, in better
understanding asymptotics of the the word growth function on groups,
namely the function assigning to $R\in\N$ the number of group elements
that may be obtained by multiplying at most $R$ generators. The
papers~\cites{bartholdi-erschler:permutational,bartholdi-erschler:givengrowth,bartholdi-erschler:orderongroups,bartholdi-erschler:imbeddings,bartholdi-erschler:distortion}
may be hard to read, and contain many repetitions as well as
references to outer literature; so that, by providing a unified
treatment of these articles, I may provide the reader with easier
access to the results and methods.

I have also attempted to define all notions in their most natural
generality, while restricting the statements to the most important or
fundamental cases. In this manner, I would like the underlying ideas
to appear more clearly, with fewer details that obscure the line of
sight. I avoided as much as possible reference to literature, taking
the occasion of reproving some important results along the way.

I have also allowed myself, exceptionally, to cheat. I do so only
under three conditions: (1) I clearly mark where there is a cheat; (2)
the complete result appears elsewhere for the curious reader; (3) the
correct version would be long and uninformative.

I have attempted to make the text suitable for a short course. In
doing so, I have included a few exercises, some of which are hopefully
stimulating, and a section on open problems. What follows is a brief
tour of the highlights of the text.

\subsection*{Wreath products}
The wreath product construction, described in~\S\ref{ss:wreath}, is an
essential operation, building a new group $W$ out of a group $H$ and a
group $G$ acting on a set $X$. Assuming\footnote{There is no need to
  require the action of $G$ on $X$ to be faithful; this is merely a
  visual aid. See~\S\ref{ss:wreath} for the complete definition.} that
$G$ is a group of permutations of $X$, the wreath product is the group
$W=H\wrwr_X G$ of $H$-decorated permutations in $G$: if elements of
$G$ are written in the arrow notation, with elements of $X$ lined in
two identical rows above each other and an arrow from each $x\in X$ to
its image, then an element of $H\wrwr_X G$ is an arrow diagram with an
element of $H$ attached to each arrow, e.g.
\[
\begin{tikzpicture}
  \foreach \i/\j/\pos in {1/2/0.3,2/1/0.3,3/4/0.3,4/5/0.7,5/3/0.5}
  \draw[->] (1.7*\i,1) -- node[pos=\pos] {\contour{white}{$h_\i$}} (1.7*\j,0);
\end{tikzpicture}
\]

One of the early uses of wreath products is as a classifier for
extensions, as discovered by Kaloujnine, see Theorem~\ref{thm:kk}:
there is a bijective correspondence between group extensions with
kernel $H$ and quotient $G$ on the one hand, and appropriate subgroups
of the wreath product $H\wrwr G$, with $G$ seen as a permutation group
acting on itself by multiplication.  We extend this result to
permutational wreath products:
\begin{thm*}[Theorem~\ref{thm:kkwreath}]
  Let $G,H$ be groups, and let $G$ act on the right on a set
  $X$. Denote by $\pi\colon H\wrwr G\to G$ the natural
  projection. Then the map $E\mapsto E$ defines a bijection between
  \[\frac{\left\{\begin{array}{l@{\;}r}E:&E\le H\text{ and the $G$-sets $X$ and $H\backslash E$ are isomorphic}\\&\text{via a homomorphism }E\to G\end{array}\right\}}{\text{isomorphism $E\to E'$ of groups intertwining the actions on $H\backslash E$ and $H\backslash E'$}}\]
    and
    \[\frac{\left\{\begin{array}{l@{\;}l}E\le H\wrwr G:&\pi(E)\text{ transitive on }X\text{ and }\\&\ker(\pi)\cap E\overset\cong\longrightarrow H\text{ via }f\mapsto f(x)\text{ for all }x\in X\end{array}\right\}}{\text{conjugacy of subgroups of }H\wrwr G}.
  \]
\end{thm*}

The wreath product $H\wrwr_X G$ is uncountable, if $H\neq1$ and $X$ is
infinite. It contains some important subgroups: $H\wr_X G$, defined as
those decorated permutations in which all but finitely many labels are
trivial; and $H\wrf_X G$, defined as those decorated permutations in
which the labels take finitely many different values. Clearly $H\wr_X
G\subseteq H\wrf_X G\subseteq H\wrwr_X G$, and $H\wr_X G$ is countable
as soon as $H,G,X$ are countable.

\subsection*{Growth of groups}
Let us summarise here the main notions; for more details,
see~\S\ref{ss:growth}.  A choice of generating set $S$ for a group $G$
gives rise to a graph, the \emph{Cayley graph}: its vertex set is $G$,
and there is an edge from $g$ to $gs$ for each $g\in G$, $s\in S$. The
path metric on this graph defines a metric $d$ on $G$ called the
\emph{word metric}. The Cayley graph is invariant under left
translation, and so is the word metric.

One of the most naive invariants of this graph is its \emph{growth},
namely the function $v_{G,S}(R)$ measuring the cardinality of a ball
of radius $R$ in the Cayley graph. If the graph exhibits some kind of
regularity, then it should translate into some regularity of the
function $v_{G,S}$.

For example, Klarner~\cites{klarner:growth1,klarner:growth2} studied
the growth of crystals (that expand according to a precise and simple
rule) via what turns out to be the growth of an abelian group.

A convenient tool to study various forms of regularity of a function
$v_{G,S}$ is the associated generating function
$\Gamma_{G,S}(z)=\sum_{R\in\N}(v(R)-v(R-1))z^R$. The regularity of
$v_{G,S}$ translates then into a property of $\Gamma_{G,S}$ such as
being a rational, algebraic, $D$-finite, \dots\ function of $z$.

We may rewrite $\Gamma_{G,S}(z)=\sum_{g\in G}z^{d(1,g)}$; then a
richer power series keeps track of more regularity of $G$:
\[\CG_{G,S}(z)=\sum_{g\in G}g z^{d(1,g)}.\]
This is a power series with co\"efficients in the group ring $\Z G$,
and again we may ask whether $\CG_{G,S}$ is rational or
algebraic\footnote{The \emph{rational subring} of $\Z G[[t]]$ is the
  smallest subring of $\Z G[[t]]$ containing $\Z G[t]$ and closed
  under Kleene's star operation $A^*=1+A+A^2+\cdots$, for all $A(z)$
  with $A(0)=0$.

  The \emph{algebraic subring} of $\Z G[[t]]$ is the set of power
  series that may be expressed as the solution $A_1$ of a non-trivial
  system of non-commutative polynomial equations
  $\{P_1(A_1,\dots,A_n)=0,\cdots,P_n(A_1,\dots,A_n)=0\}$ with
  co\"efficients in $\Z G[t]$. The solution is actually rational if
  furthermore the $P_i$ are of the form
  $c_{i,0}+\sum_{j=1}^n c_{i,j}A_j$ with $c_{i,j}\in\Z G[t]$.}.

If $G$ has an abelian subgroup of finite index~\cite{liardet:phd}, or
if $G$ is word-hyperbolic~\cite{grigorchuk-n:complete}, then
$\CG_{G,S}$ is a rational function of $z$ for all choices of $S$. We
give a sufficient condition for $\CG_{G,S}$ to be algebraic:
\begin{thm*}[Theorem~\ref{thm:wreathalg}]
  Let $H=\langle T\rangle$ be a group such that $\CG_{H,T}$ is
  algebraic, and let $F$ be a free group. Consider $G=H\wr F$,
  generated by $S=T\cup\{\text{a basis of }F\}$. Then $\CG_{G,S}$ is
  algebraic.
\end{thm*}

We then turn to studying the asymptotics of the growth function
$v_{G,S}$. Let us write $v\precsim w$ to mean that $v(R)\le w(CR)$ for
some constant $C\in\R_+$ and all $R\ge0$, and $v\sim w$ to mean
$v\precsim w\precsim v$. Then the $\sim$-equivalence class of
$v_{G,S}$ is independent of the choice of $S$, so we may simply talk
about $v_G$.

For ``most'' examples of groups, either $v_G(R)$ is bounded by a
polynomial in $R$ or $v_G(R)$ is exponential in $R$. This is, in
particular, the case for soluble, linear and word-hyperbolic
groups. There exist, however, examples of groups for which $v_G(R)$
admits an intermediate behaviour between polynomial and exponential;
they are called groups of \emph{intermediate growth}. The question of
their existence was raised by Milnor~\cite{milnor:5603}, was answered
positively by Grigorchuk~\cite{grigorchuk:growth}, and has motivated
much group theory in the second half of the 20th century.

Let $\eta_+\approx 2.46$ be the positive root of $T^3-T^2-2T-4$, and
set $\alpha=\log2/\log\eta_+\approx0.76$. We shall show that, for
\emph{every} sufficiently regular function $f\colon\R_+\to\R_+$ with
$\exp(R^\alpha)\precsim f\precsim\exp(R)$, there exists a group with
growth function equivalent to $f$:
\begin{thm*}[Theorem~\ref{thm:givengrowth}]
  Let $f\colon\R_+\to\R_+$ be a function satisfying
  \[f(2R)\le f(R)^2 \le f(\eta_+R)\text{ for all $R$ large enough}.\]
  Then there exists a group $G$ such that $v_G\sim f$.
\end{thm*}

Thus, groups of intermediate growth abound, and the space of
asymptotic growth functions of groups is as rich as the space of
functions. Furthermore, we shall show that there is essentially no
restriction on the subgroup structure of groups of intermediate
growth. Let us call a group $H$ \emph{locally of subexponential
  growth} if every finitely generated subgroup of $H$ has growth
function $\precnsim\exp(R)$. Clearly, if $H$ is a subgroup of a group
of intermediate growth then it has locally subexponential growth. We
show, conversely:
\begin{thm*}[Theorem~\ref{thm:imbed}]
  Let $B$ be a countable group locally of subexponential growth. Then
  there exists a finitely generated group of subexponential growth in
  which $B$ imbeds as a subgroup.
\end{thm*}

Finally, it may happen that a group $G$ has exponential growth, namely
that the growth rate $\lim_{R\to\infty}v_{G,S}(R)^{1/R}$ is $>1$ for
all $S$, but that the infimum of these growth rates, over all $S$, is
$1$. Such a group is called of \emph{non-uniform exponential
  growth}. The question of their existence was raised by
Gromov~\cite{gromov:metriques}*{Remarque~5.12}. Again, soluble, linear
and word-hyperbolic groups cannot have non-uniform exponential growth;
but, again, it turns out that such groups abound. We shall show:
\begin{thm*}[Theorem~\ref{thm:nueg}]
  Every countable group may be imbedded in a group of non-uniform
  exponential growth.

  Furthermore, the group $W$ in which the countable group imbeds may
  be required to have the following property: there is a constant $K$
  such that, for all $R>0$, there exists a generating set $S$ of $W$
  with
  \[v_{W,S}(r) \le \exp(K r^\alpha)\text{ for all }r\in[0,R].
  \]
\end{thm*}

\subsection*{(Self-)similar groups and branched groups}
All the constructions mentioned in the previous subsection take place
in the universe of \emph{(self-)similar groups}. Here is a brief
description of these groups; see~\S\ref{ss:ss} for details.

Just as a self-similar set, in geometry, is a set describable in terms
of smaller copies of itself, a \emph{self-similar group} is a group
describable in terms of ``smaller'' copies of itself. A
\emph{self-similar structure} on a group $G$ is a homomorphism
$\phi\colon G\to G\wr_X P$ for a permutation group $P$ of $X$. Thus
elements of $G$ may be recursively written in terms of $G$-decorated
permutations of $X$. For this description to be useful, of course, the
homomorphism $\phi$ must satisfy some non-degeneracy condition (in
particular be injective), and $P$ should be manageable, say finite.

The fact that the copies of $G$ in $G\wr_X P$ are ``smaller'' than the
original is expressed as follows: there is a norm on $G$ such that,
for $g\in G$ and $\phi(g)$ a permutation with labels $(g_x:x\in X)$,
the elements $g_x$ are shorter than $g$, at least as soon as $g$ is
long enough. For example, consider $G$ finitely generated, and denote
by $\|\cdot\|$ the word norm on $G$. One requires $\|g_x\|<\|g\|$ for
all $x\in X$ and all $\|g\|\gg1$; this is equivalent to the existence
of $\lambda\in(0,1)$ and $K\ge0$ such that $\|g_x\|\le\lambda\|g\|+K$
for all $x\in X,g\in G$.

Furthermore, in cases that interest us, the map $\phi$ is almost an
isomorphism, in that its image $\phi(G)$ has finite index in $G\wr_X
P$. Thus $\phi$ may be thought of as a \emph{virtual isomorphism}
between $G$ and $G^X$, namely an isomorphism between finite-index
subgroups. When one endows $G^X$ with the $\ell^\infty$ metric
$\|(g_x)\|=\max_{x\in X}\|g_x\|$, the condition above requires that
this virtual isomorphism be a contraction. On the other hand, endowing
$G^X$ with the $\ell^1$ metric $\|(g_x)\|=\sum_{x\in X}\|g_x\|$, the
optimal Lipschitz constant of the virtual isomorphism plays a
fundamental r\^ole in estimating the growth of $G$.

\emph{Similar} groups are a natural generalization: one is given a set
$\Omega$ and a self-map $\sigma\colon\Omega\righttoleftarrow$; for
each $\omega\in\Omega$, a group $G_\omega$ and a permutation group
$P_\omega$ of a set $X_\omega$; and homomorphisms
$\phi_\omega\colon G_\omega\to
G_{\sigma\omega}\wr_{X_\omega}P_\omega$.
Taking for $\Omega$ a singleton recovers the notion of self-similar
group. Taking $\Omega=\N$ and $\sigma(n)=n+1$ defines in full
generality a similar group $G_0$; but it is often more convenient to
consider a larger family of groups in which $(G_n)_{n\in\N}$
imbeds. In particular, one obtains a \emph{topological space} of
groups, in such a manner that close groups have close properties (for
example, their Cayley graphs coincide on a large ball).

\subsection*{Acknowledgments}
I am very grateful to Yago Antolin, Laura Ciobanu and Alexey
Talambutsa for having organised the workshop in Le Louverain where I
presented a preliminary version of this text, and to the participants
of the workshop for their perspicacious questions.

A large part of the material is taken from articles written in
collaboration with Anna Erschler, and I am greatly indebted to her for
generously embarking me on her projects. It owes much to her energy
and enthusiasm that we were able to finish our joint articles.

Hao Chen, Yves de Cornulier and Pierre de la Harpe helped improve
these notes by pointing out a number of mistakes and inconsistencies.

\subsection*{Open problems}
This text presents a snapshot of what is known on growth of groups in
2014; there remain a large number of open problems. Here are some
promising directions for further research.

\setcounter{section}{0}
\renewcommand\thesection{@}

\begin{enumerate}[leftmargin=*]
\item Which groups $G$ are such that, for all generating sets $S$, the
  complete growth series $\CG_{G,S}$ is a rational function of $z$?

  This is known to hold for virtually abelian groups, and for
  word-hyperbolic groups. Conjecturally, this holds for no other
  group.

  The related question of which groups have a rational (classical)
  growth function is probably more complicated, see~\S\ref{ss:formal}.

\item Is the analytic continuation $1/\CG_{G,S}(1)$ related to the
  complete Euler characteristic of $G$, just as $1/\Gamma_{G,S}(1)$ is
  (under some additional conditions) the Euler characteristic of $G$?
  See~\cite{stallings:centerless}*{\S1.8} for complete Euler
  characteristic.

\item Does there exist a group $G$ with two generating sets $S_1,S_2$
  such that $\CG_{G,S_1}$ is rational but $\CG_{G,S_2}$ is
  transcendental?

  Such an example could be $G=F_2\times F_2$. Set $S=\{x,y\}^{\pm1}$ a
  free generating set of $F_2$, and
  $S_1=S\times\{1\}\sqcup\{1\}\times S$ and
  $S_2=S_1\sqcup\{(s,s):s\in S\}$.

  The same properties probably hold for the usual generating series
  $\Gamma_{G,S_1}$ and $\Gamma_{G,S_2}$. The radius of convergence of
  $\Gamma_{G,S_1}$ is $1/3$, but that of $\Gamma_{G,S_2}$ is
  unknown.

  This problem is strongly related to the ``Matching subsequence
  problem'', which asks for the longest length of a common subsequence
  among two independently and uniformly chosen words of length $n$
  over a $k$-letter alphabet;
  see~\cite{chvatal-sankoff:commonsubsequence}.  It is easy to see
  that, for two uniformly random reduced words of length $n$ in $F_2$,
  the longest common subword has length $\approx\gamma n$ for some
  constant $\gamma$, as $n\to\infty$. Thus a pair $(g,h)\in G$ with
  $\|g\|=\|h\|=n$ has length $2n$ with respect to $S_1$, but
  approximately $(2-\alpha)n$ with respect to $S_2$. We might call
  $\gamma$ the \emph{Chv\'atal-Sankoff constant} of $F_2$.

\item Do there exist infinite simple groups of subexponential growth?

  There is no reason for such groups \emph{not} to exist; but the
  construction methods described in this text yield groups acting on
  rooted trees, which therefore are as far as possible from being
  simple.

  There is also no reason for finitely presented groups of
  subexponential growth \emph{not} to exist; again, the obstacle is
  probably more our mathematical limitations than fundamental
  mathematical reasons.

  The following question, by de la Harpe~\cite{harpe:uniform}, is
  still open at the time of writing: ``Do there exist groups with
  Kazhdan's property (T) and non-uniform exponential growth?''

  Similarly, it is not known whether there exist simple finitely
  generated groups of non-uniform exponential growth, and whether
  there exist finitely presented groups of non-uniform exponential
  growth.


\item Do there exist groups whose growth function lies strictly
  between polynomials and $\exp(R^{1/2})$?

  See the discussion in~\S\ref{ss:asymptotic}. There exists a
  superpolynomial function $f(R)\succsim R^{(\log R)^{1/100}}$ such
  that no group has growth strictly between polynomials and
  $f(R)$. There exists no residually nilpotent group whose growth is
  strictly between polynomials and $\exp(R^{1/2})$, see
  Theorem~\ref{thm:growth:hp}.

\item What is the asymptotic growth of the first Grigorchuk group?
  What is its exact growth, for the generating set $\{a,b,c,d\}$?
  Does the growth series of the Grigorchuk group exhibit some kind of
  regularity?

  Some experiments indicate that this must be the case. For example,
  consider the quotient $G_n$ of the first Grigorchuk group that acts
  on $\{\9,\8\}^n$. It is a finite group of cardinality
  $2^{5\cdot2^{n-3}+2}$. For $n\le 7$, the diameter $D_n$ of its
  Cayley graph (for the natural generating set $\{a,b,c,d\}$) is the
  sequence $1,4,8,24,56,136,344$ and satisfies the recurrence
  $D_n=D_{n-1}+2D_{n-2}+4D_{n-3}$. If this pattern went on, the growth
  of the first Grigorchuk group would be asymptotically
  $\exp(R^{\log2/\log\eta_+})\approx\exp(R^{0.76})$.

  If a group has subexponential growth, then its growth series is
  either rational or transcendental, and if the group has intermediate
  growth, then the growth series must be transcendental;
  see~\ref{thm:polya-carlson} in~\S\ref{ss:formal}. Thus Grigorchuk
  group's growth series is transcendental. Does the series satisfy a
  functional equation? That would make it akin to the classical
  partition function $\sum_{n\ge0}p(n)z^n=\prod_{n\ge1}(1-z^n)^{-1}$,
  which (up to scaling and multiplying by $z^{1/24}$) is a
  \emph{modular function}\footnote{i.e. a function
    $A(z)=A(\exp(2\pi i\tau))$ such that the corresponding function
    $\tau\mapsto A(\exp(2\pi i\tau))$ on the upper half plane is
    invariant under a finite-index subgroup of $\SL_2(\Z)$.}. Ghys
  asked me once: ``Is the growth series $\Gamma(z)$ of Grigorchuk's
  group modular?''

\item For every $k\in\N$, the \emph{space of marked $k$-generated
    groups} $\mathscr S_k$ may be defined as the space of normal
  subgroups of the free group $F_k$, by identifying $G=\langle
  s_1,\dots,s_k\rangle$ with the kernel of the natural map $F_k\to G$
  sending generator to generator. It is a compact space. What
  properties does the set $\mathscr I$ of groups of intermediate
  growth, and the set $\mathscr N$ of groups of non-uniform
  exponential growth, enjoy in this space?  For example,

  ``Is there an uncountable open subset of $\mathscr S_k$ in which
  $\mathscr N$, or $\mathscr I$, is dense? Is $\mathscr N$ dense in
  the complement of groups of polynomial growth?''

  Recall that similar groups are families of groups
  $(G_\omega)_{\omega\in\Omega}$ indexed by a space $\Omega$. If all
  $G_\omega$ are $k$-generated, we obtain a map
  $\Omega\to\mathscr S_k$, which under favourable circumstances is
  continuous. This has been exploited
  e.g. in~\cite{nekrashevych:cantor} to produce groups of non-uniform
  exponential growth.

  It had actually been doubted, before Grigorchuk's
  discovery~\cite{grigorchuk:growth}, whether there exist groups of
  intermediate growth. This text tries to convince the reader that
  they are abundant. Giving a precise meaning to the above question
  would quantify, in some manner, the extent to which they are
  abundant.
\end{enumerate}

\subsection*{Notational conventions}
I try to adhere to standard group-theoretical notation. In particular,
the right action of a group element $g$ on a point $x$ is written
$xg$, and a left action would be written ${}^g x$. The stabilizer of
$x$ is written $G_x$. The conjugation action of a group on itself is
written $g^h=h^{-1}g h$, and the commutator of two elements is
$[g,h]=g^{-1}h^{-1}g h=g^{-1}g^h=h^{-g}h$.

I also introduce a minimal amount of new, ``fancy'' notation to
represent elements of wreath products or of self-similar groups, and
hope that it helps in achieving clarity and conciseness.

\renewcommand\thesection{\Alph{section}}
\section{Wreath products}\label{ss:wreath}
We start by the basic construction. Let $H$ be a group, and let $G$ be
a group acting on the right on a set $X$. We construct two groups
\begin{align*}
  H\wr_X G &= \big({\prod_X}'H\big)\rtimes G\qquad\text{the \emph{restricted} wreath product},\\
  H\wrwr_X G &= \big(\prod_X H\big)\rtimes G\qquad\text{the \emph{unrestricted} wreath product}.
\end{align*}

Here the unrestricted product $\prod_X H$ may be viewed as the group
of functions $X\to H$, with pointwise composition and with left
$G$-action given by pre-composition\footnote{Note that the side of the
  action changes! It is best to always use the appropriate side, so as
  to avoid inverses. Recall however that every left action can be
  converted into a right action by setting $f^g:={{}^{g^{-1}}\!f}$ and
  vice versa.}: ${}^g\!f$ is the function given by
$({}^g\!f)(x)=f(x g)$. The \emph{restricted} product $\prod_X' H$ is
then identified with finitely supported functions $X\to H$. In both
cases, this product is a subgroup of the wreath product, and is called
its \emph{base group}.

In the particular case of $X=G$ with natural right action by
multiplication, one calls $H\wrwr_G G$ the \emph{regular unrestricted
  wreath product}, and writes it simply $H\wrwr G$; and similarly for
the \emph{regular restricted wreath product} $H\wr_G G=H\wr G$.

Assume that the action of $G$ on $X$ is faithful; so that elements of
$G$ may be identified with permutations of $X$. The best way to
describe elements of $H\wrwr_X G$ or its subgroup $H\wr_X G$ is by
\emph{decorated permutations}: one writes a permutation of $X$,
decorated by elements of $H$, such as

\begin{equation}\label{eq:permutation}
  \begin{tikzpicture}[baseline=5mm]
    \foreach \i/\j/\pos in {1/2/0.3,2/1/0.3,3/4/0.3,4/5/0.7,5/3/0.5}
    \draw[->] (1.7*\i,1) -- node[pos=\pos] {\contour{white}{$h_\i$}} (1.7*\j,0);
  \end{tikzpicture}
\end{equation}
Permutations are multiplied as usual: by stacking them, and pulling
the arrows tight. Likewise, decorated permutations are multiplied by
stacking them and multiplying the labels along the composed arrows. We
do not write the labels when they are the identity. Here is a
graphical computation of a product:
\begin{equation*}
  \begin{tikzpicture}[baseline=0mm]
    \foreach \i/\j/\pos in {1/2/0.3,2/1/0.3,3/4/0.3,4/5/0.7,5/3/0.5}
    \draw[->] (1.2*\i,0.8) -- node[pos=\pos] {\contour{white}{$h_\i$}} (1.2*\j,0.06);
    \foreach \i/\j/\pos in {1/2/0.3,2/3/0.7,3/1/0.5,4/5/0.3,5/4/0.3}
    \draw[->] (1.2*\i,-0.06) -- node[pos=\pos] {\contour{white}{$k_\i$}} (1.2*\j,-0.8);
  \end{tikzpicture}\quad=\quad
  \begin{tikzpicture}[baseline=7mm]
    \foreach \i/\j/\k/\pos in {1/2/3/0.2,2/1/2/0.3,3/4/5/0.2,4/5/4/0.7,5/3/1/0.5}
    \draw[->] (1.4*\i,1.4) -- node[pos=\pos] {\contour{white}{$h_\i k_\j$}} (1.4*\k,0);
  \end{tikzpicture}.
\end{equation*}

When writing formul\ae, we must sometimes depart from the graphical
notation, in which permutations are written top-to-bottom or
left-to-right and thanks to their arrows there is no ambiguity in
knowing in which order to compose the labels. We invariably let
permutations act on the right on sets, and thus `$\sigma\tau$' means
`first $\sigma$, then $\tau$'.

\begin{exse}
  In the wreath product $W=\{\pm1\}\wr\sym(2)$, consider the element
  \[h=\begin{tikzpicture}
    \draw[->] (0,1) -- (2,0);
    \draw[->] (2,1) -- node [pos=0.3] {\contour{white}{$-1$}} (0,0);
  \end{tikzpicture}
  \]
  Show by concatenating the diagram with itself that $h$ has order
  exactly $4$. The group $W$ has order $8$; which of the order-$8$
  groups is it?
\end{exse}

Let us consider a wreath product $W=H\wrwr_X G$. In writing elements
$w,w'\in W$ algebraically, we may express them in the form $w=f g$ with
$f\colon X\to H$ and $g\in G$, or in the form $w'=g f$. In both cases,
the element $g$ and the function $f$ are unique, by definition of the
semidirect product. The compositions $(f g)(f'g')$ and $(g f)(g'f')$
are, in all cases, computed using the relation
\[g\cdot f = {}^g\!f\cdot g,
\]
namely
\begin{equation}\label{eq:wreathmult}
  (f g)(f'g')=(f\cdot{{}^g\!f'})(g g')\text{ and }(g f)(g' f')=(g
g')({}^{(g')^{-1}}\!f \cdot f').
\end{equation}

\begin{exse}
  Let $R$ be a ring, viewed as a group under addition, and let $R G$
  denote the group ring of $G$, on which $G$ acts by right
  multiplication. Show that $R\wr G$ is isomorphic to $R G\rtimes
  G$.
  More generally, let $X$ be a $G$-set; then $RX$ is a
  $G$-module. Show that $R\wr_X G$ and $RX\rtimes G$ are isomorphic.
\end{exse}

\subsection{Actions}
Assume now moreover that $H$ acts from the right on a set $Y$. Then
there are two natural sets on which $W=H\wrwr_X G$ acts:
\begin{itemize}
\item There is an action on $Y\times X$, given by $(y,x)\cdot
  f g=(y f(x),x g)$ for $(y,x)\in Y\times X$; it is called the
  \emph{imprimitive} action;
\item There is an action on $Y^X$, the set of functions $X\to Y$,
  given by $(\phi\cdot f g)(x g) = \phi(x)f(x)$ for $\phi\colon X\to Y$;
  it is called the \emph{primitive} action.
\end{itemize}

\begin{exse}
  There are natural bijections between the sets
  $(Z\times Y)\times X=Z\times (Y\times X)$ and
  $(Z^Y)^X=Z^{Y\times X}$. Assume now that a group $G$ acts on $X$, a
  group $H$ acts on $Y$ and a group $I$ acts on $Z$. Show that the
  bijections above give isomorphisms between the groups
  $(I\wr_Y H)\wr_X G$ and $(I\wr_{Y\times X}(H\wr_X G))$ as
  permutation groups, both of $Z\times Y\times X$ and of $(Z^Y)^X$.
\end{exse}

\subsection{History}
Leo Kaloujnine understood the importance of wreath products in the
early 1940's. It is said that he worked, during the second World War,
in a uniform factory and observed a rivet machine; it was made of a
rotating ring containing many rotating disks in it. Identifying the
movement of the ring with an action of $G$ and each subdisk with an
action of $H$, one sees that motions of the machine are described by
wreath product elements. In his dissertation (under \'Elie Cartan,
\cite{kaloujnine:struct}), he studied the Sylow subgroups of symmetric
groups, and showed that they were iterated wreath products.

However, this description of maximal $p$-subgroups of symmetric groups
already appears in the classical 1870 treatise by Camille
Jordan~\cite{jordan:subs}*{II.I.41}, who implicitly defined wreath
products there. This is all the more remarkable since Sylow's theorems
were only published two years later~\cite{sylow:subs}!

Kaloujnine returned to Soviet Union after the war, and contributed
greatly to the development of mathematics in Ukraine, founding in 1959
the department of algebra and mathematical logic. He is remembered for
the following important result classifying group extensions; namely,
that the wreath product is a universal object containing all
extensions:
\begin{thm}[Kaloujnine-Krasner,~\cite{kaloujnine-krasner:extensions}]\label{thm:kk}
  Let $G,H$ be groups. Denote by $\pi\colon H\wrwr G\to G$ the natural
  projection. Then the map $E\mapsto E$ defines a bijection between
  \[\frac{\big\{E : 1\to H\to E\to G\to 1\big\}}{\text{isomorphism of extensions}}\]
  and
  \[\frac{\big\{E\le H\wrwr G:\pi(E)=G\text{ and }\ker(\pi)\cap E\overset\cong\longrightarrow H\text{ via }c\in H^G\mapsto c(1)\big\}}{\text{conjugacy of subgroups of }H\wrwr G}.
  \]
\end{thm}
\begin{exse}
  Prove Theorem~\ref{thm:kk}.

  Hint: given an extension $1\to H\to E\overset\tau\to G\to 1$, choose
  a set-theoretic section\footnote{namely, a map satisfying
    $\tau(\widetilde g)=g$ for all $g\in G$} $g\mapsto\widetilde g$ of
  $\tau$, and define an imbedding $E\to H\wrwr G$ by
  \[e\mapsto\left(g\mapsto \widetilde g e(\widetilde{g\tau(e)})^{-1}\right)\tau(e).\]
\end{exse}

Theorem~\ref{thm:kk} may be interpreted more abstractly as saying
that, if $1\to H\to E\to G\to 1$ is an exact sequence and $f_0\colon
H\to K$ is a group homomorphism, then $f_0$ extends naturally to a
homomorphism $f\colon E\to K\wrwr G$. The case $f=\mathsf{id}$ and
$H=K$ is exactly the statement of the theorem, and the generalised
version is proven in exactly the same manner. The advantage of this
formulation is that one need not require that $H$ be normal in $E$;
and the more general statement is
\begin{thm}\label{thm:kkwreath}
  Let $H\le E$ be groups, and let $f_0\colon H\to K$ be a
  homomorphism. Then $f_0$ extends naturally to a
  homomorphism\footnote{recall that the core of the subgroup $H$ is
    the intersection of its conjugates. $E/\operatorname{core}(H)$ is
    the natural permutation group acting faithfully on the left coset
    space $H\backslash E$.}
  $f\colon E\to K\wrwr_{H\backslash E}E/\operatorname{core}(H)$.

  More precisely, let $G,H$ be groups, and let $G$ act on the right on a set
  $X$. Denote by $\pi\colon H\wrwr G\to G$ the natural
  projection. Then the map $E\mapsto E$ defines a bijection between
  \[\frac{\left\{\begin{array}{ll}E:&H\le E\text{ and the $G$-sets $X$ and $H\backslash E$ are isomorphic}\\&\text{via a homomorphism }E\to G\end{array}\right\}}{\text{isomorphism $E\to E'$ of groups intertwining the actions on $H\backslash E$ and $H\backslash E'$}}\]
    and
    \[\frac{\left\{\begin{array}{ll}E\le H\wrwr G:&\pi(E)\text{ transitive on }X\text{ and }\\&\ker(\pi)\cap E\overset\cong\longrightarrow H\text{ via }c\in H^X\mapsto c(x)\text{ for all }x\in X\end{array}\right\}}{\text{conjugacy of subgroups of }H\wrwr G}.
  \]
\end{thm}

Note that we could have written $E$ instead of
$E/\operatorname{core}(H)$ everywhere above. We included it to obtain
a smaller group acting on the set $X$, thus reinforcing the parallel
with Theorem~\ref{thm:kk}, and expressing $E$ as a subgroup of a
presumably easier-to-construct group. In particular, if
$[E:H]=d<\infty$, then $E/core(H)$ is a group of cardinality between
$d$ and $d!$, since it acts faithfully and transitively on the set
$H\backslash E$ of cardinality $d$.
\begin{proof}
  We follow the sketched proof of Theorem~\ref{thm:kk}; rather than a
  section $G\to E$, we choose a right transversal $T$ of $H$ in $E$;
  namely, a subset $T\subset E$ such that every $e\in E$ may uniquely
  be written in the form $ht$ with $h\in H,t\in T$.

  For $e\in E$, let $f(e)$ be the following $H$-decorated permutation
  of $H\backslash E$: the permutation is given by the natural
  right-multiplication action of $E$ on $H\backslash E$, and on every
  edge, say from $H t$ to $H u$ if $H t\cdot e=H u$, the label is
  $t e u^{-1}$.

  In formulas, this may be written as follows, although all
  verifications are easier in the ``decorated permutations'' form.
  Denote by $\pi$ the natural map $E\to E/\operatorname{core}(H)$. Set
  then $f(e)=c\pi(e)$ with $c\colon H\backslash E\to H$ given, for
  $t\in T$, by $c(Ht)=t e u^{-1}$ for the unique $u\in T$ such that
  $t e u^{-1}\in H$.

  Composing the map $c$ by the homomorphism $f_0\colon H\to K$ gives the
  first statement.
\end{proof}

Wreath products have occupied a prominent place in the theory of
permutation groups. For illustration, they appear as fundamental
constructions in the O'Nan-Scott theorem classifying maximal subgroups
of the symmetric group. We skip subcases
\textit{(iii.3)\dots(iii.viii)} which are too technical:
\begin{thm}[O'Nan-Scott,~\cite{scott:p-reps}]
  Let $G$ be a maximal subgroup of $\sym(X)$ for a finite set $X$. Then either
  \begin{enumerate}
  \item the action of $G$ is not transitive\footnote{i.e.\ there are
      at least two orbits on $X$}; then $X=Y\sqcup Z$ and
    $G=\sym(Y)\times\sym(Z)$; or
  \item the action of $G$ is transitive, but not
    primitive\footnote{i.e.\ there is a non-trivial $G$-invariant
      equivalence relation on $X$}; then $X=Y\times Z$ and
    $G=\sym(Y)\wr\sym(Z)$ in its imprimitive action; or
  \item the action of $G$ is primitive; and then either
    \begin{enumerate}
    \item[(3.i)] $G$ is an affine group over a finite field; or
    \item[(3.ii)] $X=Y^Z$ and $G=\sym(Y)\wr\sym(Z)$ in its primitive action; or
    \item[(3.iii)] \ldots\ (3.viii) \ldots\qed
    \end{enumerate}
  \end{enumerate}
\end{thm}

\subsection{Generators for wreath products}\label{ss:genwreath}
We now consider restricted wreath products $W=H\wr_X G$ in more
detail. We introduce the following notation: if $X$ be a set and $H$
be a group, then for all $x\in X$ and $h\in H$,
\[\text{we define $h@x\colon X\to H$ by }\begin{cases}(h@x)(x) = h,\\ (h@x)(y) = 1\text{ for all }y\neq x.\end{cases}.\]
\begin{wrapfigure}[21]{r}{4cm}
  \vspace{-3ex}
  \includegraphics[width=4cm]{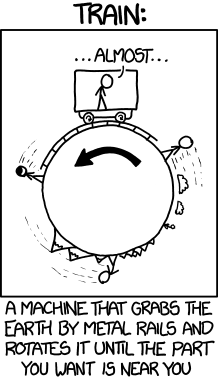}

  \centerline{\textsf{(reproduced from xkcd.com)}}
\end{wrapfigure}
One has the important formula relating conjugation in $W$ with translation:
\[\displaywidth=\parshapelength\numexpr\prevgraf+2\relax
(h@x)^g = h@xg\]
(check it by drawing the decorated permutations!)

Let $S$ be a generating set for $G$, and let $T$ be a generating set
for $H$. Let $Y\subseteq X$ be a choice of one representative from
each $G$-orbit on $X$. Then $W$ is generated by
\[\displaywidth=\parshapelength\numexpr\prevgraf+2\relax
\{t@y:t\in T,y\in Y\}\sqcup S.\]

A good, but also slightly misleading example, is the ``lamplighter
group'': this is the group $W=\{\pm1\}\wr\Z$. It is best understood by
its primitive action (on the space of functions $\{\pm1\}^\Z$):
imagine an infinite street with at each integer position a
lamppost. The lamp there can be ``on'' or ``off''. Imagine also that
the street is viewed from the perspective of the ``lamplighter'', namely
the person in charge of turning various lamps on and off. The generator
$s$ of $\Z$ means ``move up the street'', or equivalently ``shift all
lamps down'' relatively to the lamplighter; and the generator $t@1$
means ``change the state of the lamp currently in front of the
lamplighter''.

This is what is expressed by the cartoon above: relativistically
speaking, it makes no difference to think that a train moves on its
tracks, or that the tracks move under the train. The lamplighter is on
the train, and the lamp configurations are on the track.

The misleading aspect of this example is that one must remember that,
when the lamplighter moves in one direction, the lamps actually move
in the \emph{opposite} direction. In a regular wreath product, this
makes no difference; but in the general case it does.

\begin{defn}\label{def:schreier}
  Let $G=\langle S\rangle$ be a finitely generated group acting on the
  right on a set $X$. The associated \emph{Schreier graph} has vertex
  set $X$, and for each $x\in X,s\in S$ an edge from $x$ to $x s$.
\end{defn}
In the case $X=G$ of a group $G$ acting on itself by (right or left)
multiplication, one obtains the (right or left) Cayley graph of
$G$. The group $G$ acts then by graph isometries on its Cayley graph
by (left or right) multiplication.
\begin{exse}
  Consider $G=\sym(4)$ generated by $\{(1,2),(2,3),(3,4)\}$. Draw the
  Schreier graphs of (in order of difficulty)
  \begin{enumerate}
  \item the action of $G$ on $\{1,2,3,4\}$;
  \item the action on the collection of $2$-element subsets of
    $\{1,2,3,4\}$;
  \item the action of $G$ on itself by conjugation.
  \end{enumerate}
\end{exse}

It would be a great mistake to think that a path in the permutational
wreath product $W=H\wr_X G$ is a path in the Schreier graph of $G$
with decorations in $H$ along the path. We shall see where exactly the
decorations appear, but their positions are rather at the
\emph{inverses} of path points. To complicate matters, in case $G$ is
abelian the map $x\mapsto x^{-1}$ is an automorphism of $G$, whence
the above ``mistaken'' description luckily gives the correct answer.

Let us fix generating sets $G=\langle S\rangle$ and
$H=\langle T\rangle$, and assume for simplicity that $G$ acts
transitively on $X$. We also choose a base point $x\in X$. Then $W$ is
naturally generated by $T@x\sqcup S$.

We consider first the description of elements of $W$ as $f g$ with
$f\colon X\to H$ and $g\in G$. This is the ``lamplighter moving''
version, because the ``lamp vector'' $f$ is not shifted, but the
position at which it is changed varies. By~\eqref{eq:wreathmult}, the
right action on $w=f g$ of
\begin{list}{}{}
\item --- a generator $s\in S$ gives $w s=f(gs)$;
\item --- a generator $t@x\in T@x$ gives $w(t@x)=(f\cdot t@xg^{-1})g$.
\end{list}
Thus if $w=(t_0@x)s_1(t_1@x)\cdots s_\ell(t_\ell@x)=f g$, then the
support of $f$ is included in $\{x,x s_1^{-1},\dots,x (s_1\dots
s_\ell)^{-1}\}$.

We may also describe elements of $W$ as $g f$; this is the ``earth
moving'' version, because the ``lamp vector'' $f$ is shifted, and the
lamplighter always changes the lamp at a fixed
position. By~\eqref{eq:wreathmult}, the right action on $w=g f$ of
\begin{list}{}{}
\item --- a generator $s\in S$ gives $w s=(gs)\,{}^{s^{-1}}\!f$;
\item --- a generator $t@x\in T@x$ gives $w(t@x)=g(f\cdot t@x)$.
\end{list}
Thus if $w=(t_0@x)s_1(t_1@x)\cdots s_\ell(t_\ell@x)=g f$, then the
support of $f$ is included in $\{x s_1\dots s_\ell,x s_2\dots
s_\ell,\dots,x s_\ell,x\}$.

The semidirect product description of $W=H\wr_X G$ gives a
presentation of $W$ by generators and relations;
see~\cite{cornulier:fpwreath}. Generating as above $W$ by $T@Y\sqcup
S$, we get
\begin{align*}
  W=\big\langle T@Y\sqcup S\,\big|\,&\text{relations of $G$, relations of $H$, and }\\
  &\forall t,t'\in T,\,\forall y\neq y'\in Y:[t@y,t'@y'],\text{ and }\\
  &\forall t\in T,\,\forall y\in Y,\,\forall g\in G_y:[t@y,g],\text{ and }\\
  &\forall t,t'\in T,\,\forall y\in Y,\,\forall g\in (G_y\backslash G)\setminus\{G_y\}:[t@y,(t'@y)^g]\big\rangle.
\end{align*}
The exact criterion, assuming $X\neq\emptyset$ and $H\neq1$, is:
\begin{thm}[Cornulier~\cite{cornulier:fpwreath}*{Theorem~1.1}]\label{thm:cornulier}
  The wreath product $W$ is finitely presented if and only if both
  $G,H$ are finitely presented, $G$ acts on $X$ with finitely
  generated stabilizers, and $G$ acts diagonally on $X\times X$ with
  finitely many orbits.\qed
\end{thm}

\subsection*{Diestel-Leader graphs}
We present a particularly intuitive description of the Cayley graph of
``lamplighter groups'' $W=F\wr\Z$, for a finite group $F$ of
cardinality $q$, see~\cite{woess:dl}.

Let $\mathscr T$ denote the $(q+1)$-regular tree, choose a basepoint
$o$, and a geodesic ray $\omega\colon\N\to\mathscr T$ starting at
$o$. Imagine $\mathscr T$ as ``hanging from $\omega$'': orient the
edges on $\omega$ from $\omega(n)$ to $\omega(n+1)$, and orient every
other edge from its furthest point to $\omega$ to its closest. Define
then $h\colon\mathscr T\to\Z$ as follows: for each $x\in\mathscr T$,
there is a unique path in $\mathscr T$ from $o$ to $x$, and set
$h(x)=($number of edges oriented forward$)-($number of edges oriented
backward$)$ on this path\footnote{This is usually called a
  \emph{Busemann function}}.

Consider now the following graph $\mathcal B(q,q)$. Its vertex set is
$\{(x,y)\in\mathscr T\times \mathscr T:h(x)+h(y)=0\}$. There is an edge
from $(x,y)$ to $(x',y')$ precisely if $\{x,x'\}$ and $\{y,y'\}$ are
connected in $\mathscr T$. Here is a portion of $\mathcal B(2,2)$,
with one tree pointing up and one tree pointing down:

\tdplotsetmaincoords{80}{110}
\begin{center}\begin{tikzpicture}[tdplot_main_coords]
  \coordinate (x000) at (0,-3,0);
  \coordinate (x001) at (0,-2,0);
  \coordinate (x010) at (0,-1.4,0);
  \coordinate (x011) at (0,-0.4,0);
  \coordinate (x100) at (0,0.4,0);
  \coordinate (x101) at (0,1.4,0);
  \coordinate (x110) at (0,2,0);
  \coordinate (x111) at (0,3,0);
  \coordinate (0x00) at (1,-2,1);
  \coordinate (0x01) at (1,-0.7,1);
  \coordinate (0x10) at (1,0.7,1);
  \coordinate (0x11) at (1,2,1);
  \coordinate (1x00) at (-1,-2,1);
  \coordinate (1x01) at (-1,-0.7,1);
  \coordinate (1x10) at (-1,0.7,1);
  \coordinate (1x11) at (-1,2,1);
  \coordinate (00x0) at (2,-1,2);
  \coordinate (00x1) at (2,1,2);
  \coordinate (01x0) at (0.7,-1,2);
  \coordinate (01x1) at (-0.7,1,2);
  \coordinate (10x0) at (-0.7,-1,2);
  \coordinate (10x1) at (0.7,1,2);
  \coordinate (11x0) at (-2,-1,2);
  \coordinate (11x1) at (-2,1,2);
  \coordinate (000x) at (3,0,3);
  \coordinate (001x) at (-0.4,0,3);
  \coordinate (010x) at (1.4,0,3);
  \coordinate (011x) at (-2,0,3);
  \coordinate (100x) at (2,0,3);
  \coordinate (101x) at (-1.4,0,3);
  \coordinate (110x) at (0.4,0,3);
  \coordinate (111x) at (-3,0,3);
  \foreach\i in {0,1}\foreach\j in {0,1}\foreach\k in {0,1}\foreach\l in {0,1} {
    \draw (x\i\j\k) -- (\l x\i\j);
    \draw (\i x\j\k) -- (\l\i x\j);
    \draw (\i\j x\k) -- (\l\i\j x);
    \draw[thick] (x\i\j\k) -- (0x\i\j) -- (00x\i) -- (000x);
    \draw[thick] (x111) -- (\k x11) -- (\j\k x1) -- (\i\j\k x);
  };
  \fill (00x0) circle [radius=2pt];
  \fill (11x1) circle [radius=2pt];
\end{tikzpicture}\end{center}

\begin{prop}
  Consider a wreath product $G=F\wr\Z$ with a finite group $F$ of
  cardinality $q$. Denote by $s$ a generator of $\Z$, and consider for
  $G$ the generating set $S=(F@1)s\sqcup s^{-1}(F@1)$. Then the Cayley
  graph of $(G,S)$ is $\mathcal B(q,q)$.
\end{prop}
\begin{proof}
  Only in this proof, for a subset $A$ of the integers, let us denote
  by $F^{A}$ the set of finitely-supported functions $A\to F$.  The
  vertex set of $\mathscr T$ may be identified with
  $\bigsqcup_{n\in\Z}F^{(-\infty,n]}\times\{n\}$, in such a manner
  that there is an edge from $(\sigma,n)$ to
  $(\sigma|_{(-\infty,n-1)},n-1)$ for all $n\in\Z$ and all
  $\sigma\in F^{(\infty,n]}$. Vertices on the ray $\omega$ are those
  of the form $(1,n)$ with $n\le0$.

  Therefore, the vertex set of $\mathcal B(q,q)$ may be identified
  with
  $\bigsqcup_{n\in\Z}F^{(-\infty,n]}\times
  F^{(-\infty,-n]}\times\{n\}$.
  This is easily put in correspondence with $F^\Z\times\Z$ via the map
  $(\sigma,\sigma',n)\mapsto(\tau,n)$ with $\tau\colon\Z\to F$ the
  finitely supported function given by $\tau(k)=\sigma(k)$ if $k\le n$
  and $\tau(k)=\sigma'(1-k)$ if $k>n$.

  Thus we put the vertex set of $\mathcal B(q,q)$ is bijection with
  $G$. It is now routine to check that the generators in $S$ produce
  the edges of $\mathcal B(q,q)$.
\end{proof}

This description of lamplighter groups makes some of their geometric
features quite transparent. For example, let us consider a finitely
generated group $G=\langle S\rangle$, and its Cayley graph. A vertex
$v\in G$ is called a \emph{dead end} if all neighbours of $v$ are at
least as close to $1$ as $v$. More generally, let us say $v$ is
\emph{on a $k$-hill} if all paths from $v$ to an element of norm
$\|v\|+1$ has to go through a vertex of norm $\|v\|-k$. In other
words, from the top of the hill the only way of going to infinity is
to first go down at least $k$ steps.

Consider the two dots on the picture above. Say one of them is the
origin $1$. Then the other one is on a $1$-hill. More generally, any
element in the lamplighter group that is reached from the origin by
going down $k$ steps, up $2k$ steps and down again $k$ steps along a
reduced path reaches the top of a $k$-hill.

This is an fact a familiar phenomenon: consider a long street and two
remote addresses we want to visit on that street, in whichever order,
from our starting point on the street. The shortest way of doing this
is to first go to the closest, and then the other one. The
\emph{worst} possible place to start is at equal distance from both
addresses.

\section{Growth of groups}\label{ss:growth}
This section is not a treatise on growth of groups; for that, see
rather~\cite{mann:howgroupsgrow}. We do recall some elementary, basic
notions, and provide some motivation for the material to appear
later. Let $G$ be a group generated by a finite symmetric set
($S=S^{-1}$). One defines the \emph{word norm} on $G$ by
\[\|g\|=\min\{n:g=s_1\cdots s_n,\,s_i\in S\},
\]
and a distance\footnote{it is only here that we use the fact that $S$
  is symmetric, to obtain $d(g,h)=d(h,g)$; in fact, it suffices in all
  that follows to assume that $S$ generates $G$ as a monoid.} on $G$
that is invariant under left
translation\footnote{i.e. $d(t g,th)=d(g,h)$}:
\[d(g,h)=\|g^{-1}h\|.\]

Thus $G$ is viewed as a normed space, and as a metric space on which
$G$ acts by isometries via left translation.

\subsection{Formal growth}\label{ss:formal} One may be interested
in regularity properties of the metric space $G$; these are best
studied via the \emph{growth series}, the formal power series
\[\Gamma_G(z)=\sum_{g\in G}z^{\|g\|}\in\Z[[z]].\]

The natural questions that arise are: what is the domain of
convergence of $\Gamma_G(z)$? What can be said of analytic
continuations of $\Gamma_G(z)$? What are its singularities? Is the
function $\Gamma_G(z)$ rational (i.e.\ in $\Q(z)$)? or at least
algebraic (i.e.\ there exists a two-variable polynomial
$F(y,z)\in\Z[y,z]$ with $F(\Gamma_G(z),z)\equiv0$)?

The consideration of the power series $\Gamma_G(z)$, and of the above
questions, is justified by the answers that have been given:
\begin{itemize}
\item $\Gamma_G(z)$ converges in a disk of radius at least $1/\#S$. If
  $S$ is symmetric, then the convergence radius is it fact at least
  $1/(\#S-1)$, with equality if and only $G$ is a free product of
  $\Z$'s and $C_2$'s with its natural generating set, as in
  Definition~\ref{defn:freelike} below.
\item If groups $G,H$ are respectively generated by $S,T$, then the
  direct product $G\times H$ is naturally generated by $S\sqcup
  T$. One then has
  \[\Gamma_{G\times H}(z)=\Gamma_G(z)\Gamma_H(z),\]
  see Proposition~\ref{prop:CG:x}.
\item If groups $G,H$ are respectively generated by $S,T$, then the
  free product $G*H$ is naturally generated by $S\sqcup T$. One
  then has
  \[\frac1{\Gamma_{G*H}(z)}=\frac1{\Gamma_G(z)}+\frac1{\Gamma_H(z)}-1,\]
  see Proposition~\ref{prop:CG:*}.
\item If the group $G$ is virtually abelian~\cite{benson:zn}, or
  word-hyperbolic~\cite{gromov:hyperbolic}, or the discrete Heisenberg
  group~\cite{duchin-shapiro:heisenberg}
  $H_3=\big(\begin{smallmatrix}1&\Z&\Z\\&1&\Z\\&&1\end{smallmatrix}\big)$,
  then $\Gamma_G$ is a rational function of $z$ for all choices of the
  finite generating set $S$.
\item Wreath products give some examples of power series $\Gamma_G$
  that are algebraic functions, as we shall see
  in Corollary~\ref{cor:parry} below.
\item If $G$ is a $2$-step nilpotent group with cyclic derived
  subgroup, then there exist generating sets for $G$ such that
  $\Gamma_G$ is a rational function of $z$. However, if $G$ is the
  $5$-dimensional Heisenberg group
  $H_5=\bigg(\begin{smallmatrix}1&\Z&\Z&\Z&\Z\\&1&&&\Z\\&&1&&\Z\\&&&1&\Z\\&&&&1\end{smallmatrix}\bigg)$,
  then there exist generating sets for which $\Gamma_G(z)$ is
  transcendental; see~\cite{stoll:heisenberg}.
\end{itemize}

In respect to this last point, note that, for $G$ nilpotent, the growth
of $G$ is polynomial so $\Gamma_G(z)$ converges in the unit disk. It
is either rational or transcendental, by the Fatou
theorem~\cite{fatou:seriesentieres}. More is known:
\begin{thm}[P\'olya-Carlson~\cite{carlson:potenzreihen}]\label{thm:polya-carlson}
  Let $A(z)=\sum_{n\ge0} a_n z^n$ be a power series with integer
  co\"efficients. If $A$ is not rational, then $A$ does not extend
  analytically beyond the unit circle.
\end{thm}
  
Most importantly, the growth series is a convenient object that
encodes information on $G$. A quite satisfactory theory of ``Euler
characteristic'' has been developed for groups,
see~\cite{chiswell:euler}. Here is a special case: if $G$ is the
fundamental group of a cellular complex $\mathscr X$ with contractible
universal cover, one declares $\chi(G)$ to be $\chi(\mathscr X)$. More
generally, if $G$ has a finite-index subgroup $H$ which is the
fundamental group of a space $\mathscr Y$, one sets
$\chi(G)=\chi(\mathscr Y)/[G:H]$; this makes sense because if $G$ is
the fundamental group of $\mathscr X$ then $H$ is the fundamental
group of a $[G:H]$-sheeted covering of $\mathscr X$, whose Euler
characteristic is $[G:H]\chi(\mathscr X)$. In particular, if $G$ is
finite then $\chi(G)=1/\#G$. This led to the idea that $1/\Gamma_G(z)$
could behave like an Euler characteristic, and that its limit
$1/\Gamma_G(1)$ could express $\chi(G)$. This is not always true, but
it does hold in some illustrative cases.

Let us compute for instance the growth series of a free group $F_k$
generated by a
basis\footnote{i.e. $S=\{x_1,x_1^{-1},\dots,x_k,x_k^{-1}\}$ and $F_k$
  may be identified with reduced words over $S$.}. It follows from the
formula for free products, or by direct counting if one notes, for all
$\ell\ge1$, that there are $2k(2k-1)^{\ell-1}$ elements of norm $\ell$
in $F_k$, that
\begin{equation}\label{eq:growthFm0}
  \Gamma_{F_k}(z)=\frac{1+z}{1-(2k-1)z}.
\end{equation}
The value $\Gamma_{F_k}(1)$ is uniquely defined by analytic
continuation, and one has $1/\Gamma_{F_k}(1)=1-k$, in agreement with
$F_k$ being the fundamental group of a graph with $1$ vertex and $k$
edges. See~\cites{floyd-p:fuchsian,MR97d:20031a,smythe:euler,lewin:graphgroup}
for more such examples of the `$1/\Gamma_G(1)=\chi(G)$'
phenomenon.

\subsection{Complete growth series}\label{ss:completegrowth}
There exists a stronger property than having a rational growth series:
a group $G=\langle S\rangle$ has a \emph{rational geodesic combing} if
there exists a finite directed graph with edge labels in $S$ and a
fixed ``initial'' vertex, such that the set $L\subseteq S^*$ of words
read from the initial vertex along paths in the graph has the
following property: $L$ maps bijectively to $G$ by the natural
evaluation map of words as elements of $G$, and the words in $L$ have
minimal length among all words in $S^*$ having the same evaluation in
$G$. Kervaire suggested to consider the \emph{complete} growth series
\[\CG_G(z)=\sum_{g\in G}g z^{\|g\|}\in\Z G[[z]].
\]
Note that $\CG_G$ depends on the choice of generating set $S$, even
though we do not mention it explicitly. The series $\CG_G(z)$ is a
power series with co\"efficients in the group ring, and one may again
ask whether it is rational or algebraic. Since $\Z G$ need not be
commutative, let us define more precisely these notions; we refer
to~\cite{salomaa-s:fps} for details. Let
$\Lambda\subseteq\overline\Lambda$ be rings. An \emph{algebraic system
  over $\Lambda$} in variables $X_1,\dots,X_n$ is a
non-degenerate\footnote{Let us not detail this too much; suffice it to
  say that the system must have a unique solution once its initial
  terms $f_1(0),\dots,f_n(0)$ have been fixed.}  $n$-tuple of
polynomials $P_1,\dots,P_n$ in non-commuting indeterminates
$X_1,\dots,X_n$ and co\"efficients in $\Lambda$. In a \emph{linear
  system over $\Lambda$}, the polynomials are restricted to have
degree $1$ and contain the indeterminate on the right; i.e.\ the $P_i$
are sums of monomials all belonging to
$\Lambda\cup\bigcup_{1\le i\le n}\Lambda X_i$. A \emph{solution} is an
$n$-tuple $(f_1,\dots,f_n)\in\overline\Lambda^n$ such that all
$P_i(f_1,\dots,f_n)=f_i$ for all $i=1,\dots,n$.

We then say that a power series $F(z)\in\Z G[[z]]$ is \emph{rational},
respectively \emph{algebraic}, if it is the first co\"ordinate of the
solution of a linear, respectively algebraic system over the
polynomial ring $\Z G[z]$. A more direct definition of the ring of
rational functions is that it is the smallest subring of $\Z G[[z]]$
containing $\Z G[z]$ and closed under Kleene's \emph{quasi-inversion},
the operation $F(z)^*=(1-F(z))^{-1}=1+F(z)+F(z)^2+\cdots$ defined for
all $F(z)\in \Z G[[z]]$ with $F(0)=0$.

\begin{exse}
  If $G$ admits a rational geodesic combing, then its growth series is
  rational.

  Hint: define one variable $X_i$ for each vertex $i$ of the graph
  defining the combing, and encode the edges of the graph into
  polynomials.
\end{exse}

\begin{exse}\label{exse:quasi-geodesic combing}
  If the growth series of $G$ is rational, then $G$ admits a
  \emph{quasi-geodesic combing}: a language $L\subset S^*$, recognised
  by a finite graph as above, with the property that, for some
  constant $C\in\N$, all words $s_1\dots s_\ell\in L$ have the
  property $\ell\le C\|s_1\cdots s_\ell\|$.

  Hint: consider a polynomial system defining $\CG_{G,S}(z)$, make
  sure that every term in $\Z G$ is accompanied by at least one factor
  $z$. Write the terms in $\Z G$ as linear combinations of words over
  $S$, yielding a polynomial system over $\Z S^*$. Let $C$ be the
  maximal length of all these words over $S$ that appear in the
  polynomial system. A solution to the polynomial system will be a sum
  of monomials of the form $s_1\dots s_\ell z^n$, where
  $n=\|s_1\cdots s_\ell\|\ge \ell/C$.
\end{exse}

These notions strengthen the ones for the classical growth series: if
$\CG_G(z)$ is rational or algebraic, then its image under
the augmentation map $\Z G\twoheadrightarrow\Z$ is rational or
algebraic. On the other hand, statements concerning the complete
growth series are usually not much harder to prove than the analogous
ones concerning the classical growth series:

\begin{prop}\label{prop:CG:x}
  Let the groups $G,H$ and $G\times H$ be respectively generated by
  $S,T$ and $S\sqcup T$. One then has
  \[\CG_{G\times H}(z)=\CG_G(z)\CG_H(z).\]
\end{prop}
\begin{proof}
  Every element $(g,h)\in G\times H$ satisfies
  $\|(g,h)\|=\|g\|+\|h\|$; so
  \[\CG_{G\times H}(z) = \sum_{(g,h)\in G\times H}g h z^{\|g\|+\|h\|}
  = \sum_{g\in G}g z^{\|g\|}\sum_{h\in H}h z^{\|h\|} = \CG_G(z)\CG_H(z).\qedhere\]
\end{proof}

\begin{prop}\label{prop:CG:*}
  Let the groups $G,H$ and $G*H$ be respectively generated by
  $S,T$ and $S\sqcup T$. One then has
  \[\frac1{\CG_{G*H}(z)}=\frac1{\CG_G(z)}+\frac1{\CG_H(z)}-1.\]
\end{prop}
\begin{proof}
  Every element of $w\in G*H$ may be uniquely written in the form
  $w=h_0 g_1 h_1\cdots g_\ell$ with $h_0\in H$,
  $h_1,\dots,h_{\ell-1}\in H\setminus\{1\}$, $g_1,\dots,g_{\ell-1}\in
  G\setminus\{1\}$, $g_\ell\in G$. Thus
  \begin{align*}
    \CG_{G*H}(z) &= \sum_{\ell\ge0}\CG_H(z)((\CG_G(z)-1)(\CG_H(z)-1))^\ell\CG_G(z)\\
    &= \CG_H(z)\frac1{1-(\CG_G(z)-1)(\CG_H(z)-1)}\CG_G(z);
    \intertext{so}
    \frac1{\CG_{G*H}(z)} &= \frac1{\CG_G(z)}(\CG_G(z)+\CG_H(z)-\CG_G(z)\CG_H(z))\frac1{\CG_H(z)}\\
    & = \frac1{\CG_H(z)} + \frac1{\CG_G(z)} -1.\qedhere
  \end{align*}
\end{proof}

These results are generalized to \emph{graph products}
in~\cite{allen+:growthproducts}: given a graph $\Gamma$ with vertex
set $V$ and a group $G_v$ for each $v\in V$, the \emph{graph product}
of the $G_v$ is
\[G_\Gamma:=\bigast_{v\in V}G_v\Big/\big\langle[G_v,G_w]\text{ for each edge }(v,w)\big\rangle.
\]
Recall that a \emph{clique} in a graph is a subset of the vertices any
two of which are connected by an edge. They show:
\begin{prop}[\cite{allen+:growthproducts}*{Theorem~3.8}]
  Let each group $G_v$ have generating set $S_v$, and consider the
  generating set $\bigcup_{v\in V}S_v$ of $G_\Gamma$. Then
  \[\frac1{\CG_{G_\Gamma}(z)}=\sum_{\text{clique }W\subseteq V}\prod_{v\in W}\Big(\frac1{\CG_{G_v}}-1\Big).\]
\end{prop}

There are few classes of groups in which $\CG_G(z)$ is rational for
all choices of generating set:
\begin{prop}[Liardet~\cite{liardet:phd}]
  If $G$ is virtually abelian then $\CG_G(z)$ is rational for all
  choices of generating set.\qed
\end{prop}

\begin{prop}[Grigorchuk-Nagnibeda~\cite{grigorchuk-n:complete}]
  If $G$ is word-hyperbolic then $\CG_G(z)$ is rational for all
  choices of generating set.\qed
\end{prop}

Note that there is no need to consider rings such as $\Z G$; the
definition is more naturally phrased in terms of a \emph{semiring}
such as $\N G$. The polynomials $P_i$ are restricted to be sums of
products of monomials, and no subtraction is allowed. The notions of
$\Z$-rationality and $\N$-rationality differ subtly, see
e.g.~\cite{berstel:Nrat}.

Let us now compute explicitly the complete growth series of a free
group. It simplifies a little the notation to answer a slightly more
general question. We denote throughout the text the cyclic group of
order $p$ by $C_p$:
\begin{defn}\label{defn:freelike}
  A \emph{free-like} group is a finite free product of $\Z$'s and
  $C_2$'s.

  Say that a free-like group $G$ has $m_1$ factors isomorphic to
  $C_2$ and $m_2$ factors isomorphic to $\Z$; then it has a
  symmetric generating set of size $m_1+2m_2$, consisting of one
  generator for each $C_2$ and a generator and its inverse for each
  $\Z$. The group $G$ is characterised by the property that its Cayley
  graph (see after Definition~\ref{def:schreier}) is an
  $(m_1+2m_2)$-regular undirected tree. We call such an $S$ a
  \emph{natural} generating set for $G$.
\end{defn}
For instance, the free group $F_{m/2}$ is free-like for $m$ even, and
$\text{\LARGE$\ast$}^m C_2$ is free-like.

Let $G$ be free-like, and let $S$ denote a natural generating set of
$G$ with cardinality $m$. We shall see that $G$ has rational complete
growth series. For ease of notation, we write $\overline s$ for
$s^{-1}$. We identify elements of $G$ with \emph{reduced} words over
$S$, i.e.\ words not containing consecutive $s\overline s$.

For all $s\in S$, define $F_s\in\Z G[[z]]$ by
\[F_s=\sum_{\substack{w\in G,\text{ not}\\\text{starting with
    }\overline s}} w z^{\|w\|};
\]
so $\CG_G(z)=1+\sum_{s\in S}s z F_s$. We have the linear system in
non-commutative unknowns $F_s$
\[F_s = 1 + \sum_{t\in S,t\neq \overline s}t z F_t,
\]
with solution $F_s=(1-\overline s z)(1-\sum_{t\in S}t
z+(m-1)z^2)^{-1}$, so finally
\begin{equation}\label{eq:growthFm}
  \CG_G(z)=\frac{1-z^2}{1-\sum_{s\in S}s z+(m-1)z^2},
\end{equation}
compare with~\eqref{eq:growthFm0}.

\subsection{Asymptotic growth}\label{ss:asymptotic}
We return to a group $G$ with generating set $S$, and view it as a
metric space for the word metric. We consider the volume growth of
balls in the metric space $(G,d)$; this is the \emph{growth function}
$v_G\colon\R_+\to\N$ given by
\[v_G(R)=\#\{g\in G: \|g\|\le R\}.\]

It is naturally related to the formal power series $\Gamma(z)$: indeed
$v_G(R)$ is the sum of all co\"efficients of $\Gamma(z)$ of degree
$\le R$; equivalently, for $R\in\N$ it is the degree-$R$ co\"efficient
of $\Gamma(z)/(1-z)$.  Thus, by Tauberian and Abelian theorems (see
e.g.~\cite{nathanson:density}), asymptotics of $v_G(R)$ as
$R\to\infty$ may be related to asymptotics of $\Gamma_G(z)$ as $z\to$
the convergence radius. In particular, the function $v_G(R)$
grows as $R^d$ if and only if $\Gamma_G(z)$ converges in the unit disk
and has an order-$d$ pole singularity at $1$.

The norm $\|\cdot\|$ depends on the choice of generating set $S$, but
only mildly: different choices of generating sets give equivalent
norms, and equivalent metrics. If for $v,w\colon\R_+\to\N$ we write
$v\precsim w$ to mean that $v(R)\le w(CR)$ for a constant $C\in\R_+$
and all $R\gg0$, and we write $v\sim w$ to mean $v\precsim w\precsim
v$, then
\begin{lem}\label{lem:indepgens}
  The $\sim$-equivalence class of $v_G$ is independent of the
  choice of generating set.
\end{lem}
\begin{proof}
  Let $S,S'$ be two finite generating sets for $G$, and let us
  temporarily write $\|g\|_S,\|g\|_{S'}$ and
  $v_{G,S},v_{G,S'}$ for the norms and growth functions with
  respect to $S,S'$. There exists then a constant $C\in\N$ such that
  $\|s'\|_S\le C$ for all $s'\in S'$, and thus $\|g\|_{S'}\le
  C\|g\|_S$. This gives $v_{G,S}(R)\le v_{G,S'}(CR)$. The
  reverse inequality holds by symmetry.
\end{proof}

Note, as a consequence, that all exponentially-growing functions are
equivalent, and that $R^d$ and $C\cdot R^d$ are equivalent as soon as
$d>0$. The exponential growth rate
\begin{equation}\label{eq:lambda}
  \lambda_{G,S}=\lim v_{G,S}(R)^{1/R}
\end{equation}
is nevertheless worthy of consideration, and will be discussed
in~\S\ref{ss:nug}.

\subsection{History} Interest in asymptotic growth of groups dates
back at least to the early 1950's, in the works of
Krause~\cite{krause:growth}, Efremovich~\cite{efremovich:proximity}
and \v Svarc~\cite{svarts:growth}; they were seeking coarse invariants
of manifolds based on their fundamental group. Milnor noted
in~\cite{milnor:curvature} that, if $G$ is the fundamental group of a
compact riemannian manifold $\mathscr M$, then $v_G$ is equivalent to
the volume growth of balls in the universal cover of $\mathscr M$.

Here is a schematic of the known equivalence classes of growth
functions of groups. Note the two dots for the two groups of order
$4$, respectively the two groups $\Z^4$ and the Heisenberg group $H_3$
with quartic growth:
\[\begin{tikzpicture}
  \node at (0,0.5) {$1$};
  \node at (0,0) {$\bullet$};
  \node at (0,-0.8) {$1$};

  \node (z/2) at (0.3,1) {$\tfrac{\Z}{2\Z}$};
  \node at (0.3,0) {$\bullet$};
  \draw[dotted] (z/2) -- (0.3,0);
  \node at (0.3,-0.8) {$2$};

  \node at (0.6,0) {$\bullet$};
  \node at (0.6,-0.8) {$3$};

  \node at (0.9,0.2) {$\bullet$};
  \node at (0.9,-0.2) {$\bullet$};
  \node at (0.9,-0.8) {$4$};

  \node (z/5) at (1.2,1) {$\tfrac{\Z}{5\Z}$};
  \node at (1.2,0) {$\bullet$};
  \draw[dotted] (z/5) -- (1.2,0);
  \node at (1.2,-0.8) {$5$};

  \node at (1.6,0) {$\cdots$};
  \node at (1.6,-0.8) {$\cdots$};

  \node at (2,0.5) {$\Z$};
  \node at (2,0) {$\bullet$};
  \node at (2.0,-0.6) {$R$};

  \node (z2) at (2.3,1) {$\Z^2$};
  \node at (2.3,0) {$\bullet$};
  \node (r2) at (2.3,-1.0) {$R^2$};
  \draw[dotted] (z2) -- (r2);

  \node at (2.6,0) {$\bullet$};
  \node at (2.6,-0.6) {$R^3$};

  \node at (2.9,1.0) {$\Z^4$};
  \node at (2.9,0.5) {$H_3$};
  \node at (2.9,0.2) {$\bullet$};
  \node at (2.9,-0.2) {$\bullet$};
  \node (r4) at (2.9,-1.0) {$R^4$};
  \draw[dotted] (2.9,-0.2) -- (r4);

  \node at (3.4,0) {$\cdots$};
  \node at (3.4,-0.8) {$\cdots$};

  \draw (4.6,0) ellipse (10mm and 4mm);
  \node at (4.6,0) {gap};
  \node at (5.6,-0.8) {$v_{\min}$};

  \node at (6.25,0) {$???$};

  \node at (7,0.5) {$W_{012}$};
  \node at (7,0) {$\bullet$};
  \node at (7,-0.8) {$\exp(R^{0.76\dots})$};

  \draw[pattern=north west lines] (8.5,0) ellipse (14mm and 4mm);

  \node at (10,0.5) {$F_2$};
  \node at (10,0) {$\bullet$};
  \node at (10,-0.8) {$\exp(R)$};
\end{tikzpicture}\]

The left of the graph is occupied by finite groups; the growth of a
finite group is equivalent to the constant function taking value the
order of the group. Abelian groups, and more generally
virtually\footnote{A property is said to \emph{virtually} hold if it
  holds for a finite-index subgroup.} nilpotent groups have polynomial
growth of type $R^d$ for an integer $d$. The converse is a deep result
by Gromov:
\begin{thm}[Gromov~\cite{gromov:nilpotent}]\label{thm:gromov}
  A finitely generated group has growth function bounded by a
  polynomial if and only if it is virtually nilpotent.\qed
\end{thm}

It follows also from Gromov's argument that there exists a
superpolynomial function $v_{\min}(R)$ such that all groups with
growth $\precnsim v_{\min}$ are virtually nilpotent; so there are no
functions with growth strictly between polynomial and
$v_{\min}$. Explicit estimates in~\cite{shalom-tao:gromov} imply that
one may take $v_{\min}(R)=R^{(\log R)^{1/100}}$, although the gap is
probably larger. Note also that there need not exist a
$\precnsim$-largest function $v_{\min}$. If one restricts oneself to
virtually residually\footnote{A property is said to hold
  \emph{residually} if for every non-trivial element there exists a
  quotient in which this element remains non-trivial and the property
  holds.} nilpotent groups, then the gap extends at least to
$\exp(R^{1/2})$, see~\cite{grigorchuk:hp}; if one restricts oneself to
virtually residually solvable groups, then the gap extends at least to
$\exp(R^{1/6})$, see~\cite{wilson:gap}.

Milnor asked in 1968, in a famous problem in the ``American Math
Monthly''~\cite{milnor:5603}, whether there exist groups whose growth
function is neither polynomial nor exponential. He also conjectured in
that note that groups of polynomial growth are precisely the virtually
nilpotent groups. Milnor and Wolf showed
in~\cites{milnor:solvable,wolf:solvable} that virtually solvable
groups have either polynomial or exponential growth, and the
\emph{in}existence of groups with growth between polynomial and
exponential became known as ``Wolf's conjecture''. Recall the
celebrated ``Tits' alternative''~\cite{tits:linear}: a finitely
generated subgroup of a linear group in characteristic $0$ either is
virtually solvable or contains a non-abelian free subgroup; from this
it follows that linear groups always have polynomial or exponential
growth (see furthermore~\S\ref{ss:nug}).

However, groups of \emph{intermediate growth} exist, and
Grigorchuk~\cite{grigorchuk:growth} gave such an example, known as the
\emph{first Grigorchuk group} $G_{012}$; see~\S\ref{ss:ss}.

The growth of $G_{012}$ is not known, even up to $\sim$-equivalence;
conjecturally, it is the same as the growth of the group
$W_{012}(C_2)$, which will be introduced
in~\S\ref{ss:ssgrowth}. The hatched region above indicates that, in
fact, there are many groups of intermediate growth, and that any
``reasonable'' function between $\exp(R^{0.76\dots})$ and $\exp(R)$ is
equivalent to the growth function of a group, see
Theorem~\ref{thm:givengrowth}.

There are at least two arguments for considering asymptotic growth
rather than exact growth of groups. Firstly, the asymptotics of the
growth function does not depend on the generating set, by
Lemma~\ref{lem:indepgens}, so is an invariant of the group itself.
Secondly, we expect ``most'' growth series to be transcendental power
series, so that they are probably difficult to describe, manipulate or
expand; this happens e.g.\ for groups of subexponential growth, whose
growth series converges in the unit disk so is either rational or
transcendental, by Fatou's theorem (see
Theorem~\ref{thm:polya-carlson}).

\section{Growth of regular wreath products}\label{ss:growthrw}
We consider in this section a wreath product $W=H\wr_X G$, and compute
its growth series. We assume that generating sets $S,T$ for $G,H$
respectively have been chosen, and that the growth series of $G$ and
$H$ are known.

\subsection{Wreath products over finite sets}
As a first step, let us suppose that the set $X$ is finite, say
$X=\{x_1,\dots,x_d\}$. Then, as generating set for $W$, we may take
$\{t@x:t\in T,x\in X\}\sqcup S$. For this generating set, we have
\[\Gamma_W(z)=\Gamma_H(z)^{\#X}\Gamma_G(z),
\]
and by a small abuse of notation the same relation on complete growth series:
\[\CG_W(z)=\Big(\prod_{i=1}^d \CG_H(z)@x_i\Big) \CG_G(z).
\]
Indeed, every element $w\in W$ may uniquely be written in the form
$w=(h_1@x_1)\cdots(h_d@x_d)g$ for some $h_1,\dots,h_d\in H,g\in G$,
and the growth series of $\{h_i@x_i\colon h_i\in H\}$ naturally
coincides with that of $H$. In particular, if $\Gamma_G(z)$ and
$\Gamma_H(z)$ are rational, then so is $\Gamma_W(z)$.

Johnson obtained in~\cite{johnson:wreath} the same conclusion for more
complicated generating sets of $W$.

\subsection{Lamplighter groups}
The next case we consider is $G=X=\Z$, and in particular ``lamplighter
groups''. Since the computations will be generalised in the next
section, we content ourselves with a brief description of the growth
series, and for simplicity assume that $H$ is a finite group. We
consider $W=H\wr\Z$, denote a generator of $\Z$ by $s$, and let $W$ be
generated by the set $\{s,s^{-1}\}\sqcup H@1$.

Consider an element $w\in W$. If its image under the natural map
$W\to\Z$ is nonnegative, then it may be written minimally in the form
$s^{-m}(h_0@1)s(h_1@1)\cdots s(h_p@1)s^{-n}$ with $h_i\in H$,
$m,n\ge0$ and $p\ge m+n$, while if its image in $\Z$ is negative, then
it may be written minimally in the form
$s^m(h_0@1)s^{-1}(h_1@1)\cdots s^{-1}(h_p@1)s^n$ with
$p>m+n$. Furthermore, $h_0$ must be non-trivial unless $m=0$, and
$h_p$ must be non-trivial unless $n=0$.

All these constraints are local and therefore rational, except the
long-range relation between $m,n,p$. However, in terms of computing
growth series, the letters in the expression
$s^{-m}(h_0@1)s(h_1@1)\cdots s(h_p@1)s^{-n}$ can be permuted at no
cost; and the set of expressions of the form
\[(h_0@1)s^{-1}s (h_1@1)\cdots s^{-1}s(h_m@1)s(h_{m+1}@1)\cdots
s(h_{p-n}@1)s s^{-1} (h_{p-n+1}@1)\cdots s s^{-1}(h_p@1)
\]
is indeed a rational language, so that its growth function is
rational.

\begin{exse}
  Compute the growth series of $C_2\wr\Z$ with the standard generators
  $C_2@1\cup\{s,s^{-1}\}$. Note that the growth function grows
  exponentially, at the same rate as Fibonacci numbers. Could you have
  guessed the appearance of Fibonacci numbers without going through
  the calculations?
\end{exse}

\subsection{Regular wreath products with free groups}
We compute in this subsection the complete growth series of a wreath
product of the form $W=H\wr G$ for $G$ a free group. In fact, we
suppose more generally that $G$ is free-like, see
Definition~\ref{defn:freelike}, so that its Cayley graph for the
generating set $S$ is an $m$-regular tree $\mathscr T$. We keep the
convention of writing $\overline s$ for $s^{-1}\in S$. We suppose as
usual that $W$ is generated by $T@1\sqcup S$.

Consider $w\in W$, written as
$w=(h_0@1)g_1(h_1@1)\cdots g_\ell(h_\ell@1)$ with $g_i\in G$ and
$h_i\in H$. Following the arguments in~\S\ref{ss:genwreath}, one may
write it as
\[w=\prod_{i=0}^\ell(h_i@e_i)\cdot g_1\cdots g_\ell,\qquad\text{with }e_i=(g_1\cdots g_i)^{-1}.
\]
The \emph{support} of $w$ is the subgraph of $\mathscr T$ traced by
inverses of prefixes of the word $g_1\dots g_\ell$; it is the convex
hull of $\{e_0,\dots,e_\ell\}$ in $\mathscr T$. We shall count
elements of $W$ by examining their possible supports and summing over
them.

For each $s\in S$, let $\Theta_s$ denote the set of finite subtrees of
$\mathscr T$, containing $1$ and no element of $S$ except possibly
$s$. Each $\theta\in\Theta_s$ has \emph{outer} vertices, with at most
one neighbour in $\theta$, and \emph{inner} vertices, with at least
two neighbours in $\theta$. We introduce non-commutative power series
$E_s(x,y,z)$ with co\"efficients in $\Z G[z]$, which count the number
of Eulerian cycles\footnote{i.e. cycles that traverse each edge once}
in trees $\Theta_s$, weighted by length in $z$; the variables $x,y$
belong to $G$ and in particular are \emph{not} assumed central. The
series $E_s(x,y,z)$ are defined by the algebraic system
\begin{equation}\label{eq:parry1}
  E_s(x,y,z)=1+s y\overline s z^2+s x\Big(\prod_{t\in S,t\neq\overline s}E_t(x,y,z)-1\Big)\overline s z^2.
\end{equation}

\[\begin{tikzpicture}
  \draw[very thin] (0,0) node[above=2mm] {1} -- +(150:1.5)
  (0,0) -- +(180:1.5)
  (0,0) -- +(210:1.5)
  (0,0) -- (1.5,0) node[above=2mm] {$s$} -- ++(30:1.5) -- ++(0:1) ++(0:-1) -- +(-20:1)
  (1.5,0) -- ++(0:1.5) -- ++(20:1) ++(20:-1) -- ++(0:1) ++(0:-1) -- ++(-20:1)
  (1.5,0) -- ++(-30:1.5) -- ++(20:1) ++(20:-1) -- ++(0:1);
  \draw[thick,rounded corners,->] (0,0.1) -- ++(1.5,0) -- ++(30:1.6) -- ++(-60:0.2) -- ++(30:-1.6) -- ++(-1.6,0);
\end{tikzpicture}\]

The monomials in $E_s$ are in bijection with (Eulerian cycles tracing)
trees in $\Theta_s$; if a tree $\theta$ with $p$ edges has inner
vertices at $f_1,\dots,f_n$ and non-trivial outer vertices at
$f'_1,\dots,f'_{n'}$, then the monomial corresponding to it is the
product, in some order, of
$z^{2p},x^{f_1},\dots,x^{f_n},y^{f'_1},\dots,y^{f'_{n'}}$. Indeed the
equation defining $E_s$ says that a monomial counted by $E_s$ is
either the empty tree (counted as $1$), or a single edge from $1$ to
$s$ (counted as $s y\overline s z^2$), or an edge from $1$ to $s$,
followed by $m-1$ subtrees counted recursively by $E_t$ for all
$t\neq\overline s$, which are not all empty. Note that, if $\theta$
has $p$ edges, then a minimal closed path that explores all vertices
of $\theta$ has length $2p$.

Let $D_s(z)$ denote the sum of $w z^{\|w\|}$ over all elements
$w\in W$ which belong to the base group $H^G$ and whose support is an
element of $\Theta_s$. Such elements may be counted as follows:
starting from a support $\theta\in\Theta_s$, choose a word
$g_1\dots g_\ell$ in $G$ of minimal length that visits all vertices of
$\theta$; and, each time a vertex is first visited, insert an element
of $H$, which furthermore must be non-trivial if the vertex is
outer. Therefore,
\begin{equation}\label{eq:parry2}
  D_s(z)=E_s(\CG_H(z),\CG_H(z)-1,z).
\end{equation}

Next, let $F_s(z)$ denote the sum of $w z^{\|w\|}$ over all elements
$w=f g\in W$ with $f\colon G\to H$, $g\in G$ not beginning in
$\overline s$, and whose support does not contain $\overline s$. We
have a linear system
\begin{equation}\label{eq:parry3}
  F_s(z)=\prod_{t\neq\overline s}D_t(z) + \sum_{t\neq\overline s}\Big(\prod_{u\neq t,\overline s}D_u(z)\Big) t z F_t,
\end{equation}
since in every such element either $g=1$ and the support explores all
the neighbours $t$ of $1$ except $\overline s$, or $g$ begins by a
generator, say $t$, and then its support explores all neighbours of
$1$ except $\overline s,t$, then moves to $t$, and continues by an
element not starting by $\overline t$.  Finally,
\begin{equation}\label{eq:parry4}
  \CG_W(z)=\prod_{s\in S}D_s(z) + \sum_{s\in S}\Big(\prod_{t\neq s}D_t(z)\Big)s z F_s(z),
\end{equation}
for the same reasoning as above. Combining
Equations~(\ref{eq:parry1}--\ref{eq:parry4}), we deduce:
\begin{thm}\label{thm:wreathalg}
  If $H$ is a finitely generated group whose complete growth series
  $\CG_H(z)$ is algebraic, and $G$ is a free-like group, then the
  complete growth series of $W$ is also algebraic.
\end{thm}

\begin{cor}[Parry,~\cite{parry:wreath}]\label{cor:parry}
  If $H$ is a finitely generated group whose growth series
  $\Gamma_H(z)$ is algebraic, and $G$ is a free-like group, then the
  growth series of $W$ is also algebraic.

  If furthermore $\Gamma_H(z)$ is rational and $m\le2$, then
  $\Gamma_W(z)$ is also rational.

  On the other hand, if $m\ge3$ then $\Gamma_W(z)$ does not belong to
  the field generated by $z$ and $\Gamma_H(z)$.
\end{cor}
\begin{proof}
  Apply the augmentation map $\varpi\colon g\mapsto1$ to
  Equations~(\ref{eq:parry1}--\ref{eq:parry4}); this gives an
  algebraic system of degree $\max(m-1,1)$ expressing $\Gamma_W(z)$ in
  terms of $z$ and $\Gamma_H(z)$. In particular, for $m=2$ it is a
  linear system.

  Conversely, assume $m\ge3$ and let $\rho$ denote the convergence
  radius of the series of the image of $D_s$ under $\varpi$. Note that
  $\lim_{z\to\rho^-}\varpi(D_s)(z)$ is finite: if the limit were
  infinite, convergence to infinity would be order $(m-1)\times$
  itself, a contradiction. Therefore $\varpi(D_s)$ has a non-pole
  singularity, so is not in $\Q(z,\Gamma_H(z))$.
\end{proof}

\begin{exse}
  Show that, for the lamplighter group $C_2\wr\Z$, the complete growth
  series is not rational.

  Hint: use Exercise~\ref{exse:quasi-geodesic combing}.
\end{exse}

\subsection{Traveling salesmen}
To glimpse at the limit of what can be computed, consider now the case
$G=\Z^2$. No property of $\Gamma_W(z)$ is known, and this is due to
the fact that there is no good description of words of minimal norm
describing group elements.

In fact, the problem can be quite precisely stated as follows. One is
given a point $p_\infty$ and a set $\{p_1,\dots,p_\ell\}$ in $\Z^2$,
and is required to find a walk of minimal length on the grid that
starts at $(0,0)$, visits all the points $p_1,\dots,p_\ell$ in some
order, and ends at $p_\infty$. This is a classical \emph{travelling
  salesman} problem, and is known to be NP-complete,
see~\cites{garey-johnson-stockmeyer:graphpb,garey-johnson:steinertree}. It
is a small step to venture that finding a good description of minimal
paths is at least as hard as finding those paths' length.

\subsection{Asymptotic growth}
Regular wreath products, in non-degenerate cases, all have exponential
growth. This is in stark contrast to the case of permutational wreath
products, as we shall see in~\S\ref{ss:ssgrowth}.

\begin{prop}
  If $H\neq 1$ and $G$ is infinite, then $W=H\wr G$ has exponential growth.
\end{prop}
\begin{proof}
  Choose $h\neq1\in H$, and without loss of generality assume that $h$
  is a generator of $H$. Since $G$ is infinite, there exists an
  infinite word $g_1g_2\dots$ that traces a geodesic in the Cayley
  graph (see after~\ref{def:schreier}) of $G$, with $g_1,g_2,\dots$
  generators of $G$ and also of $W$. In particular, all
  $(g_1\cdots g_i)^{-1}$ are distinct. Consider then, for any
  $\ell\in\N$, the set of elements
  \[\big\{(h@1)^{\epsilon_0}g_1(h@1)^{\epsilon_1}\cdots g_\ell(h@1)^{\epsilon_\ell}:\epsilon_0,\dots,\epsilon_\ell\in\{0,1\}\big\}.
  \]
  All these elements have norm at most $2\ell+1$, and there are
  $2^{\ell+1}$ such elements. They are all distinct, since when they
  are rewritten in the form $f g_1\dots g_\ell$ one has $f((g_1\cdots
  g_i)^{-1})=h^{\epsilon_i}$ so that the $\epsilon_i$ can be recovered
  from the element. Therefore, $v_W(2\ell+1)\ge2^{\ell+1}$.
\end{proof}

\section{(Self-)similar groups}\label{ss:ss}
We begin by introducing self-similar groups. They are groups with an
additional structure:
\begin{defn}
  A group $G$ is \emph{self-similar} if it is endowed with a
  homomorphism $\phi\colon G\to G\wr_X\sym(X)$ for some set $X$. The
  map $\phi$ is called the \emph{wreath recursion} of $G$.
\end{defn}
In this text, we shall always assume that the set $X$ is finite, and
shall (unless stated otherwise) also assume that the homomorphism
$\phi$ is injective. A self-similar group is a group $G$ in which
elements may be recursively described by a permutation of $X$,
decorated by elements of $G$ itself. In case $X=\{0,1,\dots,d-1\}$, we
also write elements of $G\wr_X\sym(X)$ in the form
$\pair<g_0,\dots,g_{d-1}>\pi$ for group elements $g_0,\dots,g_{d-1}$
and a permutation $\pi\in\sym(X)$.

It is essential to understand that being self-similar is an
\emph{attribute} of a group, and not a property. Thus, for example, a
topological group is a group endowed with a topology; and every group
is a topological group, for the discrete and the coarse topology. In
the same vein, every group is self-similar, merely for the reason that
it is similar to itself. Taking $X=\{0\}$ and $\phi(g)=\pair<g>$ is
uninteresting, but is not illegal.

\subsection{Finite-state self-similar groups}\label{ss:fsa}
We describe two fundamental constructions of self-similar groups.

For the first, start by a well-understood group $F$, such as a free
group; and choose a (not necessarily injective!) homomorphism
$\tilde\phi\colon F\to F\wr_X\sym(X)$.  There exists then a maximal
quotient of $F$ on which the map $\tilde\phi$ induces an injective
wreath recursion. To wit, one defines an increasing sequence $N_i$ of
normal subgroups of $F$ by
\begin{equation}\label{eq:normalF}
  N_0=1,\quad N_{i+1}=\tilde\phi^{-1}(N_i^X),
\end{equation}
and sets $G=F/\bigcup_i N_i$. By construction, the map $\tilde\phi$
induces an injective map $\phi\colon G\to G\wr_X\sym(X)$.

An important example of group defined by this method --- and which,
essentially, cannot be defined differently --- is the \emph{first
  Grigorchuk group}, introduced in~\cite{grigorchuk:burnside} and
based on~\cite{aleshin:burnside}. Consider
\begin{equation}\label{eq:F}
  F=\langle a,b,c,d\mid a^2,b^2,c^2,d^2,bcd\rangle\cong C_2*(C_2\times C_2),
\end{equation}
and define $\tilde\phi\colon F\to F\wr\sym(2)$ by
\[\tilde\phi(a)=\tikz[baseline=-2ex]{\draw[->] (0,0) -- (0.5,-0.5); \draw[->](0.5,0) -- (0,-0.5);},\qquad
\tilde\phi(b)=\tikz[baseline=-2ex]{\draw[->] (0,0) -- node[right=-2pt] {\small$a$} (0,-0.5); \draw[->](0.5,0) -- node[right=-2pt] {\small$c$} (0.5,-0.5);},\qquad
\tilde\phi(c)=\tikz[baseline=-2ex]{\draw[->] (0,0) -- node[right=-2pt] {\small$a$} (0,-0.5); \draw[->](0.5,0) -- node[right=-2pt] {\small$d$} (0.5,-0.5);},\qquad
\tilde\phi(d)=\tikz[baseline=-2ex]{\draw[->] (0,0) -- (0,-0.5); \draw[->](0.5,0) -- node[right=-2pt] {\small$b$} (0.5,-0.5);}.
\]

It is straightforward to see that $\tilde\phi$ is a homomorphism ---
just compute the images of the relators. It is, however, remarkable
that one may compute efficiently in $G$ just using this
description. We use the same letters $a,b,c,d$ for the corresponding
generators of $G$. As an illustration, let us check that the relation
$(ad)^4$ holds in $G$. Writing the permutation diagrams horizontally,
one has
\[\phi((ad)^4)=\Big(\tikz[baseline=1.5mm]{\draw[->] (0,0.5) -- +(0.7,-0.5); \draw[->] (0,0) -- +(0.7,0.5);
  \draw[->] (0.75,0) -- +(0.7,0); \draw[->] (0.75,0.5) -- node[above=-1pt] {\small$b$} +(0.7,0);
}\Big)^4
=\Big(\tikz[baseline=1.5mm]{\draw[->] (0,0.5) -- +(0.7,-0.5); \draw[->] (0,0) -- +(0.7,0.5);
  \draw[->] (0.75,0) -- +(0.7,0); \draw[->] (0.75,0.5) -- node[above=-1pt] {\small$b$} +(0.7,0);
  \draw[->] (1.5,0.5) -- +(0.7,-0.5); \draw[->] (1.5,0) -- +(0.7,0.5);
  \draw[->] (2.25,0) -- +(0.7,0); \draw[->] (2.25,0.5) -- node[above=-1pt] {\small$b$} +(0.7,0);
}\Big)^2
=\Big(\tikz[baseline=1.5mm]{\draw[->] (0,0) -- node[below=-1pt] {\small$b$} +(0.7,0); \draw[->] (0,0.5) -- node[above=-1pt] {\small$b$} +(0.7,0);}\Big)^2
=\tikz[baseline=1.5mm]{\draw[->] (0,0) -- node[below=-1pt] {\small$b^2$} +(0.7,0); \draw[->] (0,0.5) -- node[above=-1pt] {\small$b^2$} +(0.7,0);}
=1,\]
so $(ad)^4=1$ in $G$ because $\phi$ is injective.

\begin{exse}
  Using similar calculations, compute the exponent of $ab$ and $ac$ in $G$.
\end{exse}

Note that, in that example, $G=\langle S\rangle$ for the set
$S=\{1,a,b,c,d\}$ which has the property that $\phi(S)$ is contained
in $S\times S\times\sym(2)$. More generally,
\begin{defn}
  Let $G$ be a self-similar group. A subset $S\subseteq G$ is
  \emph{state-closed} if $\phi(S)$ is contained in $S^X\times\sym(X)$.

  An element $g\in G$ is \emph{finite-state} if there exists a
  state-closed subset of $G$ containing $g$. A subset of $G$ is
  \emph{finite-state} if all its elements are finite-state.
\end{defn}

\begin{exse}
  Let $\Z=\langle t\rangle$ be endowed with the self-similar structure
  $\phi(t)=\tikz[baseline=-2ex]{\draw[->] (0,0) -- node[near
    start,left] {\small$t$} (0.5,-0.5); \draw[->] (0.5,0) --
    node[near start,right=-1pt] {\small$t^2$} (0,-0.5);}$.  Show that
  only $t^0$ is finite-state.
\end{exse}

\begin{lem}
  The product and inverse of finite-state elements is again
  finite-state.
\end{lem}
\begin{proof}
  If $g,h$ are finite-state contained respectively in state-closed
  sets $S,T$, then $g h^{-1}$ is finite-state, since it belongs to the
  finite state-closed set $ST^{-1}$.
\end{proof}

Therefore, a finitely generated self-similar group $G$ is finite-state
if and only if its generators are finite-state, and one may assume
that $G$ is generated by a state-closed set.

In that case, the wreath recursion of $G$ may conveniently be
represented by an \emph{automaton}, more precisely a \emph{Mealy
  automaton}. This is a directed graph with vertex set $S$ called its
\emph{states}, and with an edge from $s\in S$ to $t\in S$, with label
`$x|y$', whenever the decorated permutation $\phi(s)$ maps $x\in X$ to
$y\in X$ and has label $t$ on the edge $x\to y$. Thus, in a sense, the
graph is the dual of the permutation diagram, with the roles of $X$
and $S$ exchanged. The automaton generating the first Grigorchuk is,
with the convention $X=\{\9,\8\}$,
\[\begin{fsa}[baseline]
    \node[state] (b) at (1.4,3) {$b$};
    \node[state] (d) at (4.2,3) {$d$};
    \node[state] (c) at (2.8,0.7) {$c$};
    \node[state] (a) at (0,0) {$a$};
    \node[state] (e) at (5.6,0) {$1$};
    \path (b) edge node {$\8|\8$} (c) edge node {$\9|\9$} (a)
          (c) edge node {$\8|\8$} (d) edge node {$\9|\9$} (a)
          (d) edge node[above] {$\8|\8$} (b) edge node {$\9|\9$} (e)
          (a) edge [double,bend right=30] node {$\9|\8,\8|\9$} (e)
          (e) edge [double,loop right] node {$\9|\9,\8|\8$} (e);
  \end{fsa}
\]

Assume that the self-similar group was obtained as above as a quotient
of a self-similar group $F$ with a map $\tilde\phi\colon F\to
F\wr_X\sym(X)$. Recall that the \emph{word problem} asks, given a word
in the generators of a finitely generated group, to determine whether
the group element that it defines is trivial. There are groups, even
finitely presented, in which the word problem is
unsolvable~\cite{novikov:wp}; however,
\begin{lem}
  Let $F$ be a finite-state finitely generated self-similar group with
  wreath recursion $\tilde\phi$ and solvable word problem, and $G$ be
  the maximal quotient of $F$ on which the induced wreath recursion
  $G\to G\wr_X\sym(X)$ is injective. Then $G$ also has solvable word
  problem.
\end{lem}
\begin{proof}
  Assume without loss of generality that $F$ is generated by the
  finite state-closed set $S$, and denote also by $S$ the
  corresponding generating set of $G$. Given a word $w\in S^*$ of
  length $\ell$, it defines a state-closed element of $F$, belonging
  to the state-closed set $S^\ell$. Consider the corresponding
  automaton with vertex set $S^\ell$.

  Let $U\subseteq S^\ell$ denote the set of states that are reachable
  from $w\in S^\ell$ by arbitrarily long paths. This set is
  computable: set $U'_0:=\{w\}$, and for $i\ge0$ set
  $U'_{i+1}:=U'_i\cup\{$endpoints of edges starting in $U'_i\}$; then
  the $U'_i$ form an increasing sequence of subsets of $S^\ell$, hence
  stabilize, say to $U''_0$. Note that $U''_0$ is the set of states
  reachable from $w$. For all $i\ge0$, let $U''_{i+1}$ denote those
  endpoints of edges starting in $U''_i$; then the $U''_i$ form a
  decreasing sequence of subsets of $S^\ell$, hence stabilize, to $U$.

  The element of $G$ defined by $w$ is trivial in $G$ if and only if
  both all the edges starting in $U''_0$ have labels of the form
  `$x|x$' for some $x\in X$, and all elements of $U$ define trivial
  elements of $F$ under the evaluation map $S^*\to F$.

  More precisely, let $m\in\N$ be minimal such that every element of
  $U$ may be reached from $w$ by a path of length at most $m$. Then
  the conditions above imply that $w$ belongs to the normal subgroup
  $N_m$ of $F$, see~\eqref{eq:normalF}.
\end{proof}

The above proof amounts to constructing a Mealy automaton for the
action of $S^\ell$, and examining it to determine which of its states
are trivial in $G$.

\subsection{Linear groups}
Here is another construction of self-similar groups. Consider a group
$G$, a subgroup $H$, and a homomorphism $\phi_0\colon H\to G$. By the
``permutational Kaloujnine-Krasner theorem''~\ref{thm:kkwreath}, there
exists a natural extension $\phi\colon G\to G\wr_X\sym(X)$, with
$X=H\backslash G$, in such a manner that
$\phi(h)=\pair<\dots,\phi_0(h),\dots>\dots$ for all $h\in H$, with the
`$\phi_0(h)$' in position $H\in H\backslash G$.

Alternatively, this map $\phi$ may be directly constructed as follows:
choose a transversal $T$ of $H$ in $G$, namely a subset $T\subseteq G$
such that every $g\in G$ may uniquely be written in the form $ht$ with
$h\in H,t\in T$. Identify $X$ with $T$. Let then $\phi(g)$ be the
decorated permutation that sends $t\in T$ to $u\in T$ with label
$\phi_0(t g u^{-1})$ whenever $t g u^{-1}$ belongs to $H$.

Here is a fundamental example: choose a prime number $p$, and consider
\[G=\Gamma_0(p)=\begin{pmatrix}\Z&\Z\\p\Z&\Z\end{pmatrix}\cap\SL_2(\Z).\]
Consider also the matrix
$\Phi=(\begin{smallmatrix}p&0\\0&p^{-1}\end{smallmatrix})\in\SL_2(\Q)$,
and $H=G\cap G^{\Phi^{-1}}$. Set $\phi_0(h)=h^\Phi$. This example
generalizes naturally to $G$ any matrix group, such as for instance a
congruence subgroup of $\SL_n(\Z)$ for arbitrary $n$, or even
$\SL_n(\Z)$ itself. This shows that the class of linear groups over
$\Z$ is contained in the class of self-similar groups.

In fact, the only essential ingredient of the above construction is
the element $\Phi$ in the \emph{commensurator} of $G$. Recall that,
for a subgroup $G$ of a group $L$, the commensurator of $G$ is the
subgroup of those $x\in L$ such that $G\cap G^x$ has finite index in
$G$ and in $G^x$. If $G$ is an irreducible lattice in a Lie group $G$,
then $G$ may be called \emph{arithmetic}~\cite{margulis:subgroups} if
its commensurator is dense in $L$; e.g.\ the commensurator of
$\SL_n(\Z)$ in $\SL_n(\R)$ is $\SL_n(\Q)$. Then all arithmetic
lattices admit self-similar actions on rooted
trees~\cite{kapovich:arithmetic}.

\subsection{Rooted trees}
Let $X$ be a set, and consider the associated \emph{rooted regular
  tree} $\mathscr T$: its vertex set is
$X^*=\{x_i\dots x_1:x_j\in X\}=\bigsqcup_{i\ge0}X^i$, and it has an
edge between $x_{i+1}x_i\dots x_1$ and $x_i\dots x_1$ for all
$x_i\in X$.  The tree is rooted at the empty word, the unique element
of $X^0$; the set of vertices at distance $i$ from the root is
identified with $X^i$, and the Cartesian product $X^\infty$ is
naturally interpreted as the boundary $\partial\mathscr T$ of the
tree, namely the set of infinite rays emanating from the root. Here
for illustration is the top of the binary tree:
\[\begin{tikzpicture}[level distance=10mm,
  every node/.style={inner sep=1pt},
  level 1/.style={sibling distance=40mm},
  level 2/.style={sibling distance=20mm},
  level 3/.style={sibling distance=10mm}]
  \node  {$\emptyset$}
  child {node {$\9$}
    child {node {$\9\9$}
      child {node {$\9\9\9$}}
      child {node {$\8\9\9$}}
    }
    child {node {$\8\9$}
      child {node {$\9\8\9$}}
      child {node {$\8\8\9$}}
    }
  }
  child {node {$\8$}
    child {node {$\9\8$}
      child {node {$\9\9\8$}}
      child {node {$\8\9\8$}}
    }
    child {node {$\8\8$}
      child {node {$\9\8\8$}}
      child {node {$\8\8\8$}}
    }
  };
\end{tikzpicture}\]
Let $W$ denote the isometry group of $\mathscr T$, namely
the set of bijections of $X^*$ that fix the root $\emptyset$ and
preserve the edge structure of $\mathscr T$. Given $g\in W$, let
$\sigma\in\sym(X)$ denote the action of $g$ on $X=X^1$, and for all
$x\in X$ define an element $g_x\in W$ by
$(x_n\cdots x_1 x)g=(x_n\cdots x_1)g_x\,(x^\sigma)$; namely, $g_x$
describes the action of $g$ on the subtree $X^*x$ as is it carried to
$X^*x^\sigma$ by $g$.
\begin{lem}
  The map
  \[\phi\colon\begin{cases}
    W &\to W\wr\sym(X)\\
    g &\mapsto\pair<g_x:x\in X>\sigma
  \end{cases}\]
  is a group isomorphism.
\end{lem}
\begin{proof}
  Given $g_x\in W$ and $\sigma\in\sym(X)$, an element $g\in W$ may
  uniquely be defined by
  $(x_n\cdots x_1 x)g=(x_n\cdots x_1)g_x\,(x^\sigma)$. This proves
  that $\phi$ is bijective.

  To see that $\phi$ is a homomorphism, consider the $\#X$ subtrees
  below the root. They are permuted according to the permutation part
  $\sigma$ of $\phi$, and simultaneously acted upon by the decorations
  $g_x$. Composition of decorated permutations therefore coincides
  with composition of tree isometries.
\end{proof}

Let us now start with a self-similar group $G$. Its wreath recursion
$\phi\colon G\to G\wr_X\sym(X)$ then defines an action of $G$ on
$X$. Furthermore, the wreath recursion can be ``iterated'': one has
maps
\begin{align*}
  G\overset\phi\longrightarrow G\wr_X\sym(X)\overset{\phi^X}\longrightarrow&(G\wr_X\sym(X))\wr_X\sym(X)\\
  &=G\wr_{X^2}(\sym(X)\wr_X\sym(X))\longrightarrow\cdots
\end{align*}
so that $G$ acts on $X^i$ for all $i\in\N$. Furthermore, these actions
are compatible with each other, in the sense that the map
$(x_{i+1},x_i,\dots,x_1)\mapsto(x_i,\dots,x_1)$ interlaces the actions
on $X^{i+1}$ and $X^i$. Taking the inverse limit of $X^i$ under these
projection maps gives an action on the Cartesian product
$X^\infty$. Note that sequences in $X^\infty$ are infinite on the
left, namely are of the form $(\dots,x_{i+1},x_i,\dots,x_1)$.

The compatibility between the actions on $X^i$ and $X^{i+1}$ precisely
means that $G$ acts by tree isometries on the rooted regular tree
$\mathscr T$ with vertex set $X^*$.

Note that, even if the wreath recursion $\phi$ is injective, the
action of the $G$ on the tree $\mathscr T$ need not be faithful. This
is, however, the case for the examples of self-similar action of the
first Grigorchuk group (see Proposition~\ref{prop:faithful,infinite})
and of the congruence subgroup $\Gamma_0(p)$.

We introduced the self-similar structure on $\Gamma_0(p)$ and not on
$\SL_2(\Z)$ because the latter does not act on the \emph{rooted
  $p$-regular tree}, in which each vertex has degree $p+1$ except the
root which has degree $p$. If we add an edge upwards from the root,
and a rooted $p$-regular tree above it, to the rooted $p$-regular
tree, we obtain a $(p+1)$-regular tree on which the action of
$\Gamma_0(p)$ extends to an action of $\SL_2(\Z)$. In fact, this
action is already well-known, see~\cite{serre:trees}*{\S II.1}: the
$(p+1)$-regular tree is the \emph{Bruhat-Tits tree} of
$\SL_2(\Z_p)$. Its vertices are homothety classes of lattices
$\cong\Z_p^2$ in $\Q_p$, and there is an edge between classes
$\Q_p^\times\Lambda$ an $\Q_p^\times\Lambda'$ if they admit
representatives $\alpha\Lambda,\alpha'\Lambda'$ with
$\alpha\Lambda\subset\alpha'\Lambda'$ and
$[\alpha'\Lambda':\alpha\Lambda]=p$. The group $\SL_2(\Q_p)$ naturally
acts on lattices, and $\SL_2(\Z_p)$ acts as the stabilizer of the root
$\Q_p^\times\Z_p^2$. The congruence subgroup $\Gamma_0(p)$ fixes an
edge adjacent to the root, and the rooted $p$-regular tree
$\mathscr T$ is spanned by those lattices of the form
$\langle (p^n,0),(x_0+x_1p+\cdots+x_{n-1}p^{n-1},1)\rangle$ for
$n\in\N$ and $x_0,\dots,x_{n-1}\in\{0,\dots,p-1\}$.

Let us remark in passing that obtaining an action on a rooted tree is
not spectacular in itself: every countable residually-$p$ group acts
on a rooted $p$-regular tree. Indeed, choose a descending sequence
$G=G_0>G_1>\cdots$ of subgroups with $[G:G_i]=p^i$ and
$\bigcap G_i=1$. Let the vertices of $\mathscr T$ be the set of right
cosets of all $G_i$, with an edge between $G_i g$ and $G_{i+1} g$ for
all $i\in\N,g\in G$; and let $G$ act by right multiplication on
$\mathscr T$.

This action is in general not self-similar, nor is it ``economical'',
in the sense that the permutation group acting on $X^i$ may have order
comparable to $(\#X)^i$ rather than $(\#X!)^{\#X^i}$.

Finally, let us return to the construction of a self-similar group $G$
as a quotient of a self-similar group $F$ so that the wreath recursion
becomes injective. Knowing that the action of $G$ is faithful helps in
solving the word problem, in case $G$ is not finite-state:
\begin{lem}
  Let $F$ be a self-similar group with wreath recursion $\tilde\phi$
  and solvable word problem, and let $G$ be the maximal quotient of
  $F$ on which the induced wreath recursion is injective. Assume that
  the action of $G$ on the tree $\mathscr T$ is faithful. Then $G$
  also has solvable word problem.
\end{lem}
\begin{proof}
  Let $S$ be a generating set for $F$, and consider $w\in S^*$. We
  start two semi-algorithms in parallel; the first one will stop if
  $w$ is non-trivial in $G$, and the second one will stop if $w$ is
  trivial in $G$.

  If $w$ is non-trivial in $G$, then it will act non-trivially on some
  vertex of $\mathscr T$, and this vertex may be found by enumerating
  all vertices of $\mathscr T$ and computing the action of $w$ on it
  by applying $\tilde\phi$.

  If $w$ is trivial, then it belongs to one of the normal subgroups
  $N_i$. Going through all $i=0,1,\dots$ in sequence, and iterating
  $i$ times $\tilde\phi$ on $w$ yields $\#X^i$ elements of $F$. If all
  of them are trivial in $F$, then $w$ is trivial; otherwise continue
  with the next $i$.
\end{proof}

Here are two fundamental examples of self-similar groups. Let
$\mathscr T=X^*$ be a rooted regular tree, and let $Q$ be a group
acting transitively on $X$. Consider the iterated wreath products
$Q_i=Q\wr_X Q\wr_X\cdots\wr_X Q$, with $i$ factors; these groups act
naturally on $X^i$ by the imprimitive action.

On the one hand, there is a natural map
$Q_{i+1}\twoheadrightarrow Q_i$, given by deleting the leftmost
factor, i.e.\ naturally mapping $Q_{i+1}=Q\wr_X Q_i$ to
$Q_i\cong1\wr_X Q_i$. Set then $\overline G=\varprojlim Q_i$, the
projective limit being taken along these epimorphisms. The
self-similarity structure
$\phi\colon \overline G\to \overline G\wr_X Q$ is induced by the
identity map $Q_i\overset\cong\longrightarrow Q_{i-1}\wr_X Q$. Since
it `peels off' the rightmost factor, it is compatible with the inverse
limit. It defines a profinite self-similar group $\overline G$.

On the other hand, there is a natural map $Q_i\hookrightarrow
Q_{i+1}$, given by inserting a trivial leftmost factor, i.e.\
naturally mapping $Q_i\cong1\wr_X Q_i$ to $Q_{i+1}=Q\wr_X Q_i$. Set then
$L=\varinjlim Q_i$, the union (= injective limit) being taken along
these monomorphisms. The self-similarity structure $\phi\colon L\to
L\wr_X Q$ is induced by the identity map
$Q_i\overset\cong\longrightarrow Q_{i-1}\wr_X Q$. Since it `peels off'
the rightmost factor, it is compatible with the union. It defines a
locally finite group $L$.

\begin{exse}
  Show that $L$ is a dense subgroup of $\overline G$.
\end{exse}

A \emph{law} for a group $G$ is a word $w(x_1,x_2,\dots)$ in variables
$x_1,x_2,\dots$ such that, whenever the elements $x_1,x_2,\dots$ are
replaced by group elements from $G$, the word evaluates to $1$ in
$G$. For example, abelian groups are characterised as those groups
satisfying the law $w=[x_1,x_2]$.
\begin{exse}\label{exse:nolaw}
  Show that $L$ satisfies no non-trivial law.

  Hint: it suffices to look at the case $X=\{1,2\}$. By
  Theorem~\ref{thm:kk}, every finite $2$-group imbeds in $Q_i$ for
  some $i\in\N$, and therefore in $G$. Finally, the free group is
  residually $2$.

  See~\cite{abert:nonfree} for a general result about inexistence of
  group laws, which covers the group $L$.
\end{exse}

\subsection{Similar families of groups}\label{ss:ssfam}
The notion of self-similar group may be generalised to a \emph{family}
of similar groups.
\begin{defn}
  Let $\Omega$ be a set, and let $\sigma\colon\Omega\righttoleftarrow$ be a
  map. A \emph{similar family of groups} over $\Omega$ is a family
  $(G_\omega)_{\omega\in\Omega}$ of groups and a family of homomorphisms
  \[\phi_\omega\colon G_\omega\to G_{\sigma\omega}\wr_{X_\omega}\sym(X_\omega),\]
  for a family of sets $(X_\omega)_{\omega\in\Omega}$.
\end{defn}

Just as before, each group $G_\omega$ acts on a tree $\mathscr
T_\omega$ with vertex set
$\bigsqcup_{i\ge0}X_{\sigma^{i-1}\omega}\times\cdots\times
X_\omega$. This rooted tree is now not anymore regular, but it is
still \emph{spherically homogeneous}, in that its isometry group is
transitive on the set of vertices at given distance from the root.

As before, there are two fundamental examples of similar families of
groups. Let $(X_\omega)_{\omega\in\Omega}$ be a family of sets, and
let $(Q_\omega)_{\omega\in\Omega}$ be a family of groups, with
$Q_\omega$ acting on $X_\omega$, say transitively for simplicity.

Consider the iterated wreath products
$Q_{\omega,i}=Q_{\sigma^{i-1}\omega}\wr_{X_{\sigma^{i-2}\omega}}\cdots\wr_{X_\omega}
Q_\omega$, with $i$ factors; these groups act naturally on
$X_{\sigma^{i-1}\omega}\times\cdots\times X_\omega$ by the imprimitive
action.

On the one hand, there is a natural map
$Q_{\omega,i+1}\twoheadrightarrow Q_{\omega,i}$, given by deleting the
leftmost factor, i.e.\ naturally mapping
$Q_{\omega,i+1}=Q_{\sigma^i\omega}\wr_{X_{\sigma^{i-1}\omega}}
Q_{\omega,i}$ to $Q_{\omega,i}\cong1\wr_{X_{\sigma^{i-1}\omega}}
Q_{\omega,i}$. Set then $\overline G_\omega=\varprojlim Q_{\omega,i}$,
the projective limit being taken along these epimorphisms. The
self-similarity structure $\phi_\omega\colon \overline G_\omega\to
\overline G_{\sigma\omega}\wr_{X_\omega} Q_\omega$ is induced by the
identity map $Q_{\sigma,i}\overset\cong\longrightarrow
Q_{\sigma\omega,i-1}\wr_{X_\omega} Q_\omega$. Since it `peels off' the
rightmost factor, it is compatible with the inverse limit. It defines
a profinite self-similar group $\overline G_\omega$.

On the other hand, there is a natural map $Q_{\omega,i}\hookrightarrow
Q_{\omega,i+1}$, given by inserting a trivial leftmost factor, i.e.\
naturally mapping $Q_{\omega,i}\cong1\wr_{X_{\sigma^{i-1}\omega}}
Q_{\omega,i}$ to
$Q_{\omega,i+1}=Q_{\sigma^i\omega}\wr_{X_{\sigma^{i-1}\omega}}
Q_{\omega,i}$. Set then $L_\omega=\varinjlim Q_{\omega,i}$, the
injective limit being taken along these monomorphisms. The
self-similarity structure $\phi\omega\colon L_\omega\to
L_{\sigma\omega}\wr_{X_\omega} Q_\omega$ is induced by the identity
map $Q_{\sigma,i}\overset\cong\longrightarrow
Q_{\sigma\omega,i-1}\wr_{X_\omega} Q_\omega$. Since it `peels off' the
rightmost factor, it is compatible with the union. It defines a
locally finite self-similar group $L_\omega$.

\subsection{\boldmath The Grigorchuk family $G_\omega$}\label{ss:grigorchuk}
We shall concentrate particularly on one specific example. Write
$\{0,1,2\}$ for the three non-trivial homomorphisms $C_2\times C_2\to
C_2$, identified for definiteness as follows. We view the source
$C_2\times C_2=\{1,b,c,d\}$ as a subgroup of the group $F$ given
in~\eqref{eq:F}, and the range $C_2=\{1,a\}$ in that same group
$F$. The three homomorphisms are then uniquely defined by
$\ker(0)=\langle b\rangle$ and $\ker(1)=\langle c\rangle$ and
$\ker(2)=\langle d\rangle$. Set then
\[\Omega=\{0,1,2\}^\infty,\qquad\sigma(\omega_0\omega_1\omega_2\dots)=\omega_1\omega_2\dots.\]

\noindent We start by the similar family $(F)_{\omega\in\Omega}$ with
maps $\phi_\omega\colon F\to F\wr\sym(2)$ given by
\[\tilde\phi_\omega(a)=\tikz[baseline=-2ex]{\draw[->] (0,0) -- (0.5,-0.5); \draw[->](0.5,0) -- (0,-0.5);},\quad\text{and for all }x\in\{b,c,d\}:\quad
\tilde\phi_\omega(x)=\tikz[baseline=-2ex]{\draw[->] (0,0) -- node[right=-2pt] {\small$\omega_0(x)$} +(0,-0.5); \draw[->](1.0,0) -- node[right=-2pt] {\small$x$} +(0,-0.5);},
\]
we define normal subgroups $(N_{\omega,i})_{i\in\N,\omega\in\Omega}$ of
$F$ by $N_{\omega,0}=1$ and
$N_{\omega,i+1}=\tilde\phi_\omega^{-1}(N_{\omega,i}^2)$, and set
$G_\omega=F/\bigcup_{i\in\N}N_{\omega,i}$. This is the same
construction as above, and computes $G_\omega$ as the maximal quotient
of $F$ such that the maps $\tilde\phi_\omega\colon F\to F\wr\sym(2)$
descend to injective maps $\phi_\omega\colon G_\omega\to
G_{\sigma\omega}\wr\sym(2)$.

In particular, letting $G$ denote the Grigorchuk group defined
in~\S\ref{ss:fsa}, we have isomorphisms
\begin{align*}
  &\begin{tikzpicture}
    \path [use as bounding box,red] (-2,0) rectangle (8,0);
    \node (g) at (0,0) {$G$};
    \node (g012) at (2,0) {$G_{(012)^\infty}$};
    \node (g120) at (4,0) {$G_{(120)^\infty}$};
    \node (g201) at (6,0) {$G_{(201)^\infty}$};
    \draw[double equal sign distance] (g) -- node[above] {$\sim$} (g012);
    \draw[double equal sign distance] (g012) -- node[above] {$\sim$} (g120);
    \draw[double equal sign distance] (g120) -- node[above] {$\sim$} (g201);
  \end{tikzpicture}
  \intertext{identifying the generating sets as follows:}
  &\begin{tikzpicture}
    \path [use as bounding box,red] (-2,0) rectangle (8,0);
    \node (s) at (0,0) {$\{a,b,c,d\}$};
    \node (s012) at (2,0) {$\{a,d,c,b\}$};
    \node (s120) at (4,0) {$\{a,b,d,c\}$};
    \node (s201) at (6,0) {$\{a,c,b,d\}.$};
  \end{tikzpicture}
\end{align*}

The groups $G_\omega$ all act on the binary rooted tree, and it is
easy to see that the orbit of the ray $\8^\infty$ is dense. Therefore,
the groups $G_\omega$ could just as well have been defined by their
actions on their respective orbit $\8^\infty G_\omega$. These are
naturally graphs, called \emph{Schreier graphs}, with vertex set
$\8^\infty G_\omega$, and with an edge from $\8^\infty g$ to
$\8^\infty gs$ for each generator $s\in\{a,b,c,d\}$, see
Definition~\ref{def:schreier}.

\begin{prop}
  The graph $\8^\infty G_\omega$ is isometric to the half-infinite
  line $\N$ with multiple edges and loops. Under this identification
  with $\N$, the action of $G_\omega$ is given by
  \begin{xalignat*}{2}
    a(2j)&=2j+1, & a(2j+1)&=2j,\\
    \text{and for all }x\in\{b,c,d\},\quad x(0)&=0, & x(2^i(2j+1))&=2^i(2\omega_i(j)+1).
  \end{xalignat*}
\end{prop}
\begin{proof}
  Consider the infinite dihedral group $D=\langle a,x\mid
  a^2,x^2\rangle$, and the wreath recursion $\phi\colon D\to
  D\wr\sym(2)$ defined by
  \[\phi(a)=\tikz[baseline=-2ex]{\draw[->] (0,0) -- (0.5,-0.5); \draw[->](0.5,0) -- (0,-0.5);},\qquad
  \tilde\phi(x)=\tikz[baseline=-2ex]{\draw[->] (0,0) --
    node[right=-2pt] {\small$a$} (0,-0.5); \draw[->](0.5,0) --
    node[right=-2pt] {\small$x$} (0.5,-0.5);}.
  \]
  It defines a faithful action of $D$ on the binary rooted tree
  $\mathscr T$ with vertex set $\{\9,\8\}^*$, and the action on the
  ray $\8^\infty$ is isomorphic to the action on the set of cosets
  $\langle x\rangle\backslash D$, since the stabilizer in $D$ of
  $\8^\infty$ is $\langle x\rangle$. We abbreviate
  $\8^\infty=:\overline\8$. The Schreier graph of the latter is a
  half-infinite line
  \[\begin{tikzpicture}
    \footnotesize
    \path (0,0) edge[loop left] node {$x$} ();
    \foreach\i/\l in {0/\overline\8,1/\overline\8\9,2/\overline\8\9\9,3/\overline\8\9\8,4/\overline\8\9\9\8,5/\overline\8\9\9\9,6/\overline\8\9\8\9,7/\overline\8\9\8\8,8/\overline\8\9\9\8\8,9/\overline\8\9\9\8\9,10/\overline\8\9\9\9\9} \fill (\i,0) circle [radius=2pt] node[rotate=-45,right=1mm] {$\l$};
    \foreach\i in {0,2,...,8} \path (\i,0) edge node[above] {$a$} (\i+1,0);
    \foreach\i in {1,3,...,9} \path (\i,0) edge node[above] {$x$} (\i+1,0);
    \path (10,0) edge[densely dotted] (10.5,0);
  \end{tikzpicture}
  \]
  Since the action of generators of $D$ change only a single symbol on
  sequences in $\{\9,\8\}^\infty$, the identification of the Schreier
  graph's vertices with $\N$ is explicit: it is the ``Gray
  code''~\cite{gray:pulse} enumeration starting from the left-infinite
  word $\overline\8$. Thus the sequence $\dots\8\8 x_i\dots x_1$ is
  identified with the integer $\sum_{j=1}^i(1-x_j)2^{j-1}$, reading
  the number in base $2$ with $0$'s and $1$'s switched.

  Now, to obtain the Schreier graph of $G_\omega$, one replaces each
  `$x$' edge by a pair of edges labeled by two letters out of
  $\{b,c,d\}$, and puts loops at the extremities of the edge labeled
  by the remaining letter.  The choice of which letter becomes a loop
  is determined by the position of the edge on the graph and the
  sequence $\omega$.
\end{proof}

For example, here is the Schreier graph of the action of the
first Grigorchuk group $G_{012}=\langle a,b,c,d\rangle$ on $\8^\infty G$:
\begin{equation}\label{eq:grigorbit}
  \begin{tikzpicture}[xscale=0.68,baseline]
    \footnotesize
    \path (0,0) edge[loop above] node {$b$} ();
    \path (0,0) edge[loop left] node {$c$} ();
    \path (0,0) edge[loop below] node {$d$} ();
    \foreach\i in {0,2,...,14} \path (\i,0) edge node[above] {$a$} (\i+1,0);
    \foreach\i/\g in {1/b,3/d,5/b,7/c,9/b,11/d,13/b,15/b} \path (\i,0) edge[bend left] node[above] {$\g$} (\i+1,0);
    \foreach\i/\g in {1/c,3/b,5/c,7/d,9/c,11/b,13/c,15/c} \path (\i,0) edge[bend right] node[below] {$\g$} (\i+1,0);
    \foreach\i/\g in {1/d,2/d,3/c,4/c,5/d,6/d,7/b,8/b,9/d,10/d,11/c,12/c,13/d,14/d,15/d,16/d} \path (\i,0) edge[loop above] node[above] {$\g$} ();
    \path (16,0) edge[densely dotted] (16.5,0);
  \end{tikzpicture}
\end{equation}

\section{Growth estimates for self-similar groups}\label{ss:ssgrowth}
One of the purposes of this section is to reprove the following
result. Let $\Omega'$ denote the subset of $\Omega$ consisting of
sequences containing infinitely many of each of the symbols $0,1,2$.
\begin{thm}[Grigorchuk,~\cite{grigorchuk:gdegree}]\label{thm:grigorchuk}
  If $\omega\in\Omega'$, namely if $\omega$ contains infinitely many of
  each of the symbols $0,1,2$, then $G_\omega$ has intermediate
  word growth.
\end{thm}
We shall in fact prove much more, in preparation for the
construction, in~\S\ref{ss:perm}, of groups with prescribed growth. We
mainly follow~\cite{bartholdi-erschler:givengrowth}.

\subsection{A lower bound via algebras}\label{ss:algebralb}
We begin by a general lower bound on growth, coming from the theory of
Hopf algebras. Recall that the lower central series of a group $G$ is
defined by $\gamma_1(G)=G$ and $\gamma_{n+1}(G)=[\gamma_n(G),G]$ for
all $n\ge1$.
\begin{thm}[Grigorchuk,~\cite{grigorchuk:hp}]\label{thm:growth:hp}
  Let $G$ be a finitely generated group, and assume that there is a
  subgroup $H<G$ such that $\gamma_n(H)\neq\gamma_{n+1}(H)$ for all
  $n\in\N$. Then $G$'s growth function satisfies
  \[\gamma_G\succsim \exp(\sqrt R).\]

  In particular, if $G$ is residually virtually nilpotent, then either
  $G$ is virtually nilpotent (in which case $\gamma_G$ is polynomial) or
  $\gamma_G\succsim \exp(\sqrt R)$.
\end{thm}

Before embarking on the proof, let us set up some algebraic notions.
Let $\mathbb K$ be a field, and let $G$ be a group. The group ring
$\mathscr A=\mathbb K G$ is the $\mathbb K$-vector space with basis
$G$, and multiplication extended linearly. It is a \emph{Hopf
  algebra}: it admits a \emph{coproduct}, which is an algebra
homomorphism $\Delta\colon\mathscr A\to\mathscr A\otimes\mathscr A$
defined on the basis $G$ by $g\mapsto g\otimes g$, a \emph{counit},
which is an algebra homomorphism
$\varepsilon\colon\mathscr A\twoheadrightarrow\mathbb K$ defined on
the basis by $g\mapsto1$; and an \emph{antipode}, which is an
antihomomorphism $\sigma\colon\mathscr A\to\mathscr A$ defined on the
basis by $g\mapsto g^{-1}$. Various axioms are satisfied, in
particular the coproduct is \emph{coassociative}:
$(1\otimes\Delta)\circ\Delta=(\Delta\otimes1)\circ\Delta\colon\mathscr
A\to(\mathscr A)^{\otimes3}$,
and \emph{cocommutative}: $\Delta=\tau\circ\Delta$, for
$\tau\colon\mathscr A^{\otimes2}\to\mathscr A^{\otimes2}$ the map
$x\otimes y\mapsto y\otimes x$ flipping both
factors. See~\cite{sweedler:ha} for details.

Denote by $\varpi$ the kernel of $\varepsilon$, called the
\emph{augmentation ideal}.  The \emph{associated graded} of $\mathscr
A$ is the vector space
\[\overline{\mathscr A}=\bigoplus_{n\ge0}\varpi^n/\varpi^{n+1}.\]
\begin{lem}
  The associated graded $\overline{\mathscr A}$ is a graded,
  cocommutative Hopf algebra.
\end{lem}
\begin{proof}
  We have
  $\Delta(\varpi)\le\mathscr A\otimes\varpi+\varpi\otimes\mathscr A$;
  so $\Delta(\varpi^n)\le\sum_{i=0}^n\varpi^i\otimes\varpi^{n-i}$. Now
  given $\overline x\in\varpi^n/\varpi^{n+1}$, choose $x\in\varpi^n$
  representing it; write
  $\Delta(x)=\sum\sum_{i=0}^n y_i\otimes z_{n-i}$ with
  $y_i,z_i\in\varpi^i$; and set
  $\Delta(\overline x)=\sum\sum_{i=0}^n
  \overline{y_i}\otimes\overline{z_{n-i}}$
  where $\overline{y_i},\overline{z_i}\in\varpi^i/\varpi^{i+1}$ are
  the images of their respective representatives. It is easy to show
  that this definition does not depend on the choices of $x,y_i,z_i$;
  and it defines a coassociative and cocommutative coproduct since
  $\mathscr A$'s coproduct was already coassociative and
  cocommutative.
\end{proof}

Let $H$ be a Hopf algebra. An element $x\in H$ is called
\emph{primitive} if $\Delta(x)=x\otimes 1+1\otimes x$. The set of
primitive elements in a Hopf algebra forms a Lie subalgebra of $H$ for
the usual bracket $[x,y]=x y-y x$. Conversely, if $L$ is a Lie algebra,
then its universal enveloping algebra is a Hopf algebra whose
primitive elements are $L$.  Note that, in characteristic $p$, one
should consider \emph{restricted} Lie algebras.
\begin{prop}[\cite{milnor-moore:hopf}*{Theorem~6.11}]\label{prop:mm}
  Let $H$ be a cocommutative, primitively generated, graded Hopf
  algebra. Then it is the universal enveloping algebra of its
  primitive elements.
\end{prop}
\begin{proof}
  Let $P$ denote the Lie algebra of primitive elements in $H$. By the
  universal property of $U(P)$, there is a map $f:U(P)\to H$ which is
  graded, and surjective because $P$ generates $H$. We show that $f$
  is also injective. Consider a homogeneous element $x\in U(P)$, say
  of degree $n$. If $n=0$, then $x\in\ker(f)$ if and only if
  $x=0$. Assume then that $f$ is injective on elements of degree $<n$.

  We have $\Delta(x)=1\otimes x+x\otimes 1+y$ for some $y\in
  U(P)_{<n}\otimes U(P)_{<n}$. If $f(x)=0$, then $f(y)=0$; but we had
  assumed $f$ to be injective on elements of degree $<n$, so $y=0$ and
  $x\in P$. By assumption, $f$ is injective on $P$, so $x=0$ and
  therefore $f$ is injective on elements of degree $n$ as well.
\end{proof}

We apply these considerations to $\mathscr A=\mathbb K G$. First, we
identify the primitive elements in $\overline{\mathbb K G}$. Let us
define the series $\gamma^{\mathbb K}_n(G)=\{g\in G\mid
g-1\in\varpi^n\}$ of normal subgroups of $G$.
\begin{prop}
  The space of primitive elements in $\overline{\mathbb K G}$ is
  \begin{equation}\label{eq:liealg}
    \mathscr L^{\mathbb K}(G)=\bigoplus_{n\ge1}(\gamma^{\mathbb K}_n(G)/\gamma^{\mathbb K}_{n+1}(G))\otimes\mathbb K.
  \end{equation}
\end{prop}
\begin{proof}
  The natural map $g\mapsto g-1$ from $\gamma^{\mathbb K}_n(G)$ to
  $\varpi^n/\varpi^{n+1}$ extends to a Lie algebra isomorphism from
  $\gamma^{\mathbb K}_n(G)$ onto primitive elements of
  $\overline{\mathbb K G}$.
\end{proof}
In fact, $\gamma^{\mathbb K}_n(G)$ only depends on the characteristic
$p$ of $\mathbb K$ and, up to extension of scalars,
$\mathscr L^{\mathbb K}(G)$ depends only on $p$. Furthermore, it may
be identified directly within $G$, and is a variant of the lower
central series~\cites{jennings:gpring,jennings:gpringnilp}. Indeed one
has $\gamma^p_1(G)=G$, and if $p=0$ then
\[\gamma_{n+1}^0(G)=\langle g\in G\mid g^t\in[G,\gamma_n^0(G)]\text{ for some }t>0\rangle,\]
while if $p>0$ then
\[\gamma_{n+1}^p(G)=\langle[G,\gamma_n^p(G)]\{x^p\mid x\in\gamma_{\lceil n/p\rceil}(G)\}.\]

\begin{prop}\label{prop:alg-gp}
  Let $S$ be such that $S\cup S^{-1}$ generates $G$, and let
  $v_{G,S}$ denote the corresponding growth function. Let
  $w(n)=\dim_{\mathbb K}\varpi^n/\varpi^{n+1}$ denote the growth
  function of $\overline{\mathbb K G}$. Then
  \[v_{G,S}(n)\ge w(0)+w(1)+\dots+w(n)\text{ for all }n\in\N.\]
\end{prop}
\begin{proof}
  Consider first an element $x=1-g\in\varpi$, and write
  $g=s_1^{\epsilon_1}\cdots s_\ell^{\epsilon_\ell}$ as a product of
  generators and inverses. Using the identities
  \begin{align*}
    1-g h &=(1-g)+(1-h)-(1-g)(1-h),\\
    1-g^{-1}&=-(1-g)+(1-g)(1-g^{-1}),
  \end{align*}
  we get $x\equiv\sum_{i=1}^\ell\epsilon_i(1-s_i)$ modulo $\varpi^2$.

  The ideal $\varpi^n$ is generated, qua ideal, by all
  $x=(1-g_1)\cdots(1-g_n)$ with $g_i\in G$. By the above,
  $\varpi^n/\varpi^{n+1}$ is generated, again qua ideal, by all
  $x=(1-s_1)\cdots(1-s_n)$ with $s_i\in S$. Now
  $\varpi^n/\varpi^{n+1}$ has trivial multiplication, so the
  $(1-s_1)\cdots(1-s_n)$ also generate $\varpi^n/\varpi^{n+1}$ as a
  vector space.

  This generating set is contained in the linear span of
  $B_{G,S}(n)\subseteq\mathbb K G$; so for $0\le i\le n$ we have
  $w(i)\le\dim(\varpi^i\cap\mathbb
  KB_{G,S}(n))-\dim(\varpi^{i+1}\cap\mathbb KB_{G,S}(n))$ and the
  claim follows.
\end{proof}

\begin{proof}[Proof of Theorem~\ref{thm:growth:hp}]
  Consider a subgroup $H$ such that $\gamma_n(H)\neq\gamma_{n+1}(H)$
  for all $n\in\N$. Without loss of generality, suppose $H$ is
  finitely generated. For $n\in\N$, let $\mathcal P(n)$ be the set of
  prime numbers $p$ such that
  $(\gamma_n(H)/\gamma_{n+1}(H))\otimes\mathbb F_p\neq0$. Each
  $\mathcal P(n)$ is non-empty because $\gamma_n(H)/\gamma_{n+1}(H)$
  is a finitely generated abelian group, and
  $\mathcal P(n+1)\subseteq\mathcal P(n)$ because the commutator map
  $(\gamma_n(H)/\gamma_{n+1}(H))\times
  H\to\gamma_{n+1}(H)/\gamma_{n+2}(H)$
  is onto. There exists therefore a prime number $p$ such that
  $\gamma_n(H)/\gamma_{n+1}(H)\otimes\mathbb F_p\neq0$ for all
  $n\in\N$.  In particular, $\gamma_n^p(H)\neq\gamma_{n+1}^p(H)$ for
  all $n\in\N$.

  Let $\mathbb K$ be a field of characteristic $p$. It follows that
  $\overline{\mathbb K H}$ contains a primitive element $x_n$ of
  degree $n$ for all $n\in\N$, namely $x_n=g_n-1$ for some
  $g_n\in\gamma_n^p(H)\setminus\gamma_{n+1}^p(H)$; so
  $\overline{\mathbb K H}$ contains $\pi(n)$ linearly independent
  elements of degree $n$, where $\pi(n)$ denotes the partition
  function. Indeed to every partition $n=i_1+\dots+i_k$ with
  $i_1\le\cdots\le i_k$ associate the element $x_{i_1}\cdots x_{i_k}$;
  these elements are linearly independent by
  Proposition~\ref{prop:mm}. Therefore, the growth function of
  $\overline{\mathbb K H}$ is at least $\pi(n)$.

  It now follows from Proposition~\ref{prop:alg-gp} that
  $v_{H,S}(R)\ge \pi(R)$ holds for any generating set $S$ of $H$.
  Classical results on partitions~\cite{hardy-r:asymptotic} tell us that
  $\pi(n)\propto \exp(\sqrt n)$; so \emph{a fortiori}
  $v_G(R)\succsim \exp(\sqrt R)$.
\end{proof}

Note that, for Grigorchuk's group $G=G_{012}$, the quotients
$\gamma_n(G)/\gamma_{n+1}(G)$ have bounded rank, in fact $1$ or $2$,
see~\cite{rozhkov:lcs}; so that no improvement on the lower bound can
be obtained using Proposition~\ref{prop:alg-gp}.

\subsection{\boldmath Metrics on $G_\omega$}
We already saw that the groups $G_\omega$ are contracting, namely if
$\phi(g)=\pair<g_0,g_1>\pi$ then $g_0$ and $g_1$ are shorter than
$g$. We shall need a strengthening of this property: we assign norms
$\|\cdot\|_\omega$ to the groups $G_\omega$ to obtain relations of the
form
\begin{equation}\label{eq:supercontract}
  \|g_0\|_{\sigma\omega}+\|g_1\|_{\sigma\omega}\le\frac2{\eta_\omega}\big(\|g\|_\omega+\|a\|_\omega\big)
\end{equation}
with $\eta_\omega>2$ as large as possible.

We do this by assigning norms $\in\R_+$ to the generators $a,b,c,d$ of
$G_\omega$, and extend $\|\cdot\|_\omega$ to $G_\omega$ by the
triangular inequality:
\[\|g\|_\omega=\min\{\|s_1\|+\cdots+\|s_n\|:g=s_1\cdots s_n,\;s_i\in S\}.
\]
For this purpose, consider the open $2$-simplex
\[\Delta=\{(\beta,\gamma,\delta)\in\R^3:\max\{\beta,\gamma,\delta\}<\frac12,\beta+\gamma+\delta=1\}.
\]
Its extremal points are $(\frac12,\frac12,0)$ and its permutations. A
choice of $p_\omega=(\beta,\gamma,\delta)\in\Delta$ defines the
following norm on the generators of $G_\omega$:
\[\|a\|_\omega=1-2\max\{\beta,\gamma,\delta\},\qquad\|b\|_\omega=\beta-\|a\|_\omega,\;\|c\|_\omega=\gamma-\|a\|_\omega,\;\|d\|_\omega=\delta-\|a\|_\omega.
\]
In particular, note that the triangular inequality
$\|x\|_\omega+\|y\|_\omega\le\|xy\|_\omega$ is sharp for $x,y\in\{b,c,d\}$.
\[\begin{tikzpicture}
  \coordinate (A0) at (9,4);
  \coordinate (A1) at (3,-4);
  \coordinate (A2) at (-3,4);
  \coordinate (M01) at ($(A0)!0.5!(A1)$);
  \coordinate (M02) at ($(A0)!0.5!(A2)$);
  \coordinate (M12) at ($(A1)!0.5!(A2)$);
  \coordinate (C012) at ($(A0)!0.6667!(M12)$);
  \draw[thick] (M12) node[below left] {$(0,\frac12,\frac12)$} --
  (M01) node[below right] {$(\frac12,\frac12,0)$} --
  (M02) node[above] {$(\frac12,0,\frac12)$} -- cycle;
  \node[text width=3.5cm] at (8,2) {The simplex $\Delta$, and a 3-cycle preserved by $\overline M_2\overline M_1\overline M_0$};
  \draw[thin] (M01) -- (C012);
  \draw[thin] (M02) -- (C012);
  \draw[thin] (M12) -- (C012);
  \node at (4.0,1.4) {$\overline M_0(\Delta)$};
  \node at (2.8,0.4) {$\overline M_1(\Delta)$};
  \node at (2.2,1.8) {$\overline M_2(\Delta)$};
  \node at ($0.4053*(A0)+0.3285*(A2)+0.2662*(A1)$) {$\bullet$};
  \node at ($0.4053*(A1)+0.3285*(A0)+0.2662*(A2)$) {$\bullet$};
  \node at ($0.4053*(A2)+0.3285*(A1)+0.2662*(A0)$) {$\bullet$};
\end{tikzpicture}\]

We extend the family of groups $(G_\omega)_{\omega\in\{0,1,2\}^\N}$ to
a family $(G_\omega)_{\omega\in\{0,1,2\}^\Z}$; namely, the parameter
space is now $\Omega=\{0,1,2\}^\Z$ with the two-sided shift map
$\sigma\colon\Omega\righttoleftarrow$. The group $G_\omega$ itself only
depends on the restriction of $\omega$ to $\N$, but the norm on
$G_\omega$ depends on the restriction of $\omega$ to $-\N$.

Analogously to before, we denote by $\Omega'$ the subset of $\Omega$
consisting of sequences that contain infinitely many $0,1,2$ in both
directions. For $\omega\in\Omega'$, we construct the point
$p_\omega\in\Delta$ as follows. Consider the matrices
\[M_0=\begin{pmatrix}1 & 1 & 1\\ 0 & 2 & 0\\ 0 & 0 & 2\end{pmatrix},\qquad
M_1=\begin{pmatrix}2 & 0 & 0\\ 1 & 1 & 1\\ 0 & 0 & 2\end{pmatrix},\qquad
M_2=\begin{pmatrix}2 & 0 & 0\\ 0 & 2 & 0\\ 1 & 1 & 1\end{pmatrix}.\]
Define then
\[\eta\colon\Delta\times\{0,1,2\}\to(2,3],\qquad\overline M_\lambda\colon\Delta\righttoleftarrow,\qquad\mu\colon\Delta\to(0,\tfrac13)\]
by setting, for all $\lambda\in\{0,1,2\}$ and all $p\in\Delta$,
\begin{align*}
  \eta(p,\lambda)&=(1\;1\;1)\cdot M_\lambda(p)&&\text{the $\ell^1$-norm of }M_\lambda(p),\\
  \overline M_\lambda(p)&=\frac{M_\lambda(p)}{\eta(p,\lambda)}&&\text{the projection of }M_\lambda(p)\text{ to }\Delta,\\
  \mu(p)&=\min\{\beta,\gamma,\delta\}&&\text{the minimal distance of $p$ to a vertex of }\Delta.
\end{align*}

Endow $\Delta$ with the Hilbert metric
$d_\Delta(V_1,V_2)=\log(V_1,V_2;V_-,V_+)$, computed using the
cross-ratio of the points $V_1,V_2$ and the intersections $V_-,V_+$ of
the line containing $V_1,V_2$ with the boundary of $\Delta$. The
transformations $\overline M_\lambda$ are projective transformations
of $\Delta$, and are therefore contracting:

\begin{lem}[Essentially~\cite{birkhoff:jentzsch}]\label{lem:birkhoff}
  Let $K$ be a convex subset of affine space, and let $A\colon
  K\righttoleftarrow$ be a projective map. Then $A$ contracts the
  Hilbert metric.

  If furthermore $A(K)$ contains no lines from $K$ (that is, $A(K)\cap
  \ell\neq K\cap \ell$ for every line $\ell$ intersecting $K$), then
  $A$ is strictly contracting.
\end{lem}
\begin{proof}
  Since $A$ is projective, it preserves the cross-ratio on lines, so
  we have $d_{A(K)}(A(V_1),A(V_2))=d_K(V_1,V_2)$ for all
  $V_1,V_2\in K$. Furthermore, on the line $\ell$ through $V_1,V_2$,
  the intersection points $\ell\cap\partial A(K)$ are not further from
  $V_1,V_2$ than $\{V_+,V_-\}=\ell\cap\partial K$; the Hilbert metric
  decreases as $V_\pm$ are moved further apart from $V_1,V_2$, and
  this gives strict contraction under the condition
  $A(K)\cap \ell\neq K\cap \ell$.
\end{proof}

We are ready to define the points $p_\omega\in\Delta$, and therefore
the metrics $\|\cdot\|_\omega$. Choose an arbitrary point
$p\in\Delta$, and set
\begin{equation}\label{eq:pomega}
  p_\omega=\lim_{n\to\infty} \overline M_{\omega_{-1}}\circ\overline
  M_{\omega_{-2}}\circ\cdots\circ \overline M_{\omega_{-n}}(p).
\end{equation}
Note that, since the transformations $\overline M_\lambda$ are
contracting, the limit $p_\omega$ is independent of the choice of
$p$. Note also that, by our assumption that the negative part of
$\omega$ contains infinitely many $0$, $1$ and $2$, the limit
$p_\omega$ does not belong to the boundary of $\Delta$. From now on, we write
\[\eta_\omega:=\eta(p_\omega,\omega_0).\]

\begin{lem}\label{lem:supercontract}
  For all $g\in G_\omega$ with $\phi(g)=\pair<g_0,g_1>\pi$, we have
  the inequality~\eqref{eq:supercontract}
  \[\|g_0\|_{\sigma\omega}+\|g_1\|_{\sigma\omega}\le\frac2{\eta_\omega}\big(\|g\|_\omega+\|a\|_\omega\big);\]
  and furthermore, if $g\not\in\{b,c,d\}$ then up to replacing $g$
  with a conjugate we have
  \[\|g_0\|_{\sigma\omega}+\|g_1\|_{\sigma\omega}\le\frac2{\eta_\omega}\|g\|_\omega.\]
\end{lem}
\begin{proof}
  Without loss of generality, we suppose $\omega_0=0$, and write
  $p_\omega=(\beta,\gamma,\delta)$ and $p_{\sigma\omega}=\overline
  M_{\omega_0}p_\omega=(\beta',\gamma',\delta')$. We have
  \[\eta_\omega=3-2\beta,\quad(\beta',\gamma',\delta')=\frac{(\beta+\gamma+\delta,2\gamma,2\delta)}{\eta_\omega}=\frac{(1,2\gamma,2\delta)}{\eta_\omega}.\]
  Thus
  \begin{align*}
    \|b\|_{\sigma\omega}+\|\omega_0(a)\|_{\sigma\omega} &= \|b\|_{\sigma\omega} = \beta'-\|a\|_{\sigma\omega} = \beta'-(1-2\beta')=3\beta'-1\\
    & =\frac3{\eta_\omega}-1=2\frac\beta{\eta_\omega}=\frac2{\eta_\omega}\|ab\|_\omega,\\
    \|c\|_{\sigma\omega}+\|\omega_0(a)\|_{\sigma\omega} &= \frac{2\gamma}{\eta_\omega}=\frac2{\eta_\omega}\|ac\|_\omega,\\
    \|d\|_{\sigma\omega}+\|\omega_0(a)\|_{\sigma\omega} &= \frac{2\delta}{\eta_\omega}=\frac2{\eta_\omega}\|ad\|_\omega.
  \end{align*}
  Now given $g\in G_\omega$, write it as a word of minimal norm as
  $g=a^?x_1\cdots a x_\ell a^?$, with $x_i\in\{b,c,d\}$, and the $a^?$
  mean that the initial and final `$a$' may be present or absent. Thus
  $\|g\|_\omega\ge\|ax_1\|_\omega+\cdots+\|ax_\ell\|_\omega-\|a\|_\omega$. On
  the other hand, each `$a$' in the expression of $g$ contributes
  nothing to $g_0$ and $g_1$, while each `$x_i$' contributes an
  `$x_i$' and a `$\omega_0(x_i)$' to $g_0$ and $g_1$, in some
  order. Summing together the inequalities above gives the
  claimed~\eqref{eq:supercontract}.

  The second claim follows, because the extra `$\|a\|_\omega$' term
  occurs only if $g$ both starts and ends with a letter in
  $\{b,c,d\}$, and this case can be prevented by conjugating $g$ by
  its last letter.
\end{proof}  

Lemma~\ref{lem:supercontract} can be used to prove statements on
$G_\omega$ by induction. For example,
\begin{prop}\label{prop:faithful,infinite}
  If $\omega\in\Omega'$, then the action of $G_\omega$ on the tree
  $\mathscr T=X^*$ is faithful, and it is transitive on each orbit
  $X^i$. In particular, $G_\omega$ is infinite.
\end{prop}
\begin{proof}
  We first use induction on $\|\cdot\|_\omega$, simultaneously on all
  $\omega\in\Omega'$, to show that if $g\in G_\omega$ acts trivially
  on $\mathscr T$ then $g=1$.

  The induction starts by noting that the generators $b,c,d$ act
  non-trivially by our assumption that $\omega$ contains infinitely
  many $0,1,2$. Indeed, without loss of generality consider $b$; let
  $k\in\N$ be minimal such that $\omega_k(b)\neq1$; then $b$ acts
  non-trivially on $X^{k+1}$.

  Consider then $g\in G_\omega$ acting trivially on $\mathscr T$. In
  particular, $g$ fixes $X$, so $\phi_\omega(g)=\pair<g_0,g_1>$.  By
  Lemma~\ref{lem:supercontract}, both $g_0$ and $g_1$ are shorter, so
  by induction they are trivial; thus $g=1$ because $\phi_\omega$ is
  injective.

  To check that $G_\omega$ acts transitively on $X^i$, it suffices to
  show that the stabilizer $H$ of $\8$ acts transitively on
  $X^{i-1}\8$; because then $X^{i-1}\8 G_\omega=X^{i-1}\8\langle
  a\rangle=X^i$. Now $H$ contains $b,c,d,y^a$ for a letter
  $y\in\{b,c,d\}$ such that $\omega(y)=a$; and the action of
  $b,c,d,y^a$ on $w\8$ is $w b\8,w c\8,w d\8,w a\8$ respectively, so that
  the $H$-orbit of $w\8$ is $W\8$ for a $G_{\sigma\omega}$-orbit
  $W\subseteq X^{i-1}$. Again we are done by induction.
\end{proof}

Note that the action of $G_\omega$ is still faithful if $\omega$ only
contains infinitely many of two symbols; it is \emph{not} faithful if
$\omega$ contains finitely many of two symbols.

\subsection{\boldmath The $G_\omega$ are infinite torsion groups}
One of Burnside's questions~\cite{burnside:question} asks whether there
exist infinite, finitely generated groups in which every element has
finite order. The first such examples were constructed by
Golod~\cite{golod:nil}; here, we show that the groups $G_\omega$ are
other examples:
\begin{thm}\label{thm:torsion}
  If $\omega\in\Omega'$, then $G_\omega$ is an infinite torsion
  $2$-group.
\end{thm}
\begin{proof}
  The group $G_\omega$ is infinite by
  Proposition~\ref{prop:faithful,infinite}. We prove the claim by
  induction on $\|g\|$. It is easy to check that the generators
  $a,b,c,d$ all have order $2$. Consider $g\in G_\omega$, with
  $\phi_\omega(g)=\pair<g_0,g_1>\pi$. If $\pi=()$, then $g_0,g_1$ are
  shorter than $g$ by Lemma~\ref{lem:supercontract} so have finite
  order, say $2^{n_0},2^{n_1}$ respectively. Then $g$ has order
  $2^{\max\{n_0,n_1\}}$. If $\pi=(\9,\8)$, then
  $g^2=\pair<g_0g_1,g_1g_0>$, and $g_0g_1$ is shorter than $g$ again
  by Lemma~\ref{lem:supercontract}, so has finite order, say
  $2^n$. Then $g$ has order $2^{n+1}$.
\end{proof}

Note that, if $\omega$ contains finitely many copies of a symbol, then
$G_\omega$ is not a torsion group anymore. In fact, suppose that
$\omega$ contains no $0$; then $\omega_i(b)=a$ for all $i$, and the
element $ab$ has infinite order, since
$\phi_\omega((ab)^{2n})=\pair<(b a)^n,(ab)^n>$ for all $n\in\Z$.

We also note that the groups $G_\omega$ resemble very much infinitely
iterated wreath products; namely the map $\phi_\omega\colon
G_\omega\to G_{\sigma\omega}\wr\sym(2)$ is almost an isomorphism:
\begin{defn}
  Let $(G_\omega)_{\omega\in\Omega}$ be a similar sequence of
  groups. It is called \emph{branched} if every $G_\omega$ has a
  finite-index subgroup $K_\omega$ such that
  \[K_{\sigma\omega}^{X_\omega}\le\phi_\omega(K_\omega).
  \]
  If the subgroups $K_\omega$ are merely required to be non-trivial,
  then $(G_\omega)$ is called \emph{weakly branched}.

  Every group in a (weakly) branched family of groups is also called
  \emph{(weakly) branched}.
\end{defn}

\begin{prop}\label{prop:Gbranched}
  The groups $G_\omega$ are branched for all $\omega\in\Omega'$.
\end{prop}
\begin{proof}
  For each $\omega\in\Omega'$, let $x_\omega\in\{b,c,d\}$ be such that
  $\omega(x_\omega)=1$, and set
  $K_\omega=\langle[x_\omega,a]\rangle^{G_\omega}$. Choose also
  $y_\omega\in\{b,c,d\}\setminus\{x_\omega\}$. Then $1\times
  K_{\sigma\omega}$ is normally generated by
  $\pair<1,[x_{\sigma\omega},a]>$, and
  \[\pair<1,[x_{\sigma\omega},a]>=\phi_\omega([x_\omega,y_\omega^a])=\phi_\omega([x_\omega,a][x_\omega,a]^{a y_\omega})\in \phi_\omega(K_\omega);\]
  the same computation holds for $K_{\sigma\omega}\times1$.

  We now show that the groups $K_\omega$ have finite index.  Consider
  first the quotient $G_\omega/\langle
  x_\omega\rangle^{G_\omega}$. This group is generated by two
  involutions $a$ and $y_\omega$, so is a finite dihedral group,
  because $G_\omega$ is torsion. It follows that $\langle
  x_\omega\rangle^{G_\omega}$ has finite index in $G_\omega$. Then
  $\langle x_\omega\rangle^{G_\omega}/K_\omega=\langle x_\omega\rangle
  K_\omega$ has order $2$, so $K_\omega$ also has finite index.
\end{proof}

\begin{prop}\label{prop:branched=>fullwreath}
  If $G$ is a $p$-torsion weakly branched group, then it contains
  $L=\bigcup\wr^i C_p$ as a subgroup.
\end{prop}
\begin{proof}[Sketch of proof]
  Let $(G_\omega)_{\omega\in\Omega}$ be a similar family of groups,
  with $G=G_\omega$, and let $K_\omega\le G_\omega$ be the subgroups
  given by the condition that $(G_\omega)_{\omega\in\Omega}$ is
  branched. For each $\omega\in\Omega$, let $g_\omega\in K_\omega$ be
  an element of order $p$, and let $n(\omega)\in\N$ be such that
  $g_\omega$ acts non-trivially on
  $X_{\sigma^{n(\omega)-1}\omega}\times\cdots\times X_\omega$; let
  $v_\omega\in X_{\sigma^{n(\omega)-1}\omega}\times\cdots\times
  X_\omega$ be a point on a non-trivial orbit.

  Define then simultaneously and recursively $L_\omega=\langle
  L_{\sigma^{n(\omega)}\omega}@v_\omega, g_\omega\rangle$. It contains
  an element $g_\omega$ permuting $p$ copies of
  $L_{\sigma^{n(\omega)}\omega}$, so is isomorphic to $L$.
\end{proof}

If $G$ contains no torsion, or torsion of different primes, analogous
(but harder-to-state) results hold. In particular, by
Exercise~\ref{exse:nolaw}, branched groups satisfy no law. This
recovers a result by Ab\'ert~\cite{abert:nonfree}.

\subsection{\boldmath Lower growth estimates for $G_\omega$}
The proof of Theorem~\ref{thm:grigorchuk} requires upper and lower
bounds on the growth function of $G_\omega$. A lower bound is easily
provided by Theorem~\ref{thm:growth:hp}: the group $G_\omega$ is
infinite (Proposition~\ref{prop:faithful,infinite}) and torsion
(Theorem~\ref{thm:torsion}), so cannot be virtually nilpotent, since
nilpotent groups with finite-order generators are finite. However, in
some cases a direct argument also gives the bound
$v_{G_\omega}(R)\succsim\exp(\sqrt R)$.

Define indeed maps $\tilde\theta_\omega\colon F\to F$ by
\begin{equation}\label{eq:theta}
  \begin{cases}
    \tilde\theta_\omega(a) &=a y a\text{ if }\omega_0=\cdots=\omega_{k-1}\neq\omega_k,\text{ and }\omega_0(y)=a,\;\omega_k(y)=1,\\
    \tilde\theta_\omega(x)&=x\text{ for all }x\in\{b,c,d\}.
  \end{cases}
\end{equation}
A direct calculation shows that $\tilde\theta_\omega$ induces a
map $\theta_\omega\colon G_{\sigma\omega}\to G_\omega$, with
\[\phi_\omega(\theta_\omega(g))=\tikz[baseline=-2ex]{\draw[->] (0,0) --
    node[right=-2pt] {\small$*$} (0,-0.5); \draw[->](0.5,0) --
    node[right=-2pt] {\small$g$} (0.5,-0.5);}
\]
for some $*\in\langle a,y\rangle$ using the notation introduced
in~\eqref{eq:theta}. Furthermore, $\langle a,y\rangle$ is a dihedral
group of order $2^{k+2}$.
\begin{exse}
  Prove that $\theta_\omega$ is a group homomorphism.
\end{exse}

Let us now \textbf{cheat}, and assume that the element $*$ is always
trivial, rather than an element of a finite group of order
$2^{k+2}$. If $k$ is bounded, this is unimportant; however, if $k$ is
unbounded then an additional argument is really required. 

Denote by $B_\omega(R)$ the ball of radius $R$ in $G_\omega$, and
abbreviate $v_\omega(R)=v_{G_\omega}(R)$. Consider the map
$G_{\sigma\omega}^2\to G_\omega$ given by
$(g_0,g_1)\mapsto\theta_\omega(g_0)^a\cdot\theta_\omega(g_1)$. For
$R_0,R_1$ even, it defines an injective map
$B_{\sigma\omega}(R_0)\times B_{\sigma\omega}(R_1)\to
B_\omega(2(R_0+R_1))$, hence
$v_{\sigma\omega}(R)^2\le v_\omega(4R)$ for all $R$ even. We
conclude
$v_\omega(2\cdot4^k)\ge v_{\sigma^k\omega}(2)^{2^k}\ge5^{2^k}$,
so $v_\omega(R)\ge \exp(\frac12\log5\sqrt R)$.

\subsection{\boldmath Upper growth estimates for $G_\omega$}
We are ready to give an upper bound on the growth of $G_\omega$. As
before, we abbreviate $v_\omega(R)=v_{G_\omega}(R)$ for the growth
function of $G_\omega$ We start by a
\begin{lem}\label{lem:gamma(Amu)}
  For every $A>0$ there exists $B\in\N$ such that, for all
  $\omega\in\Omega'$, we have $v_\omega(A\mu_\omega)\le B$.
\end{lem}
\begin{proof}
  It suffices to bound the number of elements of $G_\omega$ whose
  minimal expression has length $\le A\mu_\omega$ and has the form
  $g=a x_1\cdots a x_\ell$ for some $x_i\in\{b,c,d\}$, since there are
  at most $8$ times more elements of norm $\le A\mu_\omega$, namely
  all $\{1,b,c,d\}g\{1,a\}$. Now by definition $g$ has norm at least
  $\ell\mu_\omega$, and there are at most $3^\ell$ such $g$'s, so one
  may take $B=8\cdot3^A$.
\end{proof}

\begin{lem}\label{lem:etamu}
  For all $\omega\in\Omega'$ we have
  $\eta_\omega\mu_{\sigma\omega}\ge\mu_\omega+\|a\|_\omega$.
\end{lem}
\begin{proof}
  Write $p_\omega=(\beta,\gamma,\delta)$, and assume without loss of
  generality $\omega_0=0$ so
  $\mu_{\sigma\omega}=2/\eta_\omega\,\min\{\gamma,\delta\}$. Consider
  the six orderings $\beta<\gamma<\delta$ etc.\ in turn to check the
  inequalities
  \[\eta_\omega\mu_{\sigma\omega}=2\min\{\gamma,\delta\}\ge\min\{\beta,\gamma,\delta\}+(1-\max\{\beta,\gamma,\delta\})=\mu_\omega+\|a\|_\omega.\qedhere\]
\end{proof}

\begin{lem}\label{lem:concavehull}
  Let $f$ be a positive sublinear function, namely $f(n)/n\to 0$ as
  $n\to\infty$. Then $f$ is bounded from above by a concave sublinear
  function.
\end{lem}
\begin{proof}
  For every $\theta\in(0,1)$, let $n_\theta$ be such that $f(n)-\theta
  n$ is maximal. Given $n\in\R$, let $\zeta<\theta$ be such that
  $n\in[n_\theta,n_\zeta]$ with maximal $\zeta$ and minimal $\theta$,
  and define $\overline f(n)$ on $[n_\theta,n_\zeta]$ by linear
  interpolation between $(n_\theta,f(n_\theta))$ and
  $(n_\zeta,f(n_\zeta))$. Clearly $\overline f\ge f$, and $\overline
  f(n)/n$ is decreasing and coincides infinitely often with $f(n)/n$,
  so it converges to $0$.
\end{proof}

\begin{prop}\label{prop:upperG}
  There is an absolute constant $B$ such that, for all
  $\omega\in\Omega'$ and all $k\in\N$,
  \[v_\omega(\eta_\omega\cdots\eta_{\sigma^{k-1}\omega}\mu_{\sigma^k\omega})\le B^{2^k}.\]
\end{prop}
\begin{proof}
  For all $i\in\N$, set
  $\alpha_i(R)=18(R+2)v_{\sigma^i\omega}(R+\mu_{\sigma^i\omega})$. We
  \textbf{cheat}, in assuming that the functions $\alpha_i$ are
  log-concave, i.e.\ satisfy
  $\alpha_i(R_0)\alpha_i(R_1)\le\alpha_i((R_0+R_1)/2)^2$. This
  assumption is in fact harmless, since each function $\alpha_i$ can
  be replaced by its log-concave majorand: the smallest log-concave
  function that is pointwise larger than $\alpha_i$, given by
  Lemma~\ref{lem:concavehull}. For details,
  see~\cite{bartholdi-erschler:givengrowth}*{\S3.1}.

  The proposition will follow from the inequalities $\alpha_i(R)\le
  \alpha_{i+1}(R/\eta_{\sigma^i\omega})^2$ for all $i,R$; because then
  \begin{align*}
    v_\omega(\eta_\omega\cdots\eta_{\sigma^{k-1}\omega}\mu_{\sigma^k\omega}) &\le \alpha_0(\eta_\omega\cdots\eta_{\sigma^{k-1}\omega}\mu_{\sigma^k\omega})\\
    &\le\alpha_1(\eta_{\sigma\omega}\cdots\eta_{\sigma^{k-1}\omega}\mu_{\sigma^k\omega})^2\\
    \le\cdots &\le\alpha_k(\mu_k)^{2^k}\le B^{2^k}\text{ by Lemma~\ref{lem:gamma(Amu)}}.
  \end{align*}

  To simplify notation, we consider only the case $i=0$, since all
  cases are the same. We have
  \begin{align*}
    \alpha_0(R) = 18(R+2)v_\omega(R+\mu_\omega) &\le
    18(R+2)\hspace{-2em minus 1fil}\sum_{R_0+R_1\le \tfrac2{\eta_\omega}
      (R+\mu_\omega+\|a\|_\omega)}\hspace{-2em minus 1fil}2v_{\sigma\omega}(R_0)v_{\sigma\omega}(R_1)\\
    &\le 36(R+2)^2 \hspace{-2em minus 1fil}\max_{R_0+R_1\le \tfrac2{\eta_\omega}
      (R+\mu_\omega+\|a\|_\omega)}\hspace{-2em minus 1fil}v_{\sigma\omega}(R_0)v_{\sigma\omega}(R_1)\\
    &\le
    6^2(3/\eta_\omega)^2(R+2)^2v_{\sigma\omega}((R+\mu_\omega+\|a\|_\omega)/\eta_\omega)^2\\
    &\le \big(18(R+2)/\eta_\omega\,v_{\sigma\omega}(R/\eta_\omega+\mu_{\sigma\omega})\big)^2\\
    &\le\alpha_1(R/\eta_\omega)^2.\qedhere
  \end{align*}
\end{proof}

We are now ready to conclude the proof of
Theorem~\ref{thm:grigorchuk}, by showing that the groups $G_\omega$
have subexponential growth. There are in fact different methods for
this. Let $\lambda_\omega$ be the exponential growth rate of
$G_\omega$, as in~\eqref{eq:lambda}; we are to show
$\lambda_\omega=1$.

Since by assumption $\omega$ contains infinitely many $0,1,2$, there
are infinitely many positions $k\in\N$ with $\omega_k=0$ and that are
separated by $1$ and $2$ in $\omega$. For these $k$, the point
$p_{\sigma^k\omega}\in\Delta$ belongs to the subsimplex
$\{\beta<\gamma\wedge\beta<\delta\}$. Thus $\beta<\tfrac13$ on that
subsimplex, $\eta_{\sigma^k\omega}=3-2\beta>7/3$ is uniformly bounded
away from $2$, and $\gamma,\delta>\tfrac16$ so
$\mu_{\sigma^{k+1}\omega}>\frac1{18}$. Thus by Proposition~\ref{prop:upperG}
\[\log\lambda_\omega=\lim\frac{\log v_\omega(R)}{R}\le\liminf\frac{2^k\log B}{\eta_\omega\cdots\eta_{\sigma^{k-1}\omega}\mu_{\sigma^k\omega}}=0\]
since on a subsequence the $\mu_{\sigma^k\omega}$ are bounded away
from $0$, all terms $2/\eta_{\sigma^i\omega}$ are bounded by $1$, and
infinitely many of them are bounded by $6/7$.

A more ``abstract'' proof may be obtained by noting that the map
$\omega\mapsto\lambda_\omega$ is continuous and bounded by $3$, and
that the proof of Proposition~\ref{prop:upperG} gives
$\log\lambda_\omega\le2/\eta_\omega\log\lambda_{\sigma\omega}$. Since
the action of $\sigma$ on $\Omega'$ is ergodic, we must have
$\log\lambda_\omega=0$ for all $\omega\in\Omega'$.

Let us compute more precisely an upper bound for the growth of the
first Grigorchuk group $G_{(012)^\infty}$. Since the sequence
$\omega=(012)^\infty$ is $3$-periodic, we can find $p_\omega\in\Delta$
explicitly. The calculation is made even simpler by noting that
$p_{\sigma\omega}$ and $p_{\sigma^2\omega}$ are cyclic permutations of
$p_\omega$; thus $p_\omega$ is the normalised eigenvector of
$M_0\cdot(\begin{smallmatrix}0&0&1\\1&0&0\\0&1&0\end{smallmatrix})$,
and its spectral radius is $\eta_+\approx2.46$, the positive root of the
characteristic polynomial $T^3-T^2-2T-4$. Thus for the first
Grigorchuk group we get $v_\omega(\eta_+^k\mu_\omega)\le
B^{2^k}$ for all $k$, and therefore
\[v_\omega(R)\precsim\exp(R^{\log 2/\log\eta_+})\approx\exp(R^{0.76}).\]

\section{Growth of permutational wreath products}\label{ss:perm}
The upper and lower bounds on the growth of $G_\omega$ are both of
intermediate type $\exp(R^\alpha)$, but do not match. We consider, in
this section, permutational wreath products based on the groups
$G_\omega$.

Choose a sequence $\omega\in\Omega'$ and a ray
$\xi\in\{\9,\8\}^\infty$, and consider the ray's orbit $X=\xi
G_\omega$. Choose a group $H$. Set then
\[W_\omega(H):=H\wr_X G_\omega.
\]
(Even though the notation does not make it clear, the group
$W_\omega(H)$ depends on $\xi$.) We shall show, in this section:

\begin{thm}\label{thm:wreathgrowth}
  If $H$ has subexponential growth, then so does $W_\omega(H)$.
\end{thm}

\begin{thm}\label{thm:givengrowth}
  Let $\eta_+\approx2.46$ be the positive root of $T^3-T^2-2T-4$. Let
  $f\colon\R_+\to\R_+$ be a function satisfying
  \begin{equation}\label{eq:gdoubling}
    f(2R)\le f(R)^2 \le f(\eta_+R)\text{ for all $R$ large enough}.
  \end{equation}
  Then there exists $\omega\in\Omega'$ such that
  $v_{W_\omega(C_2)}\sim f$.
\end{thm}

We shall give more illustrations of the growth functions that may occur
in~\S\ref{ss:illustrations}. We content ourselves with the following:
\begin{thm}\label{thm:grigWgrowth}
  For any finite group $H$, the group $W_{012}(H)$ has growth
  \[v_{W_{012}(H)}\sim\exp(R^{\log2/\log\eta_+}).\]
\end{thm}

The proofs of Theorems~\ref{thm:wreathgrowth}
and~\ref{thm:givengrowth} rely on estimates of the \emph{support}
$\subset X$ of an element of $W_\omega(H)$ of norm $\le R$. Recall
that every element of $W_\omega(H)$ may be written in the form $cg$
with $c\colon X\to H$ and $g\in G_\omega$; its support is
$\{x\in X:c(x)\neq1\}$. To better understand the support of elements
of $W_\omega(H)$, let us introduce the following
\begin{defn}\label{def:invorbit}
  Let $G$ be a group acting on the right on a set $X$ with basepoint
  $\xi$. For a word $w=w_1\dots w_\ell\in G^*$, its \emph{inverted
    orbit} is the set
  \[\mathcal O(w)=\{\xi w_{i+1}\cdots w_\ell:0\le i\le \ell\}.
  \]
  If furthermore $G$ is given with a metric $\|\cdot\|$, then its
  \emph{inverted orbit growth} is the function $\Delta\colon\R_+\to\N$
  given by
  \[\Delta(R)=\max\{\#\mathcal O(w):\|w\|\le R\}.\qedhere
  \]
\end{defn}

We write $\mathcal O_\omega(w)$ and $\Delta_\omega(R)$ in the case of
$G=G_\omega$ with its metric $\|\cdot\|_\omega$. Thus, for example,
taking $\xi=\8^\infty$, the inverted orbit of $acadab$ is
\[\mathcal O(acadab)=\{\xi acadab,\xi cadab,\xi adab,\xi dab,\xi ab,\xi
b,\xi\}=\{\overline\8\9\8\9,\overline\8,\overline\8\9\9\},
\]
see the Schreier graph at the end of~\S\ref{ss:ss}.

\noindent Note that a basepoint $\xi$ is implicit in the definitions;
yet,
\begin{exse}\label{exse:invgrowth}
  Assume that $G$ acts transitively on $X$. Show that $\Delta(R)$ depends
  only mildly of the choice of $\xi$, in the following sense: if
  $\xi,\xi'\in X$ are two choices of a basepoint and
  $\Delta(R),\Delta'(R)$ are the corresponding inverted orbit growth
  functions, then there exists a constant $C\in\R$ such that
  $\Delta(R)\le\Delta'(R+C)$ and $\Delta'(R)\le\Delta(R+C)$ for all $R$.
\end{exse}

\begin{prop}\label{prop:invgrowth}
  There exists a universal constant $C$ such that, for all
  $\omega\in\Omega'$ and all $k\in\N$,
  \[2^k\le\Delta_\omega(\eta_\omega\cdots\eta_{\sigma^{k-1}\omega}\mu_{\sigma^k\omega})\le C\,2^k.\]
\end{prop}
\begin{proof}
  For the upper bound, we note that Lemma~\ref{lem:supercontract}
  applies just as well to the group $G_\omega$ as to the monoid
  \[R:=S^*/\{b^2=c^2=d^2=1,b c=d,c d=b,db=c\},\]
  see Equation~\eqref{eq:F}. Indeed, in a minimal-length
  representative of an element of $R$, the number of `$a$' is at least
  the number of $b,c,d$-letters minus one, and this is the only
  property required for Lemma~\ref{lem:supercontract}. Now given a
  word $w\in S^*$, its inverted orbit may be read from the image of
  $w$ in $R$. Every element of $R$ has a unique \emph{reduced form}:
  the reduced form of a word $w\in S^*$ is the word
  $\overline w\in S^*$ obtained by replacing every subword equal to a
  left-hand side of a relation by the corresponding right-hand side.

  Without loss of generality and merely at the cost of increasing the
  constant $C$, we may suppose $\xi=\8^\infty$. We claim that the
  inverted orbit of a word $w\in S^*$ coincides with the inverted
  orbit of its reduction $\overline w$. To see this, consider
  $w=w_1\dots w_\ell\in S^*$, a subword $w_j w_{j+1}$ equal to a
  left-hand side of a relation, and the word $w'$ obtained by
  replacing $w_j w_{j+1}$ by the right-hand side of the relation. All
  terms $\xi w_{i+1}\dots w_\ell$ with $i\neq j$ clearly appear both
  in $\mathcal O(w)$ and $\mathcal O(w')$. For the remaining term in
  $\mathcal O(w)$, we have
  $\xi w_{i+1}\dots w_\ell=\xi w_{i+2}\dots w_\ell$ because $w_{i+1}$
  fixes $\xi$, so this term also belongs to $\mathcal O(w')$.

  If $w\in F$ satisfies $\tilde\phi_\omega(w)=\pair<w_0,w_1>\pi$ and
  $\xi=\xi'\9$, then
  \[\mathcal O_\omega(w)\subseteq\mathcal
  O_{\sigma\omega}(w_0)\9\sqcup\mathcal O_{\sigma\omega}(w_1)\8,
  \]
  where the inverted orbits $\mathcal O_{\sigma\omega}$ are computed
  with respect to the basepoint $\xi'$; and similarly if
  $\xi=\xi'\8$. We therefore get
  \[\Delta_\omega(R)\le\max_{R_0+R_1\le2/\eta_\omega(R+\|a\|_\omega)}\big(\Delta_{\sigma\omega}(R_0)+\Delta_{\sigma\omega}(R_1)\big).\]
  The same argument as in Proposition~\ref{prop:upperG} finishes the
  proof of the upper bound.

  For the lower bound, it suffices to exhibit for all $k\in\N$ a word
  of length at most
  $\eta_\omega\cdots\eta_{\sigma^{k-1}\omega}\mu_{\sigma^k\omega}$ and
  inverted orbit of size at least $2^k$. For that purpose, define
  self-substitutions $\zeta_x$ of $\{ab,ac,ad\}^*$, for
  $x\in\{0,1,2\}$, by
  \[\begin{array}{rlll}
    \zeta_0:&ab\mapsto adabac,&ac\mapsto acac,&ad\mapsto adad,\\
    \zeta_1:&ab\mapsto abab,&ac\mapsto abacad,&ad\mapsto adad,\\
    \zeta_2:&ab\mapsto abab,&ac\mapsto acac,&ad\mapsto acadab,
  \end{array}
  \]
  and note that for any word $w\in\{ab,ac,ad\}^*$ representing an
  element of $F$ we have
  \[\tilde\phi_\omega(\zeta_{\omega_0}(w))=\begin{cases}
    \tikz[baseline=1.5mm]{\draw[->] (0,0) -- node[above=-1pt] {\small$w$} +(1.0,0); \draw[->] (0,0.5) -- node[above=-1pt] {\small$a w a$} +(1.0,0);} & \text{ if $\zeta_{\omega_0}(w)$ contains an even number of `$a$',}\\
    \tikz[baseline=1.5mm]{\draw[->] (0,0) -- node[near start,below] {\small$w$} node[near end,above=-1pt] {\small$a$} +(1.0,0.5); \draw[->] (0,0.5) -- node[near start,above=-1pt] {\small$a$} node[near end,below] {\small$w$} +(1.0,-0.5);} & \text{ if $\zeta_{\omega_0}(w)$ contains an odd number of `$a$'}.
  \end{cases}
  \]
  In particular, $\zeta_{\omega_0}$ induces a homomorphism
  $G_{\sigma\omega}\to G_\omega$.

  By induction, we see that for any non-trivial $w\in\{ab,ac,ad\}^*$
  (representing an element of $G_{\sigma^k\omega}$) we have
  \[\Delta_\omega(\zeta_{\omega_{0}}\cdots\zeta_{\omega_{k-1}}(w))\ge2^k.\]

  Note then that, if $Z\in\N^3$ count the numbers of $ab,ac,ad$
  respectively in $w$, then $M_x^{\mathrm t}Z$ counts the numbers of
  $ab,ac,ad$ respectively in $\zeta_x(w)$. Indeed without loss of
  generality consider $x=0$; then every $ab$ in $w$ contributes one
  each of $ab,ac,ad$ to $\zeta_x(w)$, while every $ac$ and $ad$ in $w$
  contributes two copies of itself to $\zeta_x(w)$.

  Let $as\in\{ab,ac,ad\}$ be such that $\|as\|_{\sigma^k\omega}$ is
  minimal --- recall the notation $p_\omega$ from~\eqref{eq:pomega}:
  if $p_{\sigma^k\omega}=(\beta,\gamma,\delta)$ and
  $\beta\le\gamma,\delta$ then $s=b$, etc. Let $W$ be the basis vector
  in $\R^3$ with a `$1$' at the position which $as$ has in
  $\{ab,ac,ad\}$: if $\beta\le\gamma,\delta$ then $W=(1,0,0)^{\mathrm
    t}$, etc. Set
  $w=\zeta_{\omega_{0}}\cdots\zeta_{\omega_{k-1}}(as)$. We have
  $\Delta_\omega(w)\ge2^k$, and
  \begin{align*}
    \|w\|_\omega&= Z^t p_\omega = W^{\mathrm t}M_{\omega_{k-1}}\cdots M_{\omega_{0}}p_\omega\\
    &= \eta_\omega\cdots\eta_{\sigma^{k-1}\omega}W^{\mathrm t}p_{\sigma^k\omega} = \eta_\omega\cdots\eta_{\sigma^{k-1}\omega}\mu_{\sigma^k\omega}.\qedhere
  \end{align*}
\end{proof}

Finally, let us introduce the ``choice of inverted orbits growth''
function, first generally for a group $G$, with given metric
$\|\cdot\|$, acting on a set $X$ with basepoint $\xi$:
\[\Sigma(R):=\#\{\mathcal O(w):\|w\|\le R\}.\]
This function counts the number of subsets that may occur as inverted
orbit of a word of length at most $R$. Since $\mathcal O(w)$ is a
subset of cardinality at most $R+1$ of the Schreier graph $X$, and
furthermore lies in the ball of radius $R$ about $\xi$ in $X$, we get
the crude estimate $\Sigma(R)\le\binom{v_{X,\xi}(R)+R}{R}$ based on
the growth function $v_{X,\xi}$ of balls centered at $\xi$ in the
graph $X$. However, in the particular case of the groups $G_\omega$,
we can do better:
\begin{prop}\label{prop:invchoices}
  There is an absolute constant $D$ such that for all
  $\omega\in\Omega'$ and all $k\in\N$ we have
  \[\Sigma_\omega(\eta_\omega\cdots\eta_{\sigma^{k-1}\omega}\mu_{\sigma^k\omega})\le D^{2^k}.\]
\end{prop}
\begin{proof}
  Consider $w\in G_\omega^*$ with
  $\tilde\phi_\omega(w)=\pair<w_0,w_1>$. The inverted orbit of $w$ is
  determined by the inverted orbits of $w_0$ and $w_1$, two words of
  total $\sigma\omega$-length at most
  $2/\eta_\omega(\|w\|_\omega+\|a\|_\omega)$ by
  Lemma~\ref{lem:supercontract}. Therefore,
  \[\Sigma_\omega(\eta_\omega R)\le\sum_{R_0+R_1\le \tfrac2{\eta_\omega}
    (R+\|a\|_\omega)}\Sigma_{\sigma\omega}(R_0)\Sigma_{\sigma\omega}(R_1),\]
  and the same argument as in Proposition~\ref{prop:upperG} applies.
\end{proof}

\subsection{\boldmath The growth of $W_\omega(H)$}
We start by general estimates on the growth of a permutational wreath
product:
\begin{prop}\label{prop:growthW}
  Let $H$ be a group with growth function $v_H$, and suppose that
  $v_H$ is log-concave.

  Let $G$ be a group acting transitively on a set $X$ with basepoint
  $\xi$, and let $v_G$ denote the growth function of $G$.  Denote the
  inverted orbit growth of $G$ on $(X,\xi)$ by $\Delta$, and denote
  its inverted orbit choice growth by $\Sigma$.

  Consider the wreath product $W=H\wr_X G$, generated by $S\cup T@\xi$
  for the generating sets $S,T$ of $G,H$ respectively. Then
  \begin{align*}
    & v_G(R) v_H(R/\Delta(R))^{\Delta(R)}\le v_W(3R),\\
    v_W(R) \le & v_G(R) v_H(R/\Delta(R))^{\Delta(R)}(2R)^{\Delta(R)}\Sigma(R),\\
    v_W(R) \le & v_G(R) (\#H)^{\Delta(R)}\Sigma(R)\text{ if $H$ is finite}.
  \end{align*}
\end{prop}
Note that the assumption that $v_H$ be log-concave is mild, thanks to
Lemma~\ref{lem:concavehull}.
\begin{proof}
  We begin by the lower bound. For every $R\in\N$, consider a word
  $w\in G^*$ of norm $\le R$ realizing the maximum $\Delta(R)$; write
  $\mathcal O(w)=\{x_1,\dots,x_k\}$ for $k=\Delta(R)$. Choose then $k$
  elements $a_1,\dots,a_k$ of norm $\le R/k$ in $A$. Define $f\in
  \sum_X H$ by $f(x_i)=a_i$, all unspecified values being $1$. Then
  $w f\in W$ may be expressed as a word of norm $R+|a_1|+\dots+|a_k|\le
  2R$ in the standard generators of $W$, by inserting
  $a_1@\xi,\dots,a_k@\xi$ appropriately into the word $w$.

  Furthermore, different choices of $a_i$ yield different elements of
  $W$. Finally multiplying $w f$ with an arbitrary $g\in G$ of length
  at most $R$, we obtain $v_G(R) v_H(R/k)^k$ elements in the ball of
  radius $3R$ in $W$.

  For the upper bound, consider a word $w$ of norm $R$ in $W$, and
  let $f\in\sum_X H$ denote its value in the base of the wreath
  product. The support of $f$ has cardinality at most $\Delta(R)$, and
  may take at most $\Sigma(R)$ values.

  Write then $\sup(f)=\{x_1,\dots,x_k\}$ for some $k\le\Delta(R)$, and
  let $a_1,\dots,a_k\in H$ be the values of $w$ at its support; write
  $\ell_i=\|a_i\|$.

  Since $\sum\ell_i\le R$, the norms of the different elements on the
  support of $f$ define a composition of a number not greater than $R$
  into at most $k$ summands; such a composition is determined by $k$
  ``marked positions'' among $R+k$, so there are at most
  $\binom{R+k-1}{k}$ possibilities, which we bound crudely by
  $(2R)^k$. Each of the $a_i$ is then chosen among $v_H(\ell_i)$
  elements, and (by the assumption that $v_H$ is log-concave) there
  are $\prod v_H(\ell_i)\le v_H(R/k)^k$ total choices for the elements
  in $H$. 

  We have now decomposed $w$ into data that specify it uniquely, and
  we multiply the different possibilities for each of the pieces of
  data. Counting the possibilities for the value of $w$ in $G$, the
  possibilities for its support in $X$, and the possibilities for the
  elements in $H$ on its support, we get
  \[v_W(R)\precsim v_G(R)v_A(R/k)^k(2R)^k\Sigma(R),\]
  which is maximised by $k=\Delta(R)$.

  Finally, if $H$ is finite then we may more simply bound the possible
  values of $f$ by $(\#H)^k$.
\end{proof}

\begin{cor}\label{cor:growthW}
  Let $G,H$ be groups of subexponential growth. Let $G$ act
  transitively on a set $X$ with basepoint $\xi$, with sublinear
  inverted orbit growth and subexponential inverted orbit choice
  growth. Then the wreath product $W=H\wr_X G$ has subexponential
  growth.\qed
\end{cor}

\begin{proof}[Proof of Theorem~\ref{thm:wreathgrowth}]
  Assume that $H$ has subexponential growth; then by
  Lemma~\ref{lem:concavehull} there exists a log-concave
  subexponentially growing function $v_H$ bounding the growth of $H$
  from above.

  By Propositions~\ref{prop:invgrowth} and~\ref{prop:invgrowth}, the
  function $\Delta$ is sublinear and the function $\Sigma$ is
  subexponential. By Proposition~\ref{prop:upperG}, the growth of
  $G_\omega$ is subexponential. Corollary~\ref{cor:growthW} then
  shows that $W_\omega(H)$ has subexponential growth.
\end{proof}

\noindent In the special case of $H$ finite and $G=G_\omega$,
Proposition~\ref{prop:growthW} gives the
\begin{cor}\label{cor:global}
  Let $H$ be a non-trivial finite group. There are then two absolute
  constants $F,E>1$ such that the growth function $v$ of
  $W_\omega(H)=H\wr_X G_\omega$ satisfies
  \[E^{2^k}\le v(\eta_\omega\cdots\eta_{\sigma^{k-1}\omega}\mu_{\sigma^k\omega})\le F^{2^k}.\]
\end{cor}
\begin{proof}
  Take together the upper bound on the growth of $G_\omega$ from
  Proposition~\ref{prop:upperG}, the bounds on the inverted orbit
  growth from Proposition~\ref{prop:invgrowth}, and the choices for
  the inverted orbits from Proposition~\ref{prop:invchoices}. The
  conclusion follows from Proposition~\ref{prop:growthW}.
\end{proof}

\begin{proof}[Proof of Theorem~\ref{thm:grigWgrowth}]
  This follows directly from Corollary~\ref{cor:global}, using the
  fact that $\eta_{\sigma^i\omega}=\eta_+$ for all $i\in\N$.
\end{proof}

\subsection{Proof of Theorem~\ref{thm:givengrowth}}
Our approach will be construct, out of the function $f$
satisfying~\eqref{eq:gdoubling}, a sequence
$\omega\in\{0,012\}^\infty$ with long stretches of $0$ when $f$ grows
fast, and long stretches of $012$ when $f$ grows slowly.

We start by introducing some shorthand notation.  For a finite
sequence $\omega=\omega_0\dots\omega_{n-1}\in\{0,1,2\}^n$ and
$p\in\Delta$, we write by extension
\[\overline M_\omega=\overline M_{\omega_0}\cdots\overline M_{\omega_{n-1}}\colon\Delta\righttoleftarrow\]
and
\[\eta(p,\omega_0\dots\omega_{n-1})=\eta(p,\omega_0)\eta(\overline M_{\omega_0}p,\omega_1)\cdots\eta(\overline M_{\omega_0\dots\omega_{n-2}}p,\omega_{n-1}).\]

For $\omega\in\{0,1,2\}^\Z$, recall the construction of
$p_\omega\in\Delta$ from~\eqref{eq:pomega}, and
$\eta_\omega=\eta(p_\omega,\omega_0)$ and
$\mu_\omega=\mu(p_\omega)$. Assume that a sequence
$\omega\in\{0,1,2\}^\Z$ is under construction, and that there exists
$k\in\N$ such that $\omega_i$ has been determined for all $i\le
k$.
Then $p_\omega$, $\eta_\omega$ and $\mu_\omega$ are determined, and so
are $p_{\sigma^i\omega},\eta_{\sigma^i\omega},\mu_{\sigma^i\omega}$
for all $i\le k$. We abbreviate
\[p_i=p_{\sigma^i\omega},\qquad \eta_i=\eta_{\sigma^i\omega},\qquad\mu_i=mu_{\sigma^i\omega}.\]

\begin{lem}\label{lem:mu_k bounded}
  If the restriction to $\N$ of the sequence $\omega$ has the form
  \begin{equation}\label{eq:omega}
    \omega=(012)^{i_1}2^{j_1}(012)^{i_2}2^{j_2}(012)^{i_3}2^{j_3}\dots,
  \end{equation}
  with $i_1,j_1,i_2,j_2,\dots\ge1$, then the $\mu_k$ are all bounded
  away from $0$.
\end{lem}
\begin{proof}
  In fact, the image of
  $\overline M_{012}$ is the open triangle spanned by
  $(\frac13,\frac13,\frac13)$, $(\frac27,\frac27,\frac37)$ and
  $(\frac4{17},\frac6{17},\frac7{17})$, so after each $012$
  the $\mu_k$ belongs to $(\frac4{17},\frac13)$.

  The image of that triangle under $\overline M_{2^i}$ is
  contained in the convex quadrilateral spanned by
  $(\frac13,\frac13,\frac13)$, $(\frac4{17},\frac6{17},\frac7{17})$,
  $(\frac14,\frac14,\frac12)$ and $(\frac15,\frac3{10},\frac12)$, so
  $\mu_k\in(\frac15,\frac13)$ for all $k$.
\end{proof}

The heart of the argument is the following lemma, which shows that
$\eta$ approaches very quickly its limiting values $2$ and $\eta_+$ as
the sequence $\omega$ contains long segments of $2$ or of $012$:
\begin{lem}\label{lem:expconv}
  There exist constants $A'\le1$, $B'\ge1$ such that, 
  \begin{enumerate}
  \item\label{lem:expconv1} For all $p\in\Delta$ and all $n\in\N$,
    \[\eta(p,(012)^n)\ge\eta_+^{3n}A';\]
  \item\label{lem:expconv2} For all $p\in\Delta$ and all $n\in\N$,
    \[\eta(p,2^n)\le2^n B'.\]
  \end{enumerate}
\end{lem}
\begin{proof}
  Let $\mathcal U$ denote the image of $\overline M_{012}$, and let
  $p_+\in\mathcal U$ denote the fixed point of $\overline
  M_{012}$. Note first that $\eta(-,012)$ is differentiable at $p_+$,
  and that $\overline M_{012}$ is uniformly contracting on $\mathcal
  U$; let $\rho<1$ be such that $\overline M_{012}$ is
  $\rho$-Lipschitz on $\mathcal U$, and let $D$ be an upper bound for
  the derivative of $\log\eta(-,012)$ on $\mathcal U$. Recall that
  $\eta_+^3=\eta(p_+,012)$.

  For all $k\in\N$, write $p_k=\overline M_{(012)^k}(p)$. For $k\ge1$
  we have $\|p_k-p_+\|\le\rho^{k-1}$, so
  $|\log\eta(p_k,012)-3\log\eta_+|<D\rho^{k-1}$, while for $k=0$ we
  write $|\log\eta(p,012)-3\log\eta_+|<3\log3$. Therefore,
  \[|\log\eta(p,(012)^n)-3n\log\eta_+|\le
  3\log3+\sum_{k=1}^{n-1}|\log\eta(p_k,012)-3\log\eta_+| \le
  3\log3+D/(1-\rho)
  \]
  is bounded over all $n$ and $p$. The estimate~\eqref{lem:expconv1}
  follows, with $A'=3^3\exp(D/(1-\rho))$.

  For the second part, consider $p_k=\overline M_{2^k}(p)$, and note
  that $p_k$ converges at exponential (Euclidean) speed to a point
  $p_\infty$ on the side $\{\delta=\frac12\}$ of $\partial\Delta$,
  since $\partial\overline
  M_2(\beta,\gamma,\delta)/\partial\delta(*,*,\frac12)=\frac12$; so we
  have $\|p_k-p_\infty\|<\rho^{k-1}$ for some $\rho<1$. As above,
  $\eta(-,2)$ is differentiable in a neighbourhood $\mathcal U$ of
  $\{\delta=\frac12\}$, and the derivative of $\log\eta(-,2)$ is
  bounded on $\mathcal U$, say by $D$. Recall that $\eta(p,2)=2$ for
  all $p\in\partial\Delta$. The same computation as above
  yields~\eqref{lem:expconv2}, with $B'=3\exp(D/(1-\rho))$.
\end{proof}

We now reformulate the statement of Theorem~\ref{thm:givengrowth} as
follows. Set $g(R)=\log f(R)$, so that we have
\begin{equation}\label{eq:g}
  g(2R)\le 2g(R)\le g(\eta_+R)
\end{equation}
for all $R$ large enough. For simplicity (since growth is only an
asymptotic property) we assume that~\eqref{eq:g} holds for all $R$.
Without loss of generality (since we are allowed to replace $g$ by an
equivalent function), we also assume that $g$ is increasing and
satisfies $g(1)=1$.

We are ready to construct $\omega$. Fix arbitrarily the value of
$\omega$ on its negative part, say
$\omega_{|-\N}=(012)^{-\infty}$. This determines an initial metric
$p_0\in\Delta$. Out of the function $g$, we will construct a sequence
$\omega$ such that, for constants $A,B$, we have
\begin{equation}\label{eq:AB}
  A\le \frac{g(\eta(p_0,\omega_0\dots\omega_{k-1}))}{2^k}\le B\text{ for
  all }k;
\end{equation}
in fact, it will suffice to obtain this inequality for a set of values
$k_0,k_1,\dots$ of $k$ such that $\sup_i(k_{i+1}-k_i)<\infty$. Indeed,
the orbit $p_i$ of $p_0$ in $\Delta'$ will remain bounded, so we will
have $\mu(p_i)\in[C,1]$ for some $C>0$. By Corollary~\ref{cor:global},
\begin{align*}
  E^{2^k}\le v(\eta(p_0,\omega_0\dots\omega_{k-1})\mu_k)\le &v(B g^{-1}(2^k)),\\
  &v(A C g^{-1}(2^k))\le v(\eta(p_0,\omega_0\dots\omega_{k-1})\mu_k)\le F^{2^k}
\end{align*}
and therefore $v(R)\sim \exp(g(R))=f(R)$.

\begin{proof}[Proof of~\eqref{eq:AB}]
  We will extend the sequence $\omega$ to be, on the positive
  integers, of the form~\eqref{eq:omega},
  \[\omega=(012)^{i_1}2^{j_1}(012)^{i_2}2^{j_2}(012)^{i_3}2^{j_3}\dots,
  \]
  with $i_1,j_1,i_2,j_2,\dots\ge1$. The $\mu_k$ are bounded away from
  $0$ by Lemma~\ref{lem:mu_k bounded}.

  Assuming by induction that $\omega'=\omega_0\dots\omega_{k-1}$ has
  been constructed, we repeat the following:
  \begin{itemize}
  \item while $g(\eta(p_0,\omega'))<2^k$, we append $012$ to
    $\omega'$;
  \item while $g(\eta(p_0,\omega'))>2^k$, we append $2$ to $\omega'$.
  \end{itemize}
  For our induction hypothesis, we assume that the stronger condition
  \[\frac122^k\le g(\eta(p_0,\omega'))\le 2^k\]
  holds for each $k$ of the form $i_1+j_1+\dots+i_m+j_m$, and that
  \[2^k\le g(\eta(p_0,\omega'))\le3^32^k\] holds for each $k$ of the form
  $i_1+j_1+\dots+i_m$; these conditions apply whenever $\omega$ is a
  product of `syllables' $(012)^{i_t}$ and $2^{j_t}$.

  Consider first the case $\frac122^k\le g(\eta(p_0,\omega'))\le2^k$;
  and let $n$ be minimal such that
  $g(\eta(p_0,\omega'(012)^n))>2^{k+3n}$. Then, for all
  $i\in\{1,\dots,n\}$, Lemma~\ref{lem:expconv}\eqref{lem:expconv1}
  gives
  $\eta(p_0,\omega'(012)^i)\ge\eta(p_0,\omega')\eta_+^{3i}A'$.
  Let $u\in\N$ be minimal such that $A'\ge\eta_+^{-u}$; this, combined
  with $g(\eta_+R)\ge2g(R)$, gives
  \[g(\eta(p_0,\omega'(012)^i))\ge g(\eta(p_0,\omega')\eta_+^{3i-u})\ge
  2^{-1-u}2^{k+3i}.
  \]
  By minimality of $n$, we have
  $g(\eta(p_0,\omega'(012)^{n-1}))\le2^{k+3(n-1)}$; since $g$
  is sublinear and $\eta\le3$, we get
  \[g(\eta(p_0,\omega'(012)^n))\le3^32^{k+3n}.\]

  Consider then the case $2^k\le g(\eta(p_0,\omega'))\le3^32^k$, which
  is similar, and let $n$ be minimal such that
  $g(\eta(p_0,\omega'2^n))<2^{k+n}$. Then, for all
  $i\in\{1,\dots,n\}$, Lemma~\ref{lem:expconv}\eqref{lem:expconv2}
  gives $\eta(p_0,\omega'2^i)\le\eta(p_0,\omega')2^i B'$; this,
  combined with $g(2R)\le2g(R)$, gives
  \[g(\eta(p_0,\omega'2^i))\le3^32^{k+i}B'.
  \]
  By minimality of $n$, we have
  $g(\eta(p_0,\omega'2^{n-1}))\ge2^{k+n-1}$; since $g$ is
  increasing, we get
  \[g(\eta(p_0,\omega'2^n))\ge\frac122^{k+n}.
  \]
  We have proved the claim~\eqref{eq:AB}, with
  $A=2^{-1-u}$ and $B=3^3B'$.
\end{proof}

\begin{rem}\label{rem:solvablewp}
  The construction of $\omega$ from $f$ is algorithmic, in the
  following sense. The initial point $p_0$ may be computed to
  arbitrary precision by an algorithm. If there exists an algorithm
  that computes values of $f$, then there exists an algorithm that,
  with $k\in\N$ as input, computes
  $g(\eta(p_0,\omega_0\dots\omega_{k-1}))$ to arbitrary precision; so
  there exists an algorithm that computes the digits of $\omega$.

  It then follows via Theorem~\ref{thm:cornulier} that the groups
  $G_\omega$ and $W_\omega(H)=H\wr_X G_\omega$ are recursively
  presented, for recursively presented $H$.
\end{rem}

\subsection{Illustrations}\label{ss:illustrations}
We now consider illustrations of Theorem~\ref{thm:givengrowth}, and
examples of growth functions that may occur for groups
$W_\omega(C_2)$. In fact, the examples can be constructed in both
directions: either choose a ``nice'' function $f$ that
satisfies~\eqref{eq:gdoubling}, or choose a ``nice'' sequence $\omega$
and estimate the corresponding growth using
Corollary~\ref{cor:global}, Lemma~\ref{lem:mu_k bounded} and
Lemma~\ref{lem:expconv}. We follow both approaches.

\begin{itemize}
\item For every $\alpha\in[\log2/\log\eta_+,1]$, there exists a group
  of growth $\sim\exp(R^\alpha)$. Furthermore, for a dense set of
  $\alpha$ in that interval, there exists a periodic sequence $\omega$
  such that $W_\omega(C_2)$ has growth $\sim\exp(R^\alpha)$.

  For every $\alpha\le\beta\in[\log2/\log\eta_+,1)$, one may construct
  a function $f$ satisfying~\eqref{eq:gdoubling} that coincides, on
  arbitrarily large intervals, sometimes with the function
  $\exp(R^\alpha)$ and sometimes with the function
  $\exp(R^\beta)$. Therefore, there exists a group whose growth
  function accumulates both at $\exp(R^\alpha)$ and at
  $\exp(R^\beta)$. This recovers a result by Brieussel,
  see~\cite{brieussel:growth}.

\item there exists groups of growth $\sim\exp(R/\log R)$, of growth
  $\sim\exp(R/\log\log R)$, of growth $\exp(R/\log\cdots\log R)$.

\item Consider conversely the sequence
  $\omega=(012)2^1(012)2^2(012)2^3(012)2^4\dots$. Among the first $k$
  entries, approximately $\sqrt k$ instances of $012$ will have been
  seen; therefore
  $\eta(p_0,\omega_0\dots\omega_{k-1})\approx 2^{k+\mathcal O(1)\sqrt
    k}$. This gives a growth function of the order of
  \[\exp\big(R/\exp(\mathcal O(1)\sqrt{\log R})\big).\]

  Consider next the sequence
  $\omega=(012)2^1(012)2^2(012)2^4(012)2^8\dots$. Among the first $k$
  entries, approximately $\log k$ instances of $012$ will have been
  seen; therefore $\eta(p_0,\omega_0\dots\omega_{k-1})\approx
  2^{k+\mathcal O(1)\log k}$. This gives a growth function of the
  order of
  \[\exp\big(R/(\log R)^{\mathcal O(1)}\big).\]
  
  Consider further the sequence
  $\omega=(012)2^{2^1}(012)2^{2^2}(012)2^{2^4}(012)2^{2^8}\dots$.
  Among the first $k$ entries, approximately $\log\log k$ instances of
  $012$ will have been seen; therefore
  $\eta(p_0,\omega_0\dots\omega_{k-1})\approx 2^{k+\mathcal
    O(1)\log\log k}$. This gives a growth function of the order of
  \[\exp\big(R/(\log\log R)^{\mathcal O(1)}\big).\]

  These constructions generalise easily to give ``nice'' sequences
  $\omega$ such that $W_\omega(C_2)$ has growth of the order of
  $\exp(R/(\log\cdots\log R)^{\mathcal O(1)})$.

\item Consider the Ackermann function
  \[A(m,n)=\begin{cases}
    n+1 & \text{ if }m=0,\\
    A(m-1,1) & \text{ if $m>0$ and }n=0,\\
    A(m-1,A(m,n-1)) & \text{ if $m>0$ and }n>0,
  \end{cases}
  \]
  and set $B(n)=A(n,n)$.  Then there exists a group of growth
  $\sim\exp(R/B(R)^{-1})$; this last growth function is faster than
  any subexponential primitive-recursive function.

  Considering
  $\omega=(012)2^{A(0,0)}(012)2^{A(1,1)}(012)2^{A(2,2)}\dots$ gives a
  group $W_\omega$ whose growth function is subexponential, but at
  least as fast as $\exp(R/B(R)^{-1})$;
\end{itemize}

\section{Imbeddings and subgroups}\label{ss:imbed}
A classical result by Higman, Neumann and
Neumann~\cite{higman-n-n:embed} states that every countable group
imbeds in a finitely generated group. It was then shown that many
properties of the group can be inherited by the imbedding: in
particular, solvability (Neumann-Neumann~\cite{neumann-n:embed}),
torsion (Phillips~\cite{phillips:embedding}), residual finiteness
(Wilson~\cite{wilson:embeddingrf}), and amenability
(Olshansky-Osin~\cite{olshanskii-osin:qiembedding}).

Seen the other way round, these results show that there is little
restriction, apart from being countable, on the subgroups of a
finitely generated group; even if that group is furthermore assumed to
be residually finite, amenable, or solvable.

On the other hand, very strong restrictions exist on the subgroups of
a virtually nilpotent group: they are all finitely generated, for
example.

Since by Gromov's Theorem~\ref{thm:gromov} the finitely generated
virtually nilpotent groups are precisely the groups of polynomial
growth, we naturally ask what conditions are imposed on the subgroups
of a group of subexponential word growth. As we shall see, there are
essentially none.

In particular, there are torsion branched groups of subexponential
growth such as the group $G_\omega$. They contain iterated wreath
products, by Proposition~\ref{prop:branched=>fullwreath}, so contain
infinitely generated subgroups.

Let us say that a group has \emph{locally subexponential growth} if
all of its finitely generated subgroups have subexponential growth.
Clearly, if $G$ has subexponential growth then all its subgroups have
locally subexponential growth. This is the only restriction, and the
objective of this section is to prove the following result:
\begin{thm}[\cite{bartholdi-erschler:imbeddings}]\label{thm:imbed}
  Let $B$ be a group. Then there exists a finitely generated group of
  subexponential growth in which $B$ imbeds as a subgroup if and only
  if $B$ is countable and locally of subexponential growth.
\end{thm}

For example, this implies that there exists a group of subexponential
growth containing $\Q$ as a subgroup. I do not know any explicit such
example of group.

\subsection{Neumann's proof}
By way of motivation, we start with the classical result by Higman and
the Neumanns:
\begin{thm}[Higman, B.H. Neumann and H. Neumann~\cite{higman-n-n:embed}]\label{thm:hnn}
  Every countable group imbeds in a finitely generated group.
\end{thm}

We shall not follow the original proof (which proceeds by a sequence
of ``HNN extensions''), but rather that by the two
Neumanns~\cite{neumann-n:embed}, which uses wreath products. It
follows immediately from combining the following two propositions:
\begin{prop}\label{prop:neumann-n:1}
  Every countable group $B$ imbeds in the commutator subgroup
  $[G,G]$ of a countable group $G$.
\end{prop}
\begin{prop}\label{prop:neumann-n:2}
  For every countable group $G$, there exists a $2$-generated group
  $W$ such that $[G,G]$ imbeds in $[W,W]$.
\end{prop}

\begin{proof}[Proof of Proposition~\ref{prop:neumann-n:1}]
  Consider the following subgroup $G$ of the unrestricted wreath
  product $B\wrwr\Z$. The group $G$ is generated by $\Z$ and, for all
  $b\in B$, the function $f_b\colon\Z\to B$ defined by
  $f_b(m)=b^m$. Denoting by $t$ the generator of $\Z$, we see that
  $[t,f_b]$ is the constant function $b$; so $B$ is in fact imbedded
  in $[t,G]$.
\end{proof}

\begin{exse}
  Could we also have defined $f_b(m)=b$ for $m\ge0$ and $f_b(m)=1$ for
  $m<0$? What would be the advantages and disadvantages of this
  alternative construction?
\end{exse}

\begin{proof}[Proof of Proposition~\ref{prop:neumann-n:2}]
  Consider the following subgroup $W$ of the unrestricted wreath
  product $G\wrwr\Z$. Denote by $u$ a generator of $\Z$, and by
  $\{b_1,b_2,\dots\}$ a generating set of $G$. Choose a sparse-enough
  sequence of elements $x_1,x_2,\dots$ of $\Z$, and define $f\colon
  \Z\to G$ by $f(x_i)=b_i$, all other values being trivial. The group
  $W$ is then $W=\langle f,u\rangle$.

  Let us spell out below what it means to be ``sparse enough''. We
  write $\Z$ additively. Since in fact by
  Proposition~\ref{prop:neumann-n:1} we only need to imbed $[t,G]$ in
  $W$, we may set $f(0)=t$ and a sufficient condition on the
  $x_i\in\Z$ is that $x_i\neq0$ for all $i$; all $x_i$ are distinct;
  and $x_i+x_j\not\in\{0,x_k\}$ for all $i,j,k\in\N$. One then sees
  that $[f,f^{u^{-x_i}}]$ is a function supported only at $0$, with
  value $[t,b_i]$ there. This defines the imbedding of $[t,G]$ in $W$.
\end{proof}

Note that the group $W$ contains the standard wreath product $B\wr\Z$,
so always has exponential growth.

The above construction shows that every countable solvable group
imbeds in a $2$-generated solvable group.

\begin{exse}
  Show first that \emph{not} every countable nilpotent group imbeds
  in a $2$-generated nilpotent group.

  Show then that every finitely generated nilpotent group imbeds in
  a $2$-generated nilpotent group.
\end{exse}

\noindent Here is a useful, small improvement on Theorem~\ref{thm:hnn}:
\begin{thm}\label{thm:hnntorsion}
  Let $p\ge5$ be an integer. Then every countable group imbeds in a
  $2$-generated group both of whose generators have order $p$.
\end{thm}
\begin{proof}
  Piggybacking on Propositions~\ref{prop:neumann-n:1}
  and~\ref{prop:neumann-n:2}, it suffices to consider a $2$-generated
  group $G=\langle x,y\rangle$, and to imbed $[G,G]$ into a
  $2$-generated group $W$ with generators of order $p$. Write
  $C_p=\langle t|t^p\rangle$, and define $f\colon C_p\to G$ by
  \[f(1)=x,\qquad f(t^{-1})=y,\qquad f(t^2)=y^{-1}x^{-1},\qquad f(g)=1\text{ for all other }g\in C_p.
  \]
  Consider then the group $W=\langle t,ft\rangle$. It is clearly
  generated by two elements of order $p$, since in $(ft)^p$ all the
  co\"ordinates contain some cyclic permutation of the product $x\cdot
  y\cdot y^{-1}x^{-1}$. Now $W$ contains $f$, so it also contains
  $[f,f^t]$, which is the function $C_p\to G$ taking value $[x,y]$ at
  $1$ and $1$ elsewhere. Furthermore conjugating $[f,f^t]$ by an
  arbitrary word in $f$ and $f^t$, one obtains the function taking
  value an arbitrary conjugate of $[x,y]$ at $1$; so $W$ contains
  $[G,G]@1$.
\end{proof}

\begin{exse}
  Where have we used the assumption that $p\ge5$? Can you improve the
  above result to arbitrary $p\ge3$? Can you imbed any countable group
  into a group generated by an involution and an element of order $p$?
\end{exse}

\subsection{Finite-valued permutational wreath products}
Our goal is, starting from a countable group $B$ locally of
subexponential growth, to construct a finitely generated group $W$ of
subexponential growth containing $B$. We take inspiration from
Neumann's proof given above, with two modifications: first, we
consider permutational wreath products rather than standard ones;
secondly, we consider \emph{finite-valued permutational wreath
  products}:
\begin{defn}
  Let $H$ be a group acting on a set $X$, and let $H$ be a
  group. Their \emph{finite-valued permutational wreath product} is
  the group $H\wrf_X G$, defined as the extension of functions $X\to
  H$ with finite image by $G$:
  \[H\wrf_X G=\{(\phi,g)\in H^X\times G\colon\#\phi(X)<\infty\}.\]
\end{defn}
Note that it is a subgroup, because if
$(\phi,g)^{-1}(\phi',g')=(\phi'',g^{-1}g')$ then
$\phi''(X)\subseteq\phi(X)^{-1}\phi'(X)$ is finite. Clearly, we have
\[H\wr_X G\le H\wrf_X G\le H\wrwr G.\]
We also introduce a condition on imbeddings that guarantees control on
growth:
\begin{defn}
  Let $B$ be a group. A group $G$ is called \emph{hyper-$B$} if it is
  a directed union of finite extensions of finite powers of $B$.
\end{defn}
Clearly, if a group $B$ is locally of subexponential growth and a
group $G$ is hyper-$B$, then $G$ is also locally of subexponential
growth. Indeed, for every finite subset $S$ of $G$ there exists a
finite extension of a finite power of $B$ that contains $S$.

\begin{lem}\label{lem:hyperhyper}
  Let $G$ be a hyper-$B$ group, and let $H$ be a hyper-$G$ group. Then
  $H$ is hyper-$B$.
\end{lem}
\begin{proof}
  Consider $h\in H$; then $h$ belongs to a finite extension of a
  finite power of $G$, which may be assumed of the form $G\wr F$ for a
  finite group $F$. Let us write $h=\phi f$ with $\phi\colon F\to G$
  and $f\in F$; then $\phi(f)$ belongs for all $f\in F$ to a finite
  extension of a finite power of $B$, which can be assumed to be the
  same for all $f$. This extension may be assumed to be of the form
  $B\wr E$ for a finite group $E$. It follows that $h$ belongs to
  $B\wr_{E\times F}(E\wr F)$, a finite extension of a finite power of
  $B$; so $H$ is hyper-$B$.
\end{proof}

\begin{lem}\label{lem:hyperB}
  If $H$ is a hyper-$B$ group and $U$ is locally finite, then $H\wrf
  U$ is a hyper-$B$ group.
\end{lem}
\begin{proof}
  We first show that $H\wrf U$ is hyper-$H$. By hypothesis, $U$ is a
  directed union of finite subgroups $E$. The partitions $\mathscr
  P_0$ of $U$ into finitely many parts also form a directed poset; and
  for every such partition $\mathscr P_0$ and every finite subgroup $E\le U$
  there exists a finite partition $\mathscr P$ of $U$ that is
  invariant under $E$ and refines $\mathscr P_0$, namely the wedge (=
  least upper bound) of all $E$-images of $\mathscr P_0$.

  Consider now the directed poset of pairs $(E,\mathscr P)$ consisting
  of finite subgroups $E\le U$ and $E$-invariant partitions of $U$.
  Consider the corresponding subgroups $H^{\mathscr P}\rtimes E$ of
  $H\wrf U$. If $(E,\mathscr P)\le(E',\mathscr P')$ then $H^{\mathscr
    P}\rtimes E$ is naturally contained in $H^{\mathscr P'}\rtimes
  E'$, so these subgroups of $H\wrf U$ form a directed poset, which
  exhausts $H\wrf U$.

  It follows that $H\wrf U$ is a hyper-$H$ group, and we are done by
  Lemma~\ref{lem:hyperhyper}.
\end{proof}

\subsection{Imbedding in the derived subgroup}\label{ss:imbedG'}
Our main goal, in this section, is to prove the following proposition,
which replaces Proposition~\ref{prop:neumann-n:1}.
\begin{prop}\label{prop:imbed'}
  Let $B$ be a group. Then there exists a hyper-$B$ group $G$ such
  that $[G,G]$ contains $B$ as a subgroup.

  If $B$ is infinite, then $G$ may furthermore be supposed to have the
  same cardinality as $B$.
\end{prop}

\begin{lem}\label{lem:CB}
  Let $B$ be a group. Then there exists a subgroup $C$ of $B$,
  containing $[B,B]$, such that $B/C$ is torsion and $C/[B,B]$ is free
  abelian.
\end{lem}
\begin{proof}
  $B/[B,B]\otimes_\Z\Q$ is a $\Q$-vector space, hence has a basis,
  call it $X$. It generates a free abelian group $\Z X$ within
  $B/[B,B]$, whose full preimage in $B$ we call $C$. Then
  $B/C\otimes_\Z\Q=0$ so $B/C$ is torsion.
\end{proof}

We set up the following notation for the proof of
Proposition~\ref{prop:imbed'}. We choose a subgroup $C\le B$ as in
Lemma~\ref{lem:CB} and write $T:=B/C$. We choose a basis $X$ of
$C/[B,B]$, for every $x\in X$ we choose an element $b_x\in C$
representing it, and we define a homomorphism $\theta_x\colon
C\to\langle b_x\rangle\subseteq B$, trivial on $[B,B]$, by
$\theta_x(b_x)=b_x^{-1}$ and $\theta_x(b_y)=1$ for all $y\neq x\in
X$. In particular, we have for all $b\in C$
\[b\cdot\prod_{x\in X}\theta_x(b)\in[B,B]
\]
and the product is finite.

We write $\pi\colon B\to T$ the natural projection, and define a
section $\sigma$ of $\pi$ with the following property, which we single
out as a lemma:
\begin{lem}\label{lem:section}
  There exists a set-theoretic section $\sigma\colon T\to B$ such
  that, for every $t\in T$, the subset
  $\{\sigma(t u)\sigma(u)^{-1}:u\in T\}[B,B]$ of $B/[B,B]$ is finite.
\end{lem}
\begin{proof}
  Since every abelian torsion group is the direct sum of its
  $p$-subgroups, we may first define the section on each of $T$'s
  $p$-components, and then extend it to $T$ multiplicatively (in any
  order).

  We therefore suppose that $T$ is a $p$-group.  Recall the notation
  $\Omega_n(T)=\{t\in T:t^{p^n}=1\}$. Each quotient
  $V_n:=\Omega_n(T)/\Omega_{n-1}(T)$ is a vector space over $\mathbb F_p$,
  and the homomorphism $t\mapsto t^p$ induces an injective linear map
  $V_n\to V_{n-1}$. Choose inductively subsets $X_1,X_2,\dots$ of $T$
  such that $X_n$ maps to a basis of $V_n$ and such that $t\mapsto
  t^p$ induces an injective map $X_n\to X_{n-1}$. Set $X=\bigcup_n
  X_n$, and give an arbitrary total order to $X$. Choose for each
  $x\in X$ a $\pi$-preimage $\sigma(x)\in B$.

  Since $T$ is torsion, every element $t\in T$ belongs to
  $\Omega_n(T)$ for some $n\in\N$, so can be uniquely written as a
  product $t=x_1^{\alpha_1}\cdots x_n^{\alpha_n}$ with ordered
  $x_i\in X$ and $0<\alpha_i<p$ for all $i$. Extend then $\sigma$ by
  $\sigma(t)=\sigma(x_1)^{\alpha_1}\cdots \sigma(x_n)^{\alpha_n}$.

  Consider now $t\in T$, and write it in the form
  $t=x_1^{\alpha_1}\cdots x_n^{\alpha_n}$ as above. Extend
  $\{x_1,\dots,x_n\}$ to a finite subset
  $Y=\{x_1,\dots,x_n,\dots,x_s\}$ of $X$, by adding all $p^j$-th
  powers of all $x_i$ to it. The set
  \[T':=\{x_1^{\gamma_1}\cdots x_s^{\gamma_s}:0\le\gamma_i<p\}
  \]
  is then a finite subgroup of $T$. Consider next $u\in T$, and note
  that it may be written uniquely in the form $u=y_1^{\beta_1}\cdots
  y_m^{\beta_m}z$ with $y_i\in X\setminus Y$ and $z\in T'$. Then
  $t u=y_1^{\beta_1}\cdots y_m^{\beta_m}(z t)$ and this representation
  is unique; so $\sigma(t u)\sigma(u)^{-1}[B,B]$ belongs to the finite
  set $\{\sigma(t')[B,B]:t'\in T'\}$.
\end{proof}

\begin{exse}
  Rephrase Lemma~\ref{lem:section} in terms of the cohomology of $T$
  with co\"efficients in $\Z X$.
\end{exse}

Let $F$ be a locally finite group of cardinality $>\#X$, and fix an
imbedding of $X$ in $F\setminus\{1\}$.  As a first step, we consider
the group $G_0=B\wrf(T\times F)$, and define a map $\Phi_0\colon B\to
G_0$ as follows:
\begin{equation}\label{eq:Phi_0}
  \Phi_0(b)=(\phi,\pi(b),1)\text{ with }\phi(t,f)=\begin{cases}
    b & \text{ if }f=1,\\
    \theta_f(\sigma(t)b\sigma(t\pi(b))^{-1}) & \text{ if }f\in X,\\
    1 & \text{ otherwise.}
  \end{cases}
\end{equation}

\begin{lem}\label{lem:Phi_0}
  The map $\Phi_0$ is well-defined and is an injective homomorphism.
\end{lem}
\begin{proof}
  To see that $\Phi_0$ is well-defined, note that the argument
  $\sigma(t)b\sigma(t\pi(b))^{-1}$ belongs to $\ker(\pi)=C$, so that
  $\theta_f$ may be applied to it. Note then that, by
  Lemma~\ref{lem:section}, the expression
  $\sigma(t)b\sigma(t\pi(b))^{-1}$ takes finitely many values in
  $C/[B,B]$, so that $\phi(t,f)$ takes finitely many values for
  varying $t$ and fixed $f$. Finally,
  $\theta_f(\sigma(t)b\sigma(t\pi(b))^{-1})=1$ except for finitely
  many values of $f\in X$. In summary, the function $\phi\in
  B^{T\times F}$ is such that $\phi(t,f)$ takes only finitely many
  values.

  It is clear that $\Phi_0$ is injective: if $b\neq1$ and
  $\Phi_0(b)=(\phi,\pi(b),1)$ then $\phi(1,1)=b\neq1$. It is a
  homomorphism because all $\theta_f$ are homomorphisms.
\end{proof}

\begin{lem}\label{lem:C imbeds}
  We have $\Phi_0(C)\le[G_0,G_0]$.
\end{lem}
\begin{proof}
  If $b\in[B,B]$ then clearly $\Phi_0(b)\in[G_0,G_0]$. Since $C$ is
  generated by $[B,B]\cup\{b_x\}_{x\in X}$, it suffices to consider
  $b=b_x$.

  We define $g\in G_0$ by
  \[g=(\psi,1,1)\text{ with }\psi(t,f)=\begin{cases}b_x & \text{ if }f=1,\\ 1 & \text{ otherwise}.\end{cases}
  \]
  Then $\Phi_0(b_x)=(\phi,1,1)$ with $\phi(t,1)=b_x$ and
  $\phi(t,x)=b_x^{-1}$, all other values being trivial, according
  to~\eqref{eq:Phi_0}; so, as was to be shown,
  \[\Phi_0(b_x)=(\phi,1,1)=({}^x\!\psi^{-1}\cdot\psi,1,1)=[(1,1,x^{-1}),g]\in[G_0,G_0].\qedhere\]
\end{proof}

\noindent We finally define
\[G=G_0\wrf(\Q/\Z)\]
and a map $\Phi\colon B\to G$ by
\[\Phi(b)=(\phi,0)\text{ with }\phi(r)=\Phi_0(b)\text{ for all }r\in\Q/\Z.\]
\begin{lem}\label{lem:B imbeds}
  The map $\Phi$ is an injective homomorphism, and
  $\Phi(B)\le[G,G]$.
\end{lem}
\begin{proof}
  Clearly $\Phi$ is an injective homomorphism, since $\Phi_0$ is an
  injective homomorphism by Lemma~\ref{lem:Phi_0}.

  We identify $\Q/\Z$ with $\Q\cap[0,1)$. For every $n\in\N$, consider
  the map $\Psi_n\colon B\to G$ defined by
  \[\Psi_n(b)=(\phi,0)\text{ with }\phi(r)=\begin{cases}
    \Phi_0(b) & \text{ if }r\in[0,1/n),\\
    1 & \text{ otherwise;}
  \end{cases}
  \]
  so $\Phi=\Psi_1$.  We know from Lemma~\ref{lem:C imbeds} that
  $\Psi_n(C)$ is contained in $[G,G]$.

  Consider now $b\in B$. Since $B/C$ is torsion, there exists $n\in\N$
  such that $b^n\in C$. We define $g\in G$ by
  \[g=(\psi,0)\text{ with }\psi(r)=\Phi_0(b)^{\lfloor r n\rfloor}\text{
    for }r\in[0,1)\cap\Q.
  \]
  Let us write $h=\Phi_0(b)$, and consider the element
  $[(1,1/n),g]\cdot\Psi_n(b^n)=(\phi,0)$.  If $r\in[0,1/n)$ then
  $\phi(r)=\psi(r-1/n)^{-1}\psi(r)h^n=h$, while if $r\in[1/n,1)$ then
  $\phi(r)=\psi(r-1/n)^{-1}\psi(r)=h$; therefore
  \[\Phi(b)=[(1,1/n),g]\cdot\Psi_n(b^n)\in[G,G].\qedhere\]
\end{proof}

\begin{proof}[Proof of Proposition~\ref{prop:imbed'}]
  The first assertion is simply Lemma~\ref{lem:B imbeds}.  For the
  last one: if $B$ is infinite, we wish to find a subgroup $H$ of $G$
  with the same cardinality as $B$, such that $\Phi$ maps into
  $[H,H]$. For each $b\in B$, choose a finite subset $S_b$ of $G$ such
  that $\Phi(b)\in[\langle S_b\rangle,\langle S_b\rangle]$, and a
  subgroup $G_b$, containing $S_b$, that is virtually a finite power
  of $B$. Consider the group $H$ generated by the union of all the
  $G_b$. As soon as $B$ is infinite, all $G_b$ have the same
  cardinality as $B$, and so does $H$.
\end{proof}

\subsection{Spreading, stabilizing, rectifiable sequences}
We also extend Proposition~\ref{prop:neumann-n:2} by replacing the
wreath product $G\wrwr\Z$ by permutational wreath products of the form
$G\wrwr_X P$ for a group $P$ acting on a set $X$, and considering
subgroups of the form $W=\langle f,P\rangle$ for a function
$f\colon X\to G$.

We already encountered in Proposition~\ref{prop:growthW} sufficient
conditions for such a group $W$ to be of subexponential growth, in
case $f$ is finitely generated. As we shall see, the group $W$ may
also have subexponential growth if $f$ is \emph{infinitely} supported,
but its support is sufficiently sparse, in a sense that we describe
now.

In this section, we assume a finitely generated group
$P=\langle S\rangle$ acting on the right on a set $X$ has been fixed.
We use the same notation for $X$ as a set and as a \emph{Schreier
  graph}, namely as the graph with vertex set $X$, and with for all
$x\in X,s\in S$ an edge labelled $s$ from $x$ to $x s$, see
Definition~\ref{def:schreier}. We denote by $d$ the path metric on
this graph.

\begin{defn}\label{def:spreading}
  A sequence $(x_0,x_1,\dots)$ in $X$ is \emph{spreading} if for all
  $R$ there exists $N$ such that if $i,j\ge N$ and $i\neq j$ then
  $d(x_i,x_j)\ge R$.
\end{defn}

\begin{exple}
  If all $x_i$ lie in order on a geodesic ray starting from $x_0$ (for
  example if $X$ itself is a ray starting from $x_0$) and for all $i$
  we have $d(x_0,x_{i+1})\ge 2 d(x_0,x_i)$, then $(x_i)$ is spreading.
\end{exple}

\begin{exse}
  Show that a sequence $(x_0,x_1,\dots)$ in $X$ is spreading if and
  only if for all $R$ there exists $N$ such that if $i\neq j$ and
  $i\ge N$ then $d(x_i,x_j)\ge R$.
\end{exse}

\begin{defn}\label{def:locally stabilizing}
  A sequence $(x_i)$ in $X$ \emph{locally stabilises} if for all $R$
  there exists $N$ such that if $i,j\ge N$ then the $S$-labelled
  radius-$R$ balls centered at $x_i$ and $x_j$ in $X$ are isomorphic
  as labelled graphs.
\end{defn}

\begin{defn}\label{def:rectifiable}
  A sequence of points $(x_i)$ in $X$ is \emph{rectifiable} if for all
  $i,j$ there exists $g \in P$ with $x_i g =x_j$ and $x_k g\ne x_\ell$
  for all $k\notin\{i,\ell\}$.
\end{defn}

For example, if $X=\Z$ and $P=\Z$ acting by translations, then
$\Sigma=\{2^i:i\in\N\}$ is rectifiable, since $2^j-2^i=2^\ell-2^k$
only has trivial solutions $i=k,j=\ell$ and $i=j,k=\ell$. It is also
spreading and locally stabilizing.

\begin{exse}
  Show that the sequence $\Sigma=(x_i)\subseteq X$ is rectifiable if
  and only if for all $i,j$ there exists $g\in P$ with $x_i g = x_j$
  and $\Sigma\cap\Sigma
  g\subseteq\{x_j\}\cup\operatorname{fixed.\!points}(g)$.
\end{exse}

\begin{defn}
  Fix a point $z\in X$. A sequence $(g_i)$ in $P$ is
  \emph{parallelogram-free at $z$} if, for all $i,j,k,\ell$ with $i\neq
  j$ and $j\neq k$ and $k\neq\ell$ and $\ell\neq i$ one has
  $z g_i^{-1}g_j g_k^{-1} g_\ell\ne z$.
\end{defn}
\[\begin{tikzpicture}[decoration={markings,mark=at position 0.5 with {\arrow{>}}}] 
  \node[inner sep=0] (z) at (0,0) [label=right:$z$] {$\bullet$};
  \node[inner sep=0] (zi) at (-3,-1) [label=left:$z g_i^{-1}$] {$\bullet$};
  \node[inner sep=0] (zij) at (-4.5,0.5) {$\bullet$};
  \node[inner sep=0] (zl) at (-1.5,1.5) {$\bullet$};
  \node[inner sep=0] (zx) at (-0.2,0.2) {$\bullet$};
  \draw[postaction={decorate}] (zi) -- node[below] {$i$} (z);
  \draw[postaction={decorate}] (zi) -- node[below left] {$j$} (zij);
  \draw[postaction={decorate}] (zl) -- node[above] {$k$} (zij);
  \draw[postaction={decorate}] (zl) -- node[above right] {$\ell$} (zx);
\end{tikzpicture}\]

\begin{lem}\label{lem:parallelogram}
  If $z\in X$ and $(g_i)$ is parallelogram-free at $z$, then $(z
  g_i^{-1})$ is a rectifiable sequence in $X$.
\end{lem}
\begin{proof}
  Set $x_i=z g_i^{-1}$ for all $i\in\N$. Given $i,j\in\N$, consider
  $g=g_i g_j^{-1}$, so $x_i g = x_j$. If furthermore we have $x_k
  g=x_\ell$, then we have $z g_k^{-1} g_i g_j^{-1} g_\ell = z$, so
  either $k=i$, or $i=j$ which implies $k=\ell$, or $j=\ell$ which
  implies $k=i$, or $\ell=k$. In all cases $k\in\{i,\ell\}$ as was to
  be shown.
\end{proof}

It is clear that, if $P$ is finitely generated and $X$ is infinite,
then it admits spreading and locally stabilizing sequences. Indeed
every sequence contains a spreading subsequence, and every sequence
contains a stabilizing subsequence, and a subsequence of a spreading
or locally stabilizing sequence is again spreading, respectively
locally stabilizing.

We will content ourselves with the following rectifiable sequences,
based on the Grigorchuk groups $G_\omega=\langle a,b,c,d\rangle$, with
$\omega\in\Omega'$, for example the first Grigorchuk group $G_{012}$.
Recall the description of $G_\omega$ from~\S\ref{ss:grigorchuk}, and
in particular the Schreier graph of its action on $\8^\infty$
in~\eqref{eq:grigorbit}. We construct explicitly a spreading, locally
stabilizing, rectifiable sequence for the action of $G_\omega$ on $X$:
for all $i\in\N$, let us define
\[x_i=\8^\infty\9^i,
\]
the point at distance $2^i$ from the origin on the Schreier graph.

\begin{lem}\label{lem:Gspreading}
  For all $\omega\in\Omega$ containing infinitely many $0,1,2$'s, and
  for all $i,j\in\N$,
  \begin{enumerate}
  \item the marked balls of radius $2^{\min(i,j)}$ in $X$ around $x_i$
    and $x_j$ coincide;
  \item the distance $d(x_i,x_j)$ is $|2^i-2^j|$;
  \item there exists $g_{i,j}\in G_\omega$ of length $|2^i-2^j|$ with
    $x_i g_{i,j} = x_j$ and $x_k g_{i,j}\neq x_\ell$ for all
    $(k,\ell)\neq (i,j)$.
  \end{enumerate}
\end{lem}

\begin{proof}
  (1), (2) Consider the map $\theta_\omega$ from
  Equation~\eqref{eq:theta}. It maps the stabilizer of $\8^\infty$ in
  $G_{\sigma\omega}$ to the stabilizer of $\8^\infty$ in $G_\omega$,
  and therefore defines a self-map of $X$ by sending $\8^\infty g$ to
  $\8^\infty\theta_\omega(g)$. A direct calculation shows that it
  sends $x\in X$ to $x\9$.

  Since $\theta_\omega$ is $2$-Lipschitz on words of even length in
  $\{a,b,c,d\}$, it maps the ball of radius $n$ around $x$ to the ball
  of radius $2n$ around $x\9$. Its image is in fact a net in the ball
  of radius $2n$: two points at distance $1$ in the ball of radius $n$
  around $x$ will be mapped to points at distance $1$ or $3$ in the
  image, connected either by a segment
  \tikz[xscale=0.68,baseline]{\footnotesize \path (0,0) edge
    node[above] {$a$} (1,0);} or by a path
  \tikz[xscale=0.68,baseline]{\footnotesize \path (0,0) edge
    node[above] {$a$} (1,0); \path (1,0) edge[bend left] node[above]
    {$x$} (2,0); \path (1,0) edge[bend right] node[below] {$y$} (2,0);
    \path (2,0) edge node[above] {$a$} (3,0);} for some
  $\{x,y\}\subset\{b,c,d\}$. In particular, the $2^n$-neighbourhoods
  of the balls about the $x_m$ coincide for all $m\ge n$.

  (3) Note, first, that there exists $g_{i,j}$ with $x_i g_{i,j} =
  x_j$, because the rays ending in $\8^\infty$ form a single
  orbit. Note, also, that we have $x_k g_{i,j} = x_\ell$ for either
  finitely many $(k,\ell)\neq(i,j)$ or for all but finitely many
  $(k,\ell)$, because there is a level $N$ at which the decomposition
  of $g_{i,j}$ consists entirely of generators; if the entry at $\9^N$
  of $g_{i,j}$ is trivial or `$d$' then all but finitely many of the
  $x_k$ are fixed; while otherwise (up to increasing $N$ by at most
  one) we may assume it is an `$a$'; then $\9^{N+1} g_{i,j}=\9^N\8$,
  so $x_k\neq x_\ell$ for all $k>N+1$.

  We use the following property of the Grigorchuk groups $G_\omega$:
  for every finite sequence $u\in\{\9,\8\}^*$ there exists an element
  $h_u\in G_\omega$ whose fixed points are precisely those sequences
  in $\{\9,\8\}^\infty$ that do not start with $u$. One may take for
  $h_u$ the element $[a,b]$, $[a,c]$ or $[a,d]$ inside the copy of the
  branching subgroup of $G_\omega$ that acts on $\{\9,\8\}^\infty u$,
  see Proposition~\ref{prop:Gbranched}.

  If the entry at $\9^N$ of $g_{i,j}$ is trivial, then we multiply
  $g_{i,j}$ with $h_{\9^M}$ for some $M>\max(N,i)$, so as to fall back
  to the second case.

  Then, for each pair $(k,\ell)\neq(i,j)$ with $x_k g_{i,j}=x_\ell$,
  we multiply $g_{i,j}$ with $h_{\8\9^\ell}$, so as to destroy the
  relation $x_k g_{i,j}=x_\ell$.

  The resulting element $g_{i,j}$ satisfies the required conditions.
\end{proof}

\subsection{Subexponential growth of wreath products}
The next step in the proof is an argument controlling the growth of a
subgroup of the form $W=\langle P,f\rangle\le G\wrwr_X P$, for a
function $f\colon X\to G$ with sparse-enough (but infinite!)
support.

We select a finitely generated group $P$ of subexponential growth
acting on a set $X$ with sublinear inverted orbit growth (see
Definition~\ref{def:invorbit}; recall, from
Exercise~\ref{exse:invgrowth}, that the property of having sublinear
orbit growth is independent of a choice of basepoint) and
subexponential inverted orbit choice growth. These are the hypotheses
for Corollary~\ref{cor:growthW}, which guarantee that $G\wr_X P$ has
subexponential growth as soon as $G$ has subexponential growth. The
main example we have in mind is a Grigorchuk group $P=G_\omega$ acting
on $X=\8^\infty P$, for $\omega\in\{0,1,2\}^\infty$ containing
infinitely many times each symbol, see Lemma~\ref{lem:Gspreading}.

We also assume that there are rectifiable sequences in $X$, and (using
the results of the previous section) we fix a rectifiable, spreading,
locally stabilizing sequence $(x_0,x_1,\dots)$ of elements of $X$.

Finally, we fix a countable group $G$, and a finite or infinite
sequence of elements $(b_1,b_2,\dots)$ generating $G$.

\begin{defn}
  For an increasing finite or infinite sequence $0\le n(1) < n(2)
  <\dots$ of integers, define $f\colon X\to G$ by
  \[f(x_{n(1)})=b_1,\qquad f(x_{n(2)})=b_2,\qquad \dots,\qquad f(x)=1\text{ for other }x.
  \]
  The group $W_{(n)}=W_{n(1),n(2),\dots}$ is then defined as the
  subgroup $\langle P,f\rangle$ of the unrestricted wreath product
  $G\wrwr_X P$.
\end{defn}

If all but finitely many $b_i$ are trivial, then $W_{(n)}$ has
subexponential growth since then $G$ is finitely generated and
$W_{(n)}$ is a subgroup of $G\wr_X P$. However, if $(n(i))$ is sparse
enough then $W_{(n)}$ may have subexponential growth even if $f$ has
infinite support:
\begin{prop}\label{prop:Wsubexp}
  If $G$ has locally subexponential growth, then there exists an
  infinite sequence $(n(i))$ such that the group $W_{(n)}$ has
  subexponential growth.
\end{prop}

The proof of Proposition~\ref{prop:Wsubexp} follows from a stronger,
and independently interesting, statement: arbitrarily large balls in
$W_{(n)}$ are approximable by groups of the form
$W_{(n(1),\dots,n(i))}$, which have subexponential growth by the remark above:
\begin{prop}\label{prop:Wconvergence}
  Assume that the sequence $(x_i)$ in $X$ is spreading and locally
  stabilizing, and that all elements $b_i$ have the same order.

  Then for every increasing sequence $(m(i))$ there exists an
  increasing sequence $(n(i))$ with the following property: the ball
  of radius $m(i)$ in $W_{(n)}$ coincides with the ball of radius
  $m(i)$ in $W_{n(1),n(2),\dots,n(i)}$, via the natural identification
  $P\leftrightarrow P,f\leftrightarrow f$ between $W_{(n)}$ and
  $W_{n(1),\dots,n(i)}$.

  Furthermore, the term $n(i)$ depends only on the previous terms
  $n(1),\dots,n(i-1)$, on the initial terms $m(1),\dots,m(i)$, and on
  the ball of radius $m(i)$ in the subgroup
  $\langle b_1,\dots,b_{i-1}\rangle$ of $G$.
\end{prop}
\begin{proof}
  Choose $n(i)$ such that $d(x_j,x_k)\ge m(i)$ for all $j\neq k$ with
  $k\ge n(i)$, and such that the balls of radius $m(i)$ around
  $x_{n(i)}$ and $x_j$ coincide for all $j>n(i)$.

  Consider then an element $h\in W_{(n)}$ in the ball of radius
  $m(i)$, and write it in the form $h=(c,g)$ with $c\colon X\to G$ and
  $g\in P$. The function $c$ is a product of conjugates of $f$ by
  words of length $<m(i)$. Its support is therefore contained in the
  union of balls of radius $m(i)-1$ around the $x_j$, with $j$ either
  $\ge n(i)$ or of the form $n(k)$ for $k<i$. In particular, the
  entries of $c$ are in $\langle
  b_1,\dots,b_{i-1}\rangle\cup\bigcup_{j\ge i}\langle b_j\rangle$. For
  $j>n(i)$, the restriction of $c$ to the ball around $x_j$ is
  determined by the restriction of $c$ to the ball around $x_{n(i)}$,
  via the identification $b_i\mapsto b_j$, because the neighbourhoods
  in $X$ coincide and all cyclic groups $\langle b_j\rangle$ are
  isomorphic.

  It follows that the element $h\in W$ is uniquely determined by the
  corresponding element in $W_{n(1),\dots,n(i)}$.
\end{proof}

\begin{proof}[Proof of Proposition~\ref{prop:Wsubexp}]
  Let $Z=\langle z\rangle$ be a cyclic group whose order (possibly
  $\infty$) is divisible by the order of all the $b_i$'s.  We replace
  $G$ by $G\times Z$ and each $b_i$ by $b_i z$, so as to guarantee
  that all generators in $G$ have the same order.

  Let $\epsilon_i$ be a decreasing sequence tending to $1$. We now
  construct a sequence $m(i)$ inductively, and obtain the sequence
  $n(i)$ by Proposition~\ref{prop:Wconvergence}, making always sure
  that $m(i)$ depends only on $m(j),n(j)$ with $j<i$.

  Denote by $v_i$ the growth function of the group
  $W_{n(1),\dots,n(i)}$. Since the group $W_{n(1),\dots,n(i)}$ is
  contained in $G\wr_X P$, it has subexponential growth. Therefore,
  there exists $m(i)$ be such that
  \[v_i(m(i))\le \epsilon_i^{m(i)}.
  \]
  By Proposition~\ref{prop:Wconvergence}, the terms
  $n(i+1),n(i+2),\dots$ can be chosen in such a manner that the balls
  of radius $m(i)$ coincide in $W_{(n)}$ and $W_{n(1),\dots,n(i)}$.

  Denote now by $w$ the growth function of $W_{(n)}$. We then have
  $w(m(i))\le\epsilon_i^{m(i)}$ for all $i\in\N$. Therefore,
  \[w(R)\le\epsilon_i^{R+m(i)}\text{ for all }R>m(i),\]
  so $\limsup\sqrt[R]{w(R)}\le\epsilon_i$ for all $i\in\N$. Thus the
  growth of $W_{(n)}$ is subexponential.
\end{proof}

Finally, the rectifiability of the sequence $(x_i)$ guarantees that
functions with singleton support and arbitrary values in $[G,G]$
belong to $W_{(n)}$ for all sequences $n$:
\begin{lem}\label{lem:contains [G,G]}
  If the sequence $(x_i)$ is rectifiable, then $[W_{(n)},W_{(n)}]$
  contains $[G,G]$ as a subgroup for all choices of
  $n=n(1)<n(2)<\cdots$.
\end{lem}
\begin{proof}
  We denote by $\iota\colon G\to G^X\rtimes P$ the imbedding of $G$
  mapping the element $b\in G$ to the function $X\to G$ with value $b$
  at $x_0$ and $1$ elsewhere. We abbreviate $W=W_{(n)}$. We shall show
  that $[W,W]$ contains $\iota([G,G])$. For this, denote by $H$ the
  subgroup $\iota([G,G])\cap[W,W]$.

  We first consider an elementary commutator $g=[b_i,b_j]$. Let
  $g_i,g_j\in P$ respectively map $x_i,x_j$ to $x_0$, and be such that
  $g_i g_j^{-1}$ maps no $x_k$ to $x_\ell$ with $k\neq\ell$, except
  for $x_i g_i g_j^{-1}=x_j$. Consider $[f^{g_i},f^{g_j}]\in [W,W]$;
  it belongs to $G^X$, and has value $[b_i,b_j]$ at $x_0$ and is
  trivial elsewhere, so equals $\iota(g)$ and therefore $\iota(g)\in
  H$.

  We next show that $H$ is normal in $G^X$. For this, consider $h\in
  H$. It suffices to show that $h^{\iota(b_i)}$ belongs to $H$ for all
  $i$. Now $h^{\iota(b_i)}=h^{f^{g_i}}$ belongs to $H$, and we are
  done.
\end{proof}

We are now ready to complete the proof of Theorem~\ref{thm:imbed}.  By
Proposition~\ref{prop:imbed'}, the countable, locally subexponentially
growing group $B$ imbeds in $[G,G]$ for a countable, locally
subexponentially growing group $G$. Let $(b_1,b_2,\dots)$ be a
generating set for $G$. By Proposition~\ref{prop:Wsubexp}, there
exists an increasing sequence $(n(i))$ such that the group $W=W_{(n)}$
has subexponential growth. By Lemma~\ref{lem:contains [G,G]}, $[G,G]$
imbeds in $[W,W]$, so $B$ imbeds in $[W,W]$ and we are done.

\begin{exse}
  Give examples of sequences $(x_i)$ that are only spreading, or only
  stabilizing, and such that $W_{(n)}$ has exponential growth, even
  when the sequence $(n(i))$ grows arbitrarily fast.
\end{exse}

\section{Groups of non-uniform exponential growth}\label{ss:nug}
Recall that $v_{G,S}(R)$ denotes the growth function of a group $G$
generated by a finite set $S$. The \emph{volume entropy} of $(G,S)$ is
\[\lambda_{G,S}:=\lim_{R\to\infty}\frac{\log v_{G,S}(R)}R.\]
The limit exists because the function $v_{G,S}$ is submultiplicative
(namely, $v_{G,S}(R_1+R_2)\le v_{G,S}(R_1) v_{G,S}(R_2)$.) Indeed,
apply the following lemma to the sequence $(\log v_{G,S}(n))_{n\ge1}$:
\begin{lem}[Fekete~\cite{fekete:verteilung}]
  Let $(a_n)$ be a subadditive sequence: $a_{n+m}\le a_n+a_m$ for all
  $m,n\ge1$. Then $\lim a_n/n$ exists and equals $\inf a_n/n$.
\end{lem}
\begin{proof}
  Set $A=\inf a_n/n$ and consider any $B>A$. Choose $k\ge1$ such that
  $a_k/k<B$. Every $n\in\N$ may be written in the form $n=r k+s$ with
  $r\in\N$ and $s\in\{0,\dots,k-1\}$. Thus
  \[\frac{a_n}n=\frac{a_{r k+s}}n\le\frac{r a_k+a_s}n=\frac{a_k}k\frac{n-s}n+\frac{a_s}n,\]
  so $\limsup_{n\to\infty}a_n/n\le a_k/k\le B$. Since $B>A$ was
  arbitrary, $\limsup a_n/n\le L$ and we are done.
\end{proof}
Furthermore, the following are equivalent: $G$ has exponential word
growth; $\lambda_{G,S}>0$ for some generating set $S$; and
$\lambda_{G,S}>0$ for all generating sets $S$.

Let $G$ be a group of exponential growth. Note that, even though
$\lambda_{G,S}>0$ for all $S$, one might have $\lambda_{G,S_i}\to0$
along a sequence of generating sets $S_i$. It is easy to see that this
cannot happen for $G$ a free group of rank $k\ge2$; indeed then each
$S_i$ contains a subset of cardinality $k$ generating a free subgroup,
so $\lambda_{G,S_i}\ge\log(2k-1)>0$ for all $i$.  Let us say that a
finitely generated group $G$ has \emph{uniform exponential growth} if
$\inf_S\lambda_{G,S}>0$, and \emph{non-uniform exponential growth} if
$\lambda_{G,S}>0$ for all $S$ yet $\inf_S\lambda_{G,S}=0$.

The existence of groups of non-uniform exponential growth is asked by
Gromov in~\cite{gromov:metriques}*{Remarque~5.12};
see~\cite{harpe:uniform} for a survey. There have been quite a few
positive results: Osin showed in~\cite{osin:entropy} that virtually
solvable groups have uniform exponential growth unless they are
virtually nilpotent; Eskin, Mozes and Oh obtained the same result
in~\cite{eskin-m-o:uniform} for finitely generated linear groups in
characteristic $0$; Koubi showed in~\cite{koubi:unifexpo} that
word-hyperbolic groups have exponential growth unless they are
virtually cyclic.

A \emph{Golod-Shafarevich group} is a residually-$p$ group such that
the associated Hopf algebra $\bigoplus_{n\ge0}\varpi^n/\varpi^{n+1}$
from~\S\ref{ss:algebralb} has exponential growth; equivalently, the
Lie algebra $\bigoplus_{n\ge1}\gamma_n(G)/\gamma_{n+1}(G)$
from~\eqref{eq:liealg} has exponential growth. Since by
Proposition~\ref{prop:alg-gp} the growth of $\overline{\Bbbk G}$ is
always a lower bound for the growth of $G$, such groups have uniformly
exponential growth.

Among groups $G$ of uniform exponential growth, one may ask whether
the infimal entropy $\inf_{\langle S\rangle=G}\lambda_{G,S}$ is
realised. Recall that a group $G$ is \emph{Hopfian} if it is not
isomorphic to a proper quotient of itself. Sambusetti proves
in~\cite{MR2000j:20056} that if $G$ is a free product $G=G_1*G_2$ with
$G_1$ non-Hopfian and $G_2$ non-trivial, then $\inf_{\langle
  S\rangle=G}\lambda_{G,S}$ is not attained.

It was widely suspected, since the appearance of groups of
intermediate growth, that examples of groups of non-uniform
exponential growth should exist. The first examples of groups of
non-uniform exponential growth were exhibited by Wilson,
see~\cite{wilson:ueg}. We now give a simple construction showing that
such groups abound:
\begin{thm}[\cite{bartholdi-erschler:orderongroups}*{Theorem~E}]\label{thm:nueg}
  Every countable group may be imbedded in a group of non-uniform
  exponential growth.

  Furthermore, let $\eta_+\approx2.46$ denote the positive root of the
  polynomial $T^3-T^2-2T-4$. Then the group $W$ in which the countable
  group imbeds may be required to have the following property: there
  is a constant $K$ such that, for all $R>0$, there exists a
  generating set $S$ of $W$ with
  \[v_{W,S}(r) \le \exp(K r^{\log2/\log\eta_+})\text{ for all }r\in[0,R].
  \]
\end{thm}
\begin{proof}
  Let $B$ be a countable group. By Theorem~\ref{thm:hnntorsion}, one
  may imbed $B$ into a group $G$ generated by two elements $s,t$ of
  order $5$. Without loss of generality, assume that $G$ has
  exponential growth (if needed, replace first $B$ by $B\times F_2$).

  The group $W$ in which $G$ imbeds is the wreath product
  $G\wr_X G_{012}$ of $G$ with the first Grigorchuk group.  We also
  consider $A=C_5\times C_5=\langle s',t'\rangle$, and the wreath
  product $W'=A\wr_X G_{012}$.

  We consider the points $x_0=\9^\infty$ and $x_i=\9^\infty\8\9^i$ for
  all $i\ge1$ in the Schreier graph $X$, and the generating sets
  $S_i=\{a,b,c,d,s@x_0,t@x_i\}$ of $W$ and
  $S'=\{a,b,c,d,s'@x_0,t'@x_0\}$ of $W'$.

  Note that the sequence $(x_i)$ is spreading and locally stabilizing
  (see Definitions~\ref{def:spreading} and~\ref{def:locally
    stabilizing}); better, for all $R$ the radius-$R$ balls around
  $x_0$ and $x_i$ in $X$ are isomorphic as labelled graphs for all $i$
  large enough, because the action on a sequence $\in\{\9,\8\}^\infty$
  of an element of $G_{012}$ of length $R$ depends only on the last
  $\lceil\log_2(R)\rceil$ symbols of the sequence.

  We now claim that, for all $R\in\N$, there exists $i$ such that the
  balls of radius $R$ in the Cayley graphs of $(W,S_i)$ and $(W',S')$
  coincide. By Theorem~\ref{thm:grigWgrowth}, there exists a constant
  $K$ such that, for all $R\in\N$, the ball of radius $R$ in the
  Cayley graph of $(W',S')$ has cardinality
  $v_{W',S'}(R)\le\exp(K R^{\log2/\log\eta_+})$. Assuming the claim,
  we get $v_{W,S_i}(R)\le\exp(K R^{\log2/\log\eta_+})$, from which the
  second claim of the theorem follows.

  For every $\epsilon>0$ there exists $R$ such that
  $K R^{\log2/\log\eta_+}<\epsilon R$, so
  $v_{W,S_i}(R)\le\exp(\epsilon R)$ and $\lambda_{W,S_i}\le\epsilon$
  for some $i$. It follows then that $W$ has non-uniform exponential
  growth.

  It remains to prove the claim. Given $R\in\N$, let $i$ be large
  enough so that the distance between $x_0$ and $x_i$ in $X$ is at
  least $2R$. Consider a word $w$ in $S_i$ of length $\le R$, and let
  $w'$ be the corresponding word in $S'$ obtained by replacing
  $s@x_0,t@x_i$ respectively by $s'@x_0,t'@x_0$. We show that $w$
  represents the identity in $W$ if and only if $w'$ represents the
  identity in $W'$.

  Write $w=(c,g)$ in $W$, with $c\colon X\to G$ and $g\in G_{012}$.
  Similarly, write $w'=(c',g)$ in $W'$, with $c'\colon X\to A$ and the
  same $g\in G_{012}$. Note that the support of $c$ is contained in
  the union of the balls of radius $R$ around $x_0$ and $x_i$, and
  these balls are isomorphic and disjoint. Therefore, $c$ can be
  written in the form $c=c_1c_2$ with $c_1\colon X\to\langle s\rangle$
  and $c_2\colon X\to\langle t\rangle$, and $c_1,c_2$ have disjoint
  support so they commute. The function $c'$ may correspondingly be
  written as $c'=c'_1c'_2$ with $c'_1\colon X\to\langle s'\rangle$
  obtained by composing $c_1$ with the isomorphism
  $\langle s\rangle\to\langle s'\rangle$, and
  $c'_2\colon X\to\langle t'\rangle$ obtained by composing the
  isomorphism from the radius-$R$ ball around $x_0$ to the radius-$R$
  ball around $x_i$, the map $c_2$, and the isomorphism
  $\langle t\rangle\to\langle t'\rangle$. Therefore, $c'=1$ if and
  only if $c=1$, so the balls in $W$ and $W'$ are isomorphic.
\end{proof}




\end{document}